\documentclass[english]{amsart}

\usepackage{amsfonts,amssymb,amsmath,amsthm}
\usepackage{url}
\usepackage{enumerate}
\usepackage[english]{babel}
\usepackage{latexsym}  
\usepackage{longtable} 
\usepackage{array} 
\usepackage{setspace}
\usepackage{fancyhdr}
\usepackage{epsfig}
\usepackage{lscape}
\usepackage{pdflscape}
\usepackage{datetime}
\usepackage{ifpdf}
\usepackage{textcomp}
\usepackage{setspace}
\usepackage{eufrak}
\usepackage{euscript}

\usepackage{pslatex}          
\usepackage[pdftex]{hyperref} 
\usepackage{upgreek}

\def\includegraphics{}

\newtheorem{theorem}{Theorem}[section]
\newtheorem{lemma}[theorem]{Lemma}
\newtheorem{proposition}[theorem]{Proposition}
\newtheorem{corollary}[theorem]{Corollary}

\theoremstyle{definition}
\newtheorem{definition}[theorem]{Definition}
\newtheorem{remark}[theorem]{Remark}

\numberwithin{equation}{section}

\newcommand{\legendre}[2]{\genfrac{(}{)}{}{}{#1}{#2}}

\hypersetup{
	bookmarks=true,         
	unicode=true,          
	pdftoolbar=true,        
	pdfmenubar=true,        
	pdffitwindow=false,     
	pdfstartview={FitH},    
	pdftitle={Evaluation of Convolution Sums entailing mixed Divisor Functions for a Class of Levels},    
	pdfauthor={Eb{\'{e}}n{\'{e}}zer Ntienjem},     
	pdfsubject={2010 Mathematics Subject Classification. 11A25, 11E20, 11E25,
		11F11, 11F20, 11F27},   
	pdfkeywords={Sums of Divisors function; Dedekind eta function; Convolution
		Sums; Modular Forms; Eisenstein series; Eisenstein forms; Cusp Forms;
		Octonary quadratic Forms; Number of Representations}, 
	pdfnewwindow=true,      
	colorlinks=false,       
	linkcolor=red,          
	citecolor=green,        
	filecolor=magenta,      
	breaklinks=rue,	
	urlcolor=cyan           
}

\title[Convolution Sums entailing mixed Divisor Functions for a Class of Levels]
{Evaluation of Convolution Sums entailing mixed Divisor Functions for a Class of Levels
}

\author{Eb\'{e}n\'{e}zer Ntienjem}
\address{
Centre for Research in Algebra and Number Theory \\
School of Mathematics and Statistics\\
Carleton University\\
1125 Colonel By Drive\\
Ottawa, Ontario, K1S 5B6, Canada}
\email{ebenezer.ntienjem@carleton.ca;ntienjem@gmail.com}



\def \eN{\mbox{E.~Ntienjem}}
\def \waS{\mbox{W.~A.~Stein}}
\def \mN{\mbox{M.~Newman}}
\def \gL{\mbox{G.~Ligozat}}
\def \gK{\mbox{G.~K\"{o}hler}}
\def \ljpK{\mbox{L.~J.~P.~Kilford}}
\def \aA{\mbox{A.~Alaca}}
\def \sA{\mbox{\c{S}.~Alaca}}
\def \aP{\mbox{A.~Pizer}}
\def \ksW{\mbox{K.~S.~Williams}}
\def \jgH{\mbox{J.~G.~Huard}}
\def \gaL{\mbox{G.~A.~Lomadze}}

\def \sR{\mbox{S.~Ramanujan}}

\def \zmO{\mbox{Z.~M.~Ou}}
\def \bkS{\mbox{B.~K.~Spearman}}

\def \exwX{\mbox{E.~X.~W.~Xia}}

\def \oxmY{\mbox{O.~X.~M.~Yao}}
\def \nK{\mbox{N.~Koblitz}}
\def \tM{\mbox{T.~Miyake}}
\def \zSA{\mbox{Z.~Selcuk Aygin}}
\def \dbL{\mbox{D.~B.~Lahiri}}
\def \fU{\mbox{F.~Uygul}}
\def \hE{\mbox{H.~Eser}}
\def \bK{\mbox{B.~K\"{o}kl\"{u}ce}}
\def \nC{\mbox{N.~Cheng}}

\def \N{\mbox{$\EuFrak{N}$}}

\def \E{\mbox{$E$}}
\def \S{\mbox{$S$}}
\def \M{\mbox{$M$}}

\begin{document}


\begin{abstract}
Let $0< n,\alpha,\beta\in\mathbb{N}$ be such that $\gcd{(\alpha,\beta)}=1$. 
We carry out the evaluation of the convolution sums $\underset{\substack{
 {(k,l)\in\mathbb{N}^{2}} \\ {\alpha\,k+\beta\,l=n}
} }{\sum}\sigma(k)\sigma_{3}(l)$ and 
$\underset{\substack{
		{(k,l)\in\mathbb{N}^{2}} \\ {\alpha\,k+\beta\,l=n}
} }{\sum}\sigma_{3}(k)\sigma(l)$ 
for all levels $\alpha\beta\in\mathbb{N}$, by using in particular modular forms. 
We next apply convolution sums belonging to this class of levels 
to determine formulae for the number of representations of a positive 
integer $n$ by the quadratic forms in twelve variables  
$\underset{i=1}{\overset{12}{\sum}}x_{i}^{2}$ when the level 
$\alpha\beta\equiv 0\pmod{4}$, and 
$\underset{i=1}{\overset{6}{\sum}}\,(\,x_{2i-1}^{2}+ x_{2i-1}x_{2i} + x_{2i}^{2}\,)$  
when the level $\alpha\beta\equiv 0\pmod{3}$.  
Our approach is then illustrated by explicitly evaluating the convolution sum for 
$\alpha\beta=3$, $4$, $6$, $7$, $8$, $9$, $12$, $14$, $15$, $16$, $18$, $20$, $21$, $27$, $32$.
These convolution sums are then applied to determine explicit formulae for 
the number of representations of a positive integer $n$ by quadratic forms 
in twelve variables.
\end{abstract}

\maketitle


\section{Introduction} \label{introduction}
We denote in the following by $\mathbb{N}$ the set of natural numbers,  
$\mathbb{N}^{*}$ the set of natural numbers without zero, i.e.,  $\mathbb{N}\setminus\{0\}$, 
$\mathbb{Z}$ the set of integers, 
$\mathbb{Q}$ the set of rational numbers, 
$\mathbb{R}$ the set of real numbers. 
and $\mathbb{C}$ the set of complex numbers.

Suppose that $n\in\mathbb{N}$ and $k\in\mathbb{N}$. Then define the sum $\sigma_{k}(n)$ of the $k^{\text{th}}$ powers of the 
positive divisors of $n$ by  
\begin{equation} \label{def-sigma_k-n}
 \sigma_{k}(n) = \sum_{0<\delta|n}\delta^{k}.
 \end{equation} 
We set $\sigma_{k}(0)=0$ for each $k\in\mathbb{N}$.  
It is obvious from the definition that $\sigma_{k}(m)=0$ for all 
$m\notin\mathbb{N}$.  
We  write $d(n)$ and $\sigma(n)$ as a 
shorthand for $\sigma_{0}(n)$ and $\sigma_{1}(n)$, respectively. 

Assume that the positive integers $\alpha\leq\beta$ are given and that 
$i,j\in\mathbb{N}^{*}$. Then we define 
the convolution sums as follows 
	\begin{equation} \label{def-convolution_sum}
	W_{(\alpha, \beta)}^{2i-1,2j-1}(n) =  \sum_{\substack{
			{(k,l)\in\mathbb{N}^{2}} \\ {\alpha\,k+\beta\,l=n}
	} }\sigma_{2i-1}(k)\sigma_{2j-1}(l).
	\end{equation}
We set $W_{(\alpha,\beta)}^{2i-1,2j-1}(n)=0$ if for all 
$(k,l)\in\mathbb{N}^{2}$ it holds that $\alpha\,k+\beta\,l\neq n$.

 
In this paper, we evaluate the convolution sums 
\begin{alignat*}{5}
W_{(\alpha, \beta)}^{2i-1,2j-1}(n),& \quad 
W_{(\alpha, \beta)}^{2j-1,2i-1}(n), & \quad 
W_{(\beta, \alpha)}^{2i-1,2j-1}(n) & \quad 
\text{and} & \quad 
W_{(\beta, \alpha)}^{2j-1,2i-1}(n)
\end{alignat*}
 for the level $\alpha\beta\in\mathbb{N}$ whenever $i,j\in\mathbb{N}^{*}$ are such that $i+j=3$. 
We let  
\begin{equation*}  \label{introd-eqn-N}
\N=\{\,2^{\nu}\mho\,|\,\nu\in\{0,1,2,3\} \text{ and $\mho$ is a   
	finite product of distinct odd primes}\,\}. 
\end{equation*}    
Then these convolution sums 
are evaluated for the levels $\alpha\beta\in\N$ and $\alpha\beta\in\mathbb{N}^{*}\setminus\N$.

So far known convolution sums whose evaluation involved mixed divisor functions 
$\sigma_{2i-1}(n)$ and $\sigma_{2j-1}(n)$, where $i,j\in\mathbb{N}^{*}$ are 
such that $i+j=3$, are displayed in the following table. 
\begin{longtable}{|r|r|r|r|} \hline 
	\textbf{Level $\alpha\beta$}  &  $\mathbf{i+j}$  &  \textbf{Authors}   &  \textbf{References}  \\ \hline
	1  & $1+2, 2+1$ & \dbL,\ \sR\  & \cite{lahiri1946,lahiri1947,ramanujan} \\ \hline
	~  & ~ & \sA\ \& \fU\ \& \ksW, &  ~  \\
~  & ~ & \jgH\ \& \zmO\ \&   &  \cite{alaca-uygul-williams-2012,cheng-williams-2005,huardetal}   \\
	2  & $1+2$ & \bkS\ \& \ksW,   &  ~ \\ 
	~  & ~ &  \nC\ \& \ksW\   & ~ \\ \hline
	~ & ~ & \nC\ \& \ksW,   & ~ \\
	2  & $2+1$ &  \jgH\ \& \zmO\ \&  & \cite{cheng-williams-2005,huardetal} \\ 
 ~  & ~ & \bkS\ \& \ksW  & ~ \\ \hline
	3  & $1+2$ & \oxmY\ \&  \exwX  & \cite{yao-xia-2014} \\ \hline
	4  & $1+2,2+1$ & \nC\ \& \ksW\   & \cite{cheng-williams-2005} \\ \hline
	6  & $1+2$ & \bK\   & \cite{kokluce-2017} \\ \hline
	12 &  $1+2$ & \bK\ \& \hE\   & \cite{kokluce-eser-2017} \\ \hline
	
	\caption{Known convolution sums $W_{(\alpha, \beta)}^{2i-1,2j-1}(n)$  
		of level $\alpha\beta$} 
	\label{introduction-table-1}
\end{longtable}

The evaluation of these convolution sums, especially for a class of levels is new. 
  
We then apply the result for this class of levels to determine the convolution sums for
\begin{itemize}
 \item $\alpha\beta=7$, $8$, $9$, $14$, $15$, $16$, $18$, $20$, $21$, $27$, $32$ when $(i,j)=(1,2)$; and 
 \item $\alpha\beta=3$, $4$, $6$, $7$, $8$, $9$, $12$, $14$, $15$, $16$, $18$, $20$, $21$, $27$, $32$ when $(i,j)=(2,1)$.
\end{itemize}  
Again, these explicit convolution sums have not been evaluated as yet. 

These convolution sums are applied to establish explicit formulae for the 
number of representations of a positive integer $n$ by the quadratic forms 
in twelve variables 
\begin{equation} \label{introduction-eq-1}
a\,\sum_{i=1}^{4}\,x_{i}^{2} + b\,\sum_{i=5}^{12}\,x_{i}^{2} ,
\end{equation}
 
\begin{equation} \label{introduction-eq-2}
	a_{1}\,\sum_{i=1}^{8}\,x_{i}^{2} + b_{1}\,\sum_{i=9}^{12}\,x_{i}^{2},
\end{equation}

\begin{multline} \label{introduction-eq-3}
c\,\sum_{j=1}^{2}\,(\,x_{2j-1}^{2} + x_{2j-1}x_{2j} + x_{2j}^{2}\,) + 
d\,\sum_{j=3}^{6}\,(\,x_{2j-1}^{2} + x_{2j-1}x_{2j} + x_{2j}^{2}\,)
\end{multline}
and 
\begin{multline} \label{introduction-eq-4}
	c_{1}\,\sum_{j=1}^{4}\,(\,x_{2j-1}^{2} + x_{2j-1}x_{2j} + x_{2j}^{2}\,) + d_{1}\,\sum_{j=5}^{6}\,(\,x_{2j-1}^{2} + x_{2j-1}x_{2j} + x_{2j}^{2}\,)
\end{multline}
where $(a,b),(a_{1},b_{1}),(c,d),(c_{1},d_{1})\in \mathbb{N}^{*}\times\mathbb{N}^{*}$. 

Known formulae for the number of representations of a positive integer $n$ 
by the quadratic forms 
\begin{itemize}
	\item (\hyperref[introduction-eq-1]{\ref*{introduction-eq-1}}) and 
	(\hyperref[introduction-eq-2]{\ref*{introduction-eq-2}}) 
	are given by \bK\ \& \hE\ \cite{kokluce-eser-2017} and by 
	\aA\ \& \sA\ \& \zSA\ \cite{2016arXiv160309412A} when $(a,b)=(a_{1},b_{1})=(1,3)$.
	\item (\hyperref[introduction-eq-3]{\ref*{introduction-eq-3}}) 
are given by \oxmY\ \& \exwX\ \cite{yao-xia-2014} when $(c,d)=(1,1)$.
\end{itemize}
 
Based on this structure of the level $\alpha\beta$,
we provide a method to determine all pairs $(a,b)\in\mathbb{N}^{*}\times\mathbb{N}^{*}$ 
that are necessary for the determination 
of the formulae for the number of representations of a positive integer 
by the quadratic forms (\hyperref[introduction-eq-1]{\ref*{introduction-eq-1}}). 
Then we determine explicit formulae for the number of representations 
of a positive integer $n$ by the quadratic forms 
(\hyperref[introduction-eq-1]{\ref*{introduction-eq-1}})--
(\hyperref[introduction-eq-4]{\ref*{introduction-eq-4}})
whenever $\alpha\beta$ has the above form and is such that 
$\alpha\beta\equiv 0\pmod{4}$ or $\alpha\beta\equiv 0\pmod{3}$. 

We have obtained the results displayed in this paper by using Software for symbolic scientific computation 
which is composed of the open source software packages 
\emph{GiNaC}, \emph{Maxima}, \emph{REDUCE}, \emph{SAGE} and the commercial 
software package \emph{MAPLE}.


\section{Essential Background Knowledge} \label{modularForms}

\subsection{Modular Forms} \label{modForms}

Let $\mathbb{H}=\{ z\in \mathbb{C}~ | ~\text{Im}(z)>0\}$, 
be the upper half-plane 
and let  $\Gamma=G=\text{SL}_{2}(\mathbb{R})= \{\,\left(\begin{smallmatrix} a
    & b \\ c & d \end{smallmatrix}\right)\,\mid\, a,b,c,d\in\mathbb{R} 
\text{ and } ad-bc=1\,\}$  
be the group of $2\times 2$-matrices. Let $N\in\mathbb{N}^{*}$. Then  
\begin{eqnarray*}
\Gamma(N) & = \bigl\{~\left(\begin{smallmatrix} a & b \\ c &
  d \end{smallmatrix}\right)\in\text{SL}_{2}(\mathbb{Z})~ |
  ~\left(\begin{smallmatrix} a & b \\ c &
 d\end{smallmatrix}\right)\equiv\left(\begin{smallmatrix} 1 & 0 \\ 0 & 1 
 \end{smallmatrix}\right) \pmod{N} ~\bigr\}
\end{eqnarray*}
is a subgroup of $G$ and is called the \emph{principal congruence subgroup of 
level N}. A subgroup $H$ of $G$ is called a \emph{congruence subgroup of 
level N} if it contains $\Gamma(N)$.

For our purposes the following congruence subgroup is relevant:
 \begin{align*}
\Gamma_{0}(N) & = \bigl\{~\left(\begin{smallmatrix} a & b \\ c &
  d \end{smallmatrix}\right)\in\text{SL}_{2}(\mathbb{Z})~ | ~
   c\equiv 0 \pmod{N} ~\bigr\}. 
\end{align*}
Let $k,N\in\mathbb{N}$ and let $\Gamma'\subseteq\Gamma$ be a congruence 
subgroup of level $N$. 
Let $k\in\mathbb{Z}, \gamma\in\text{SL}_{2}(\mathbb{Z})$ and $f : 
\mathbb{H}\cup\mathbb{Q}\cup\{\infty\} \rightarrow 
\mathbb{C}\cup\{\infty\}$. 
We denote by 
$f^{[\gamma]_{k}}$ the function whose value at $z$ is $(cz+d)^{-
k}f(\gamma(z))$, i.e., $f^{[\gamma]_{k}}(z)=(cz+d)^{-k}f(\gamma(z))$. 
We use the definition of modular forms according to \nK\ \cite[p.\ 108]{koblitz-1993}.

Let us denote by $\M_{2k}(\Gamma')$ the set of modular forms of weight $2k$ 
for $\Gamma'$,  by $\S_{2k}(\Gamma')$ the set of cusp forms of 
weight $2k$ for $\Gamma'$ and by $\E_{2k}(\Gamma')$ the set of Eisenstein forms. 
The sets $\M_{2k}(\Gamma'),\,\S_{2k}(\Gamma')$ and $\E_{2k}(\Gamma')$ 
are vector spaces over $\mathbb{C}$. 
Therefore, $\M_{2k}(\Gamma_{0}(N))$ is the space of 
modular forms of weight $2k$ for 
$\Gamma_{0}(N)$, $\S_{2k}(\Gamma_{0}(N))$ is the space of 
cusp forms of weight $2k$ for $\Gamma_{0}(N)$, 
and $\E_{2k}(\Gamma_{0}(N))$ is the space of Eisenstein forms. 
The decomposition of the space of modular forms as a direct sum 
of the space generated by the Eisenstein series and the space of cusp 
forms, i.e.,  
$\M_{2k}(\Gamma_{0}(N))=\E_{2k}(\Gamma_{0}(N))\oplus\S_{2k}(\Gamma_{0}(N))$, 
is well-known; see for example 
\waS\ 's book (on line version)
\cite[p.~81]{wstein}. 

We assume in this paper that $k\in\mathbb{N}^{*}$ and that $\chi$ and
$\psi$ are primitive Dirichlet characters with conductors $L$ and $R$,
respectively. \waS\ \cite[p.~86]{wstein} has noted that  
\begin{equation} \label{Eisenstein-gen}
E_{2k,\chi,\psi}(q) = C_{0} + \underset{n=1}{\overset{\infty}{\sum}}\,\biggl(
\underset{d|n}{\sum}\,\psi(d)\chi(\frac{n}{d})\,d^{2k-1}\,\biggr)q^{n}, 
\end{equation}
where 
\begin{equation*}
C_{0} = \begin{cases}
          0 & \text{ if } L >1 \\
          -\frac{B_{2k,\chi}}{4k} & \text{ if } L=1
       \end{cases}
\end{equation*}
and $B_{2k,\chi}$ are the generalized Bernoulli numbers. 
Theorems 5.8 and 5.9 in Section 5.3 of \waS\ \cite[p.~86]{wstein} are 
then applicable.

If the primitive Dirichlet characters $\chi$ and $\psi$ are 
trivial, then their conductors $L$ and $R$ are one, respectively. Therefore, 
(\hyperref[Eisenstein-gen]{\ref*{Eisenstein-gen}}) 
can be normalized and then given as follows: 
$E_{2k}(q) = 1 - \frac{4k}{B_{2k}}\,\underset{n=1}{\overset{\infty}{\sum}}\,
\sigma_{2k-1}(n)\,q^{n}$. This will be the case whenever the level 
$\alpha\beta$ belongs to $\N$.

Let $k,m,n\in\mathbb{N}^{*}$ be such that $m$ is a positive divisor of $n$. 
We then use \tM\ \cite[Lemma 2.1.3, p.\ 41]{miyake1989} to conclude that 
\begin{equation}  \label{basis-cusp-eqn}
\M_{2k}(\Gamma_{0}(m))\subset\M_{2k}(\Gamma_{0}(n)). 
\end{equation} 
It also follows that the same inclusion relation holds for the bases, 
the space of Eisenstein forms of weight $2k$ and the spaces of cusp 
forms of weight $2k$.

\subsection{Eta Quotients}  \label{etaFunctions}

On the upper half-plane $\mathbb{H}$, the Dedekind $\eta$-function, $\eta(z)$,  
is defined by
$\eta(z) = e^{\frac{2\pi i z}{24}}\overset{\infty}{\underset{n=1}{\prod}}(1 - e^{2\pi i n z})$.
Let us set $q=e^{2\pi i z}$. Then it follows that 
\begin{equation*}
\eta(z) = q^{\frac{1}{24}}\overset{\infty}{\underset{n=1}{\prod}}(1 - q^{n}) = q^{\frac{1}{24}} F(q),\quad 
\text{where } F(q)=\overset{\infty}{\underset{n=1}{\prod}} (1 - q^{n}).
\end{equation*}

Let $j,N\in\mathbb{N}^{*}$ and $e_{j}\in\mathbb{Z}$. Following  
\gK \cite[p.\ 31]{koehler} an \emph{eta product} or \emph{eta quotient}, $f(z)$,  
is a finite product of eta functions, which is of the form
$\underset{j | N}{\overset{}{\prod}}\eta(jz)^{e_{j}}$,
wherein $N$ is the \emph{level} of the eta product.
If $2k=\frac{1}{2}\underset{j=1}{\overset{\kappa}{\sum}}e_{j}$, 
then an eta quotient $f(z)$ behaves like a modular form of weight 
$2k$ on $\Gamma_{0}(N)$ with some multiplier system.

Note that eta function, eta quotient and eta product will use 
interchangeably as synonyms.

\ljpK\ \cite[p.~99]{kilford}  and \gK\ \cite[Cor.\ 2.3, p.~37]{koehler} have 
given a formulation of the following theorem which is a result of the work
of \mN\ \cite{newman_1957,newman_1959} and \gL\ \cite{ligozat_1975}. 
This theorem will be effectively used 
to exhaustively determine $\eta$-quotients, $f(z)$, which 
belong to $\M_{2k}(\Gamma_{0}(N))$, and especially those $\eta$-quotients 
which are in $\S_{2k}(\Gamma_{0}(N))$. 

\begin{theorem}[\mN\ and \gL\ ] \label{ligozat_theorem} 
Let $N\in \mathbb{N}^{*}$, $D(N)$ be the set of all positive divisors of 
$N$, $\delta\in D(N)$ and $r_{\delta}\in\mathbb{Z}$. Let furthermore  
$f(z)=\overset{}{\underset{\delta\in D(N)}{\prod}}\eta^{r_{\delta}}(\delta z)$ 
be an eta quotient. If the following four conditions are satisfied

\begin{alignat*}{2}
\text{\textbf{(i)}} \> \overset{}{\underset{\delta\in D(N)}{\sum}}\delta\,r_{\delta} 
	\,\equiv 0 \pmod{24}, \qquad
\text{\textbf{(ii)}} \>  \overset{}{\underset{\delta\in D(N)}{\sum}}\frac{N}{\delta} 
	\,r_{\delta}\, \equiv 0 \pmod{24}, \\
\text{\textbf{(iii)}} \>  \overset{}{\underset{\delta\in D(N)}
{\prod}}\delta^{r_{\delta}}\> \text{ is a square in } \mathbb{Q}, \qquad
\text{\textbf{(iv)}} \>\> \forall d\in D(N)\> \text{ it holds } 
	\overset{}{\underset{\delta\in D(N)}{\sum}}\frac{\gcd{(\delta,d)}^{2}}	
	{\delta}\,r_{\delta} \geq 0,  
\end{alignat*}

then $f(z)\in\M_{2k}(\Gamma_{0}(N))$, where 
$2k=\frac{1}{2}\overset{}{\underset{\delta\in D(N)}{\sum}}r_{\delta}$. 

Moreover, the $\eta$-quotient $f(z)$ is an element of $\S_{2k}(\Gamma_{0}(N))$ 
if {\textbf{(iv)}} is replaced by
 
{\textbf{(iv')}} $\forall d\in D(N)$ \text{ it holds } 
$\overset{}{\underset{\delta\in D(N)}{\sum}}\frac{\gcd{(\delta,d)}^{2}}	
{\delta} r_{\delta} > 0$.
\end{theorem}

\subsection{Convolution Sums $\mathbf{W_{(\boldsymbol{\upalpha, \upbeta})}^{2i-1,2j-1}(n)}$ }
\label{convolutionSumsEqns}

Given $i,j,\alpha,\beta\in\mathbb{N}^{*}$ such that $i+j=3$, and  
$\alpha\leq\beta$, let the convolution sums be defined  by 
(\hyperref[def-convolution_sum]{\ref*{def-convolution_sum}}).
Suppose that $\gcd{(\alpha,\beta)}= \delta >1$ for some 
$\delta\in\mathbb{N}^{*}$. 
Then there exist 
$\alpha_{1}, \beta_{1}\in\mathbb{N}^{*}$ such that 
$\gcd{(\alpha_{1},\beta_{1})}=1,\ \alpha=\delta\,\alpha_{1}$ and 
$\beta =\delta\,\beta_{1}$. Hence,  
\begin{equation}  	\label{convolutionSumsEqns-gcd}
\begin{split}
	W_{(\alpha,\beta)}^{2i-1,2j-1}(n) = & \sum_{\substack{
			{(k,l)\in\mathbb{N}^{2}} \\ {\alpha\,k+\beta\,l=n}
		}} \sigma_{2i-1}(k)\sigma_{2j-1}(l)   \\ 
		=  &\sum_{\substack{
				{(k,l)\in\mathbb{N}^{2}} \\ {\delta\,\alpha_{1}\,k+\delta\,\beta_{1}\,l=n}
			}}\sigma_{2i-1}(k)\sigma_{2j-1}(l) \\ 
			= & \, W_{(\alpha_{1},\beta_{1})}^{2i-1,2j-1}(\frac{n}{\delta}). 
\end{split}
		\end{equation}
Therefore, we may simply assume that $\gcd{(\alpha,\beta)}=1$. 
As an immediate consequence of the definition, we note that
\begin{itemize}
	\item $W_{(\alpha,\beta)}^{2i-1,2j-1}(n)=W_{(\beta,\alpha)}^{2i-1,2j-1}(n)$ whenever $i=j$ holds; and 
\item $W_{(\alpha,\beta)}^{2i-1,2j-1}(n)=W_{(\beta,\alpha)}^{2j-1,2i-1}(n)$ whenever $\alpha=\beta$ holds. By 
(\hyperref[convolutionSumsEqns-gcd]{\ref*{convolutionSumsEqns-gcd}}) it is sufficient to consider only the case $\alpha=\beta=1$. 
We then apply the commutativity of the multiplication, 
namely, $\sigma(k)\sigma_{3}(n-k)=\sigma_{3}(n-k)\sigma(k)$, to yield 
\begin{equation} \label{convolsum1_3=convolsum3_1}
W_{(1, 1)}^{1,3}(n) =  \sum_{\substack{
		{(k,l)\in\mathbb{N}^{2}} \\ {k+l=n}
} }\sigma(k)\sigma_{3}(l) = \underset{k=0}{\overset{n}{\sum}
}\sigma(k)\sigma_{3}(n-k) = 
\sum_{\substack{
		{(k,l)\in\mathbb{N}^{2}} \\ {k+l=n}
} }\sigma_{3}(k)\sigma(l) = W_{(1, 1)}^{3,1}(n).
\end{equation}
\end{itemize} 

We assume that the primitive Dirichlet characters $\chi$ and $\psi$ 
\begin{enumerate}
	\item are trivial whenever $\alpha\beta\in\N$ holds. 
	\item are such that $\chi=\psi$ and that $\chi$ is a 
	Legendre-Jacobi-Kronecker symbol otherwise. 
\end{enumerate}
Then the following Eisenstein series hold: 
\begin{align}  
E_{2}(q) = 1-24\,\sum_{n=1}^{\infty}\sigma(n)q^{n}, 
\label{evalConvolClass-eqn-3} \\
E_{4}(q) = 1+240\,\sum_{n=1}^{\infty}\sigma_{3}(n)q^{n}, 
\label{evalConvolClass-eqn-4a} \\
E_{6}(q) = 1-504\,\sum_{n=1}^{\infty}\sigma_{5}(n)q^{n}, 
\label{evalConvolClass-eqn-6} 
\end{align}
\begin{align}  \label{evalConvolClass-eqn-4}
\begin{split}
E_{2k,\chi}(q^{\lambda}) & = E_{2k}(q^{\lambda})\otimes\chi  \\
 & = \chi(\lambda)\,\biggl( C_{0} + \sum_{n=1}^{\infty}\,\chi(n)\,\sigma_{2k-1}(n)\,q^{\lambda\,n} \,\biggr) \\
 & = \chi(\lambda)\,C_{0} + 
\sum_{n=1}^{\infty}\,\chi(\lambda\,n)\,\sigma_{2k-1}(n)\,q^{\lambda\,n}, \\
 & = \chi(\lambda)\,C_{0} + 
\sum_{n=1}^{\infty}\,\chi(n)\,\sigma_{2k-1}(\frac{n}{\lambda})\,q^{n}, 
\end{split}
\end{align}
where $\lambda\in\mathbb{N}^{*}$, 
\begin{equation*}
	C_{0} = \begin{cases}
		0 & \text{ if } L >1 \\
		-\frac{B_{2k,\chi}}{4k} & \text{ if } L=1
	\end{cases}
\end{equation*}
and $B_{2k,\chi}$ are the specially generalized Bernoulli numbers. 

Note that $E_{4}(q)$ and $E_{6}(q)$ are special cases of 
(\hyperref[Eisenstein-gen]{\ref*{Eisenstein-gen}}) or
(\hyperref[evalConvolClass-eqn-4]{\ref*{evalConvolClass-eqn-4}}) and hold if 
$\alpha\beta\in\N$. 
We state two relevant results for the sequel of this work.

\begin{lemma}  \label{evalConvolClass-lema-1}
Let $\alpha,\beta\in\mathbb{N}^{*}$. Then 
\begin{equation*}
(\,\alpha\, E_{2}(q^{\alpha}) - E_{2}(q)\,)\,E_{4}(q^{\beta})\in
\M_{6}(\Gamma_{0}(\alpha\beta))
\end{equation*}
and 
\begin{equation*}
 E_{4}(q^{\alpha})\,(\, E_{2}(q) - \beta\,E_{2}(q^{\beta})\,)\in
\M_{6}(\Gamma_{0}(\alpha\beta)).
\end{equation*}
\end{lemma}
\begin{proof} We will only prove the first part since the second 
part can be shown similarly. 

If $\alpha>0$, then trivially 
$0=(\,\alpha\, E_{k}(q^{\alpha}) - \alpha\,E_{k}(q^{\alpha})\, )\in \M_{k}
(\Gamma_{0}(\alpha))$ and there is nothing to prove. 
Therefore, we may suppose that $\alpha\neq 1$ in the sequel. 
We apply the result proved by W.~A.~Stein \cite[Thrms 5.8, 5.9, p.~86]{wstein} 
to deduce that   
$(\,\alpha\, E_{2}(q^{\alpha}) - E_{2}(q)\,)\in \M_{2}(\Gamma_{0}(\alpha))\subseteq 
\M_{2}(\Gamma_{0}(\alpha\beta))\subseteq 
\M_{6}(\Gamma_{0}(\alpha\beta))$ and  
$ E_{4}(q^{\beta}) \in \M_{4}(\Gamma_{0}(\beta))\subseteq
\M_{4}(\Gamma_{0}(\alpha\beta))$. Therefore, we obtain 
$(\,\alpha\, E_{2}(q^{\alpha}) - E_{2}(q)\,)\,E_{4}(q^{\beta}) \in
\M_{6}(\Gamma_{0}(\alpha\beta))$.
\end{proof}
The following identity is proven for $\alpha =1$ by \dbL \cite[p.\,149]{lahiri1946}. 
	For all $\alpha\in\mathbb{N}^{*}$, it holds that
	\begin{align} \label{Lahiri-ident-01}
	E_{2}(q^{\alpha})\,E_{4}(q^{\alpha}) & =  1 + 720\,\sum_{n\geq 1} \frac{n}{\alpha}\,\sigma_{3}(\frac{n}{\alpha})\, q^{n} - 504\,\sum_{n\geq 1} \sigma_{5}(\frac{n}{\alpha})\, q^{n}. 
	\end{align}
The following identity is shown by \sR \cite[Table IV, p.\,168]{ramanujan}. 
	For all $n\in\mathbb{N}^{*}$, it holds that
	\begin{align} \label{Ramanujan-ident-1-3}
	W_{(1,1)}^{1,3}(n) &  = \frac{21}{240}\,\sigma_{5}(n) + \frac{1}{24}\,(1 - 3n)\,\sigma_{3}(n) - \frac{1}{240}\,\sigma(n). 
	\end{align}
We then apply (\hyperref[convolutionSumsEqns-gcd]{\ref*{convolutionSumsEqns-gcd}}) 
and  (\hyperref[convolsum1_3=convolsum3_1]{\ref*{convolsum1_3=convolsum3_1}})
to deduce that  
	for all $\alpha,n\in\mathbb{N}^{*}$ 
	\begin{align} \label{Ramanujan-ident-1-3-gen}
	W_{(\alpha,\alpha)}^{3,1}(n) = W_{(\alpha,\alpha)}^{1,3}(n)  &  = \frac{21}{240}\,\sigma_{5}(\frac{n}{\alpha}) + \frac{1}{24}\,(1 - 3\,\frac{n}{\alpha})\,\sigma_{3}(\frac{n}{\alpha}) - \frac{1}{240}\,\sigma(\frac{n}{\alpha}). 
	\end{align}

When we make use of the just proven identity,  
we obviously infer that for all $\alpha\in\mathbb{N}^{*}$ 
\begin{align} \label{ntienjem-ident-01}
E_{4}(q^{\alpha})\,E_{2}(q^{\alpha}) = E_{2}(q^{\alpha})\,E_{4}(q^{\alpha}) & =  1 + 720\,\sum_{n\geq 1} \frac{n}{\alpha}\,\sigma_{3}(\frac{n}{\alpha})\, q^{n} - 504\,\sum_{n\geq 1} \sigma_{5}(\frac{n}{\alpha})\, q^{n}. 
\end{align}

\begin{theorem} \label{convolutionSum_a_b}
Let $\alpha,\beta\in\mathbb{N}^{*}$ be such that 
$\gcd(\alpha,\beta)=1$. Then
\begin{multline}
 (\,\alpha\, E_{2}(q^{\alpha}) -  E_{2}(q)\,)\,E_{4}(q^{\beta})  = 
 \alpha - 1 
    + \sum_{n=1}^{\infty}\biggl(\, 24\,\sigma(n) - 24\,\alpha\,\sigma(\frac{n}{\alpha}) \\ 
    + 240\,\alpha\,\sigma_{3}(\frac{n}{\beta}) 
    - 240\,\sigma_{3}(\frac{n}{\beta})
    + 24\times 240\,W_{(1,\beta)}^{1,3}(n) 
    -  24\times 240\,\alpha\, W_{(\alpha,\beta)}^{1,3}(n)\,\biggr)q^{n} 
    \label{evalConvolClass-eqn-11}
\end{multline}
and 
\begin{multline}
E_{4}(q^{\alpha})\,(\, E_{2}(q) - \beta\,E_{2}(q^{\beta})\,) = 
 1 - \beta  
+ \sum_{n=1}^{\infty}\biggl(\, 24\,\beta\,\sigma(\frac{n}{\beta}) -  24\,\sigma(n)   \\ 
+  240\,\sigma_{3}(\frac{n}{\alpha})  
- 240\,\beta\,\sigma_{3}(\frac{n}{\alpha}) - 24\times 240\, W_{(\alpha,1)}^{3,1}(n)
+ 24\times 240\,\beta\, W_{(\alpha,\beta)}^{3,1}(n)\,\biggr)q^{n}. 
\label{evalConvolClass-eqn-11-2}
\end{multline}
\end{theorem}
\begin{proof} We observe that 
\begin{multline}
(\,\alpha\,E_{2}(q^{\alpha}) -  E_{2}(q)\,)\,E_{4}(q^{\beta}) = 
 \alpha\,E_{2}(q^{\alpha})\,E_{4}(q^{\beta}) - E_{2}(q)\,E_{4}(q^{\beta}) 
\label{evalConvolClass-eqn-7}
\end{multline}
and 
\begin{multline}
E_{4}(q^{\alpha})\,(\,E_{2}(q) - \beta\,E_{2}(q^{\beta})\,) = 
 E_{4}(q^{\alpha})\,E_{2}(q) - \beta\,E_{4}(q^{\alpha})\,E_{2}(q^{\beta}) 
\label{evalConvolClass-eqn-7a}
\end{multline}
For the sake of simplicity, we only prove (\hyperref[evalConvolClass-eqn-11-2]{\ref*{evalConvolClass-eqn-11-2}}) 
since the other can be proven similarly.    
\begin{equation}
 E_{4}(q^{\alpha})\,E_{2}(q) = 1-24\,\sum_{n=1}^{\infty}\sigma(n)\,q^{n}  
+ 240\,\sum_{n=1}^{\infty}\sigma_{3}(\frac{n}{\alpha})\,q^{n}
- 24\times 240\,\sum_{n=1}^{\infty}\,W_{(\alpha,1)}^{3,1}(n)\,q^{n}
\label{evalConvolClass-eqn-5}
\end{equation}
and 
\begin{multline}
\beta\,E_{4}(q^{\alpha})\,E_{2}(q^{\beta})  = \beta  
 - 24\,\beta\,\sum_{n=1}^{\infty}\sigma(\frac{n}{\beta})\,q^{n}  \\
 + 240\,\beta\,\sum_{n=1}^{\infty}\sigma_{3}(\frac{n}{\alpha})\,q^{n}
 - 24\times 240\,\beta\,\sum_{n=1}^{\infty}\,W_{(\alpha,\beta)}^{3,1}(n)\,q^{n}.
  \label{evalConvolClass-eqn-8}
\end{multline}
We then put (\hyperref[evalConvolClass-eqn-5]{\ref*{evalConvolClass-eqn-5}}) 
and (\hyperref[evalConvolClass-eqn-8]{\ref*{evalConvolClass-eqn-8}})  
together to obtain
\begin{multline*}
E_{4}(q^{\alpha})\,(\, E_{2}(q) - \beta\,E_{2}(q^{\beta})\,)  = 
 1 - \beta  
+ \sum_{n=1}^{\infty}\biggl(\, 24\,\beta\,\sigma(\frac{n}{\beta}) -  24\,\sigma(n)   \\ 
+  240\,\sigma_{3}(\frac{n}{\alpha})  
- 240\,\beta\,\sigma_{3}(\frac{n}{\alpha}) - 24\times 240\, W_{(\alpha,1)}^{3,1}(n)
+ 24\times 240\,\beta\, W_{(\alpha,\beta)}^{3,1}(n)\,\biggr)q^{n} 
\end{multline*}
as asserted. 
\end{proof}
The following corollary discusses special cases of \autoref{convolutionSum_a_b}, 
which are essential for the determination of $W_{(\alpha,\beta)}^{2i-1,2j-1}(n)$.
\begin{corollary} \label{convolutionSum_sp}
Let $\alpha,\beta\in\mathbb{N}^{*}$ be such that 
$\gcd(\alpha,\beta)=1$. 
\begin{description}
	\item[If $\alpha\neq 1$ and $\beta=1$, then] 
	\begin{multline}
	(\,\alpha\, E_{2}(q^{\alpha}) -  E_{2}(q)\,)\,E_{4}(q)  = 
  \alpha - 1 
+ \sum_{n=1}^{\infty}\biggl(\, - 24\,\alpha\,\sigma(\frac{n}{\alpha})
+ 240\,\alpha\,\sigma_{3}(n)  \\ 
- 720\,n\,\sigma_{3}(n) + 504\,\sigma_{5}(n)  
-  24\times 240\,\alpha\, W_{(\alpha,1)}^{1,3}(n)\,\biggr)q^{n}. 
	\label{evalConvolClass-eqn-sp-1}
	\end{multline}
	\item[If $\beta\neq 1$ and $\alpha=1$, then] 
	\begin{multline}
E_{4}(q)\,(\, E_{2}(q) - \beta\,E_{2}(q^{\beta})\,) = 
1 - \beta  
+ \sum_{n=1}^{\infty}\biggl(\, 24\,\beta\,\sigma(\frac{n}{\beta}) -  240\,\beta\,\sigma_{3}(n)   \\ 
+  720\,n\,\sigma_{3}(n)  
- 504\,\sigma_{5}(n)
+ 24\times 240\,\beta\, W_{(1,\beta)}^{3,1}(n)\,\biggr)q^{n}. 
	\label{evalConvolClass-eqn-sp-2}
	\end{multline}
	\item[If $\beta=\alpha\neq 1$, then] 
		\begin{multline}
	(\,\beta\, E_{2}(q^{\beta}) - E_{2}(q)\,)\,E_{4}(q^{\beta})  = 
\beta - 1 
+ \sum_{n=1}^{\infty}\biggl(\, 24\,\sigma(n) - 240\,\sigma_{3}(\frac{n}{\beta}) \\ 
+ 720\,n\,\sigma_{3}(\frac{n}{\beta}) - 504\,\beta\,\sigma_{5}(\frac{n}{\beta})
+ 24\times 240\,W_{(1,\beta)}^{1,3}(n) \,\biggr)q^{n} 
	\label{evalConvolClass-eqn-sp-3}
	\end{multline}
	and
	\begin{multline}
E_{4}(q^{\alpha})\,(\, E_{2}(q) - \alpha\,E_{2}(q^{\alpha})\,) = 
1 - \alpha  
+ \sum_{n=1}^{\infty}\biggl(\, - 24\,\sigma(n) +  240\,\sigma_{3}(\frac{n}{\alpha})   \\ 
-  720\,n\,\sigma_{3}(\frac{n}{\alpha})  
+ 504\,\alpha\,\sigma_{5}(\frac{n}{\alpha})
- 24\times 240\, W_{(\alpha,1)}^{3,1}(n)\,\biggr)q^{n}. 
	\label{evalConvolClass-eqn-sp-4}
	\end{multline}
\end{description}	
\end{corollary} 
\begin{proof} We only prove the following cases as the others can be proved similarly.
	By \autoref{convolutionSum_a_b}  
	\begin{description}   
		\item[if $\alpha\neq 1$ and $\beta=1$]  then 
		\begin{multline*}
		(\,\alpha\, E_{2}(q^{\alpha}) - E_{2}(q)\,)\,E_{4}(q)  = 
  \alpha - 1
+ \sum_{n=1}^{\infty}\biggl(\, - 24\,\alpha\,\sigma(\frac{n}{\alpha})
+ 240\,\alpha\,\sigma_{3}(n)  \\ 
- 720\,n\,\sigma_{3}(n) + 504\,\sigma_{5}(n)  
-  24\times 240\,\alpha\, W_{(\alpha,1)}^{1,3}(n)\,\biggr)q^{n} 
		\end{multline*}
	when we also make use of either 
	(\hyperref[Lahiri-ident-01]{\ref*{Lahiri-ident-01}} or 
	(\hyperref[Ramanujan-ident-1-3-gen]{\ref*{Ramanujan-ident-1-3-gen}}).
	\item[if $\beta=\alpha\neq 1$]  then 
	\begin{multline*}
E_{4}(q^{\alpha})\,(\, E_{2}(q) - \alpha\,E_{2}(q^{\alpha})\,) = 
1 - \alpha  
+ \sum_{n=1}^{\infty}\biggl(\, - 24\,\sigma(n) +  240\,\sigma_{3}(\frac{n}{\alpha})   \\ 
-  720\,n\,\sigma_{3}(\frac{n}{\alpha})  
+ 504\,\alpha\,\sigma_{5}(\frac{n}{\alpha})
- 24\times 240\, W_{(\alpha,1)}^{3,1}(n)\,\biggr)q^{n} 
\end{multline*}
	when we additionally apply  
(\hyperref[Ramanujan-ident-1-3-gen]{\ref*{Ramanujan-ident-1-3-gen}}).
\end{description}
\end{proof}

From (\hyperref[evalConvolClass-eqn-sp-3]{\ref*{evalConvolClass-eqn-sp-1}}
and (\hyperref[evalConvolClass-eqn-sp-4]{\ref*{evalConvolClass-eqn-sp-2}},
and (\hyperref[evalConvolClass-eqn-sp-3]{\ref*{evalConvolClass-eqn-sp-3}}
and (\hyperref[evalConvolClass-eqn-sp-4]{\ref*{evalConvolClass-eqn-sp-4}}  
it immediately follows 
\begin{lemma} \label{convolutionSum-lema-sp}
Let $1<\alpha\in\mathbb{N}^{*}$. Then 
	\begin{align*}
 (\,\alpha\, E_{2}(q^{\alpha}) - E_{2}(q)\,)\,E_{4}(q) = 
-\,E_{4}(q)\,(\, E_{2}(q) - \alpha\,E_{2}(q^{\alpha})\,) \\
\text{and} \\
 (\,\alpha\, E_{2}(q^{\alpha}) - E_{2}(q)\,)\,E_{4}(q^{\alpha}) = 
-\,E_{4}(q^{\alpha})\,(\, E_{2}(q) - \alpha\,E_{2}(q^{\alpha})\,).
\end{align*}
\end{lemma}
	

\section{Evaluating  $W_{(\alpha,\beta)}^{1,3}(n)$ and $W_{(\alpha,\beta)}^{3,1}(n)$, where 
$\alpha\beta\in\mathbb{N}^{*}$}
\label{convolution_alpha_beta}

We carry out an explicit formula for the convolution sums $W_{(\alpha,\beta)}^{1,3}(n)$ and $W_{(\alpha,\beta)}^{3,1}(n)$ wherein  
$\alpha\beta\in\mathbb{N}^{*}$. 

\subsection{Bases of $\E_{6}(\boldsymbol{\Upgamma}_{0}(\boldsymbol{\upalpha\upbeta}))$ and $\S_{6}(\boldsymbol{\Upgamma}_{0}(\boldsymbol{\upalpha\upbeta}))$}   
\label{convolution_alpha_beta-bases}

Let $\EuScript{D}(\alpha\beta)$ denote the set of all positive divisors of 
$\alpha\beta$.

In terms of theta series, \aP\ \cite{pizer1976} has discussed the existence of a basis of the 
space of cusp forms of weight $2k\in\mathbb{N}^{*}$ for 
$\Gamma_{0}(\alpha\beta)$ when $\alpha\beta$ is not a perfect square. 
We suppose in the sequel that the weight $2k\in\mathbb{N}^{*}$ and we then 
apply the dimension formulae in \tM 's book  
\cite[Thrm 2.5.2,~p.~60]{miyake1989} or \waS 's book  
\cite[Prop.\ 6.1, p.\ 91]{wstein} to conclude that 
\begin{itemize} 
 \item  for the space of Eisenstein forms, 
 \begin{equation} \label{dimension-Eisenstein}
  \text{dim}(\E_{2k}(\Gamma_{0}(\alpha\beta)))=\underset{d|\alpha\beta}{\sum}
\,\varphi(\gcd(d,\frac{\alpha\beta}{d}))=m_{E},
\end{equation}
 where $m_{E}\in\mathbb{N}^{*}$	and $\varphi$ is the Euler's totient function. 
 
 We observe that, if $\alpha\beta\in\N$, then
 \begin{equation} \label{dimension-Eisenstein-special}
 \text{dim}(\E_{2k}(\Gamma_{0}(\alpha\beta)))=\underset{\delta|\alpha\beta}{\sum}\varphi
 (\gcd (\delta,\frac{\alpha\beta}{\delta}))=\underset{\delta|\alpha\beta}{\sum}1 = \sigma_{0}(\alpha\beta)=d(\alpha\beta).
 \end{equation}  
\item for the space of cusp forms,   
$\text{dim}(\S_{2k}(\Gamma_{0}(\alpha\beta)))=m_{S}$, where 
$m_{S}\in\mathbb{N}$. 
\end{itemize} 
We use \autoref{ligozat_theorem} $(i)-(iv')$ to exhaustively determine 
as many elements of the space $\S_{2k}(\Gamma_{0}(\alpha\beta))$ 
as possible. From these elements of the space 
$\S_{2k}(\Gamma_{0}(\alpha\beta))$ we select relevant ones for the 
purpose of the determination of a basis of this space. The proof of the 
following theorem provides a method to effectively determine such a basis.  

The so-determined basis of the vector space of cusp forms is in general not unique. 
However, due to the change of basis which is an automorphism, it is sufficient 
to only consider this basis for our purpose. 

Let $\EuScript{C}$ denote the set of primitive Dirichlet characters 
$\chi(n)=\legendre{m}{n}$ as assumed in 
(\hyperref[evalConvolClass-eqn-4]{\ref*{evalConvolClass-eqn-4}}), 
where $m,n\in\mathbb{Z}$ and $\legendre{m}{n}$ is the Legendre-Jacobi-Kronecker symbol. 
Let $D_{\chi}(\alpha\beta)$
denote the 
subset of $\EuScript{D}(\alpha\beta)$ associated with the character $\chi$. 

Let $i,\kappa$ be natural numbers. The expression of a natural number in 
the form  $\underset{i=1}{\overset{\kappa}{\prod}} p_{i}^{e_{i}}$, where 
$p_{i}$ is a prime and $e_{i}$ is in $\mathbb{N}^{*}$, modulo a permutation 
of the primes $p_{i}$, is standard. In the following we will use this form 
to express a level $\alpha\beta\in\mathbb{N}^{*}\setminus\N$.

We observe that the selection of an applicable primitive Dirichlet character 
is not obvious. Let us assume that the levels are $2\cdot 3^{2}$ and $2^{4}$. Then the 
primitive Dirichlet characters $\legendre{-3}{n}$ and $\legendre{-4}{n}$ will 
not be good candidates for the determination of a basis of 
$\E_{2k}(\Gamma_{0}(2\cdot 3^{2}))$ and $\E_{2k}(\Gamma_{0}(2^{4}))$, respectively.

\begin{definition}  \label{def-annihilator}
  Let $i,\kappa\in\mathbb{N}^{*}$ and $n\in\mathbb{N}$. Let 
  $C\in\mathbb{Z}$ be fixed. Suppose that the level
  $\alpha\beta\in\mathbb{N}^{*}\setminus\N$ is fixed and of the form 
	$\underset{i=1}{\overset{\kappa}{\prod}} p_{i}^{e_{i}}$, where $p_{i}$ is 
	a prime number and $e_{i}$ is in $\mathbb{N}^{*}$. 
	We say that the primitive Dirichlet character $\chi(n)=\legendre{C}{n}$ \emph{annihilates} 
	$\E_{6}(\Gamma_{0}(\alpha\beta))$ or is an \emph{annihilator} of 
	$\E_{6}(\Gamma_{0}(\alpha\beta))$ if for some $1\leq j\leq\kappa$ we have $1<p_{j}^{e_{j}}\in\mathbb{N}^{*}\setminus\N$ and  $M_{\chi}(q^{\delta})$ vanishes for all $1<\delta$ positive divisor of $p_{j}^{e_{j}}$. 
	
	A set $\mathcal{C}$ of primitive Dirichlet characters \emph{annihilates} 
	$\E_{6}(\Gamma_{0}(\alpha\beta))$ or is an \emph{annihilator} of $\E_{6}(\Gamma_{0}(\alpha\beta))$ if each 
	$\chi(n)\in\mathcal{C}$ is an annihilator of $\E_{6}(\Gamma_{0}(\alpha\beta))$.
\end{definition}

To illustrate the above definition, suppose that $\alpha\beta=36$ and the primitive Dirichlet characters is $\chi(n)=\legendre{-3}{n}$. Then $C=-3$ so that $|C|$is a positive divisor of $9=3^{2}$.
Since it holds that 
	\begin{equation*} 
\chi(n)=\legendre{-3}{n}= \begin{cases}
	-1 & \text{ if } n\equiv 2\pmod{3}, \\
0 & \text{ if } \text{gcd}(3,n) \neq 1, \\
1 & \text{ if } n\equiv 1\pmod{3} \\ 
\end{cases}
\end{equation*} 
and 
	\begin{equation*} 
\sigma_{3}(\frac{n}{\delta})= \begin{cases}
0 & \text{ if } \frac{n}{\delta}\notin\mathbb{N}_{0}, \\
\text{nonzero} &  \text{gcd}(3,n) \neq 1,  
\end{cases}
\end{equation*}
for all  $1<\delta\in\EuScript{D}(9)$, we apply 
(\hyperref[evalConvolClass-eqn-4]{\ref*{evalConvolClass-eqn-4}}) 
to deduce that 
\[ E_{6,\chi}(q^{\delta}) =  
\sum_{n=1}^{\infty}\,\chi(n)\,\sigma_{5}(\frac{n}{\delta})\,q^{n}= 0. \]
Therefore, the primitive Dirichlet characters  $\chi(n)=\legendre{-3}{n}$ is an annihilator of $\E_{6}(\Gamma_{0}(36))$.

The following theorem provides a strong criteria when selecting a primitive 
Dirichlet character for a given level $\alpha\beta\in\mathbb{N}^{*}\setminus\N$. 
\begin{theorem} \label{theor-annihilator}
Let $i,\kappa$ be in $\mathbb{N}^{*}$. Let $C\in\mathbb{Z}$ be fixed. 
Let $\chi$ be a primitive Dirichlet character with conductor $|C|>1$ and 
let the level $\alpha\beta\in\mathbb{N}^{*}\setminus\N$ be fixed and of the form 
$\underset{i=1}{\overset{\kappa}{\prod}} p_{i}^{e_{i}}$, where $p_{i}$ is 
a prime number and $e_{i}$ is in $\mathbb{N}^{*}$. Suppose 
furthermore that $p_{j}^{e_{j}}\in\mathbb{N}^{*}\setminus\N$ is a positive 
divisor of $\alpha\beta$ for some $1\leq j\leq\kappa$. If the conductor  
$|C|$ is a positive divisor of $p_{j}$ and hence of the level $\alpha\beta$, then 
$\chi(n)=\legendre{C}{n}$, for all $n\in\mathbb{N}$, is an annihilator of 
$\E_{6}(\Gamma_{0}(\alpha\beta))$. 
\end{theorem}
\begin{proof} 
	Suppose that $\alpha\beta\in\mathbb{N}^{*}\setminus\N$ is fixed and 
of the form 
$\underset{i=1}{\overset{\kappa}{\prod}} p_{i}^{e_{i}}$, where $p_{i}$ is 
a prime number and $e_{i}$ is in $\mathbb{N}^{*}$. As an immediate consequence of 
the structure of $\alpha\beta$ there exists $1\leq j\leq\kappa$ such that 
$p_{j}^{e_{j}}\in\mathbb{N}^{*}\setminus\N$ is a positive 
divisor of $\alpha\beta$. Since the conductor $|C|$ 
is a positive divisor of the level $p_{j}$, the existence of 
$1<\delta\in\EuScript{D}(p_{j}^{e_{j}})$ is given. 
It is well-known that for each $1\leq f\leq e_{j}$ it 
holds that $p_{j}^{f}$ is a positive divisor of $p_{j}^{e_{j}}$. 
On the other hand, it holds that 
\begin{equation} 
\legendre{C}{n}= \begin{cases}
0 & \text{ if } \text{gcd}(|C|,n) \neq 1, \\
\text{nonzero} & \text{ otherwise}. 
\end{cases}
\end{equation}
For each $1<\delta\in\EuScript{D}(p_{j}^{e_{j}})$ it holds that 
$\text{gcd}(|C|,\delta) \neq 1$. 
Since the conductor of $\chi$ is greater than one, we have $C_{0}=0$.  Now we apply  
(\hyperref[evalConvolClass-eqn-4]{\ref*{evalConvolClass-eqn-4}}) 
to infer that 
\[ E_{6,\chi}(q^{\delta}) =  
\sum_{n=1}^{\infty}\,\chi(n)\,\sigma_{5}(\frac{n}{\delta})\,q^{n}. \]
Since it also holds that 
\begin{equation} 
\sigma_{5}(\frac{n}{\delta})= \begin{cases}
0 & \text{ if } \frac{n}{\delta}\notin\mathbb{N}^{*}, \\
\text{nonzero} & \text{ otherwise}, 
\end{cases}
\end{equation}
we obtain the stated result by simply putting altogether; that is 
$E_{6,\chi}(q^{\delta})=0$ for all $1<\delta\in\EuScript{D}(p_{j}^{e_{j}})$.
\end{proof} 

If $\alpha\beta\in\N$ holds, then the primitive Dirichlet characters 
are trivial. Therefore, the set $\EuScript{C}$ is empty. Hence, the case 
where $\alpha\beta\in\N$ holds is a special case of the following theorem. 
\begin{theorem} \label{basisCusp_a_b}
\begin{enumerate}
\item[\textbf{(a)}] Let $\EuScript{C}$ be a set of 
primitive Dirichlet characters which 
is not an annihilator of $\E_{6}(\Gamma_{0}(\alpha\beta))$. 
Then the set $\EuScript{B}_{E}=\{\, E_{6}(q^{t})\,\mid\, t\in\EuScript{D}
(\alpha\beta)\,\}\cup\underset{\chi\in\EuScript{C}}{\bigcup}\{\, 
E_{6,\chi}(q^{t})\mid t\in D_{\chi}(\alpha\beta)\,\}$ 
is a basis of $\E_{6}(\Gamma_{0}(\alpha\beta))$.
\item[\textbf{(b)}] Let $1\leq i\leq m_{S}$ be positive integers, 
$\delta\in D(\alpha\beta)$ and $(r(i,\delta))_{i,\delta}$ be a 
table of the powers of\  $\eta(\delta\,z)$. Let furthermore  
$\EuFrak{B}_{\alpha\beta,i}(q)=\underset{\delta|\alpha\beta}{\prod}\eta^{r(i,\delta)}(\delta\,z)$ 
be selected elements of $\S_{6}(\Gamma_{0}(\alpha\beta))$. 
Then the set $\EuScript{B}_{S}=\{\, \EuFrak{B}_{\alpha\beta,i}(q)\,\mid\,
1\leq i\leq m_{S}\,\}$ is a basis of $\S_{6}(\Gamma_{0}(\alpha\beta))$. 
\item[\textbf{(c)}] The set $\EuScript{B}_{M}=\EuScript{B}_{E}\cup\EuScript{B}_{S}$ 
constitutes a basis of $\M_{6}(\Gamma_{0}(\alpha\beta))$. 
\end{enumerate}
\end{theorem}
\begin{remark} \label{basis-remark}
\begin{enumerate}
\item[(r1)] Each $\EuFrak{B}_{\alpha\beta,i}(q)$ can be expressed in the form 
$\underset{n=1}{\overset{\infty}{\sum}}\EuFrak{b}_{\alpha\beta,i}(n)q^{n}$, where $1\leq i\leq m_{S}$ and 
for each $n\geq 1$ the coefficient $\EuFrak{b}_{\alpha\beta,i}(n)$ is an integer. 
\item[(r2)] 
If we divide the sum that results from \autoref{ligozat_theorem} $(iv')$, when $d=N$, 
by $24$, then we obtain the smallest positive degree of $q$ in 
$\EuFrak{B}_{\alpha\beta,i}(q)$.
\end{enumerate}
\end{remark}
Since the existence of a basis of the space of cusp forms in terms 
of theta series has been proved for all square-free 
levels by \aP\ \cite{pizer1976} and since \mN\ and \gL,\  
see \autoref{ligozat_theorem} $(i)-(iv')$, have determined a method to find as 
many theta series belonging to the space of cusp forms as possible, the proof of this 
theorem is essentially restricted to show that the selected elements of 
the space of modular forms of the given level are linearly independent. 
\begin{proof} 
\begin{enumerate} 
\item[(a)] 
\waS\ \cite[Thrms 5.8 and 5.9, p.~86]{wstein} has shown that for each 
$t$ positive divisor of $\alpha\beta$ it holds that 
$E_{6}(q^{t})$ is in $\M_{6}(\Gamma_{0}(t))$. 

Let $\iota$ be the absolute difference between 
$\text{dim}(\E_{6}(\Gamma_{0}(\alpha\beta)))$ and the cardinality of $D(\alpha\beta)$. 

If $\iota = 0$, then $\alpha\beta$ is in $\N$ and hence $\EuScript{C}=\emptyset$ and $\EuScript{B}_{E}=\{\, E_{6}(q^{t})\,\mid\, t\in\EuScript{D}
(\alpha\beta)\,\}$ is linearly independent. Since the dimension of $\E_{6}(\Gamma_{0}(\alpha\beta))$ is finite, it follows that it is a 
basis of $\E_{6}(\Gamma_{0}(\alpha\beta))$.

Suppose now that $\iota > 0$. Since $\E_{6}(\Gamma_{0}(t))$ is a 
vector space and the set $\EuScript{C}$ of primitive Dirichlet 
characters does not annihilates $\E_{6}(\Gamma_{0}(\alpha\beta))$, it then 
holds for each Legendre-Jacobi-Kronecker symbol 
$\chi\in\EuScript{C}$ and $t\in D_{\chi}(\alpha\beta)$ that $0\neq E_{6,\chi}(q^{t})$ is in $\E_{6}(\Gamma_{0}(t))$.
Since the dimension of $\E_{6}(\Gamma_{0}(\alpha\beta))$ is finite, 
it suffices to show that $\EuScript{B}_{E}$ is linearly independent. 
Suppose that for each $\chi\in\EuScript{C},s\in D_{\chi}(\alpha\beta)$ we have 
$z(\chi)_{s}\in\mathbb{C}$ and  
that for each $t|\alpha\beta$ we have $x_{t}\in\mathbb{C}$.  Then 
\begin{equation*}
\underset{t|\alpha\beta}{\sum}x_{t}\,E_{6}(q^{t}) + \underset{\chi\in\EuScript{C}}
{\sum}\biggl(\,\underset{s\in D_{\chi}(\alpha\beta)}{\sum}z(\chi)_{s} E_{6,\chi}(q^{s})\,\biggr)=0. 
\end{equation*}
We recall that $\chi$ is a Legendre-Jacobi-Kronecker symbol; therefore, for all 
$0\neq a\in\mathbb{Z}$ it holds that $\legendre{a}{0}=0$. Since 
the primitive Dirichlet characters $\chi$ is not trivial and has 
a conductor $L$ which we may assume greater than one, we can deduce 
that $C_{0}=0$ in (\hyperref[evalConvolClass-eqn-4]{\ref*{evalConvolClass-eqn-4}}). 
Then we equate the coefficients of 
$q^{n}$ for $n\in D(\alpha\beta)\cup\underset{\chi\in\EuScript{C}}{\bigcup}
\{\,s| s\in D_{\chi}(\alpha\beta)\,\}$ 
to obtain the homogeneous system of linear equations in $m_{E}$ unknowns:
\begin{equation*}
\underset{u|\alpha\beta}{\sum}\,\sigma_{5}(\frac{t}{u})x_{u} + 
\underset{\chi\in\EuScript{C}}
{\sum}\,\underset{v\in D_{\chi}(\alpha\beta)}{\sum}\,\chi(t)\sigma_{5}
(\frac{t}{v})Z(\chi)_{v} = 0,\qquad t\in D(\alpha\beta).
\end{equation*}
The determinant of the matrix of this homogeneous system of linear equations 
is not zero. Hence, the unique solution is 
$x_{t}=z(\chi)_{s}=0$ for all $t\in D(\alpha\beta)$ and for all 
$\chi\in\EuScript{C},s\in D_{\chi}(\alpha\beta)$. So, the set 
$\EuScript{B}_{E}$ is linearly independent and hence  
is a basis of $\E_{6}(\Gamma_{0}(\alpha\beta))$.

\item[(b)] We show that each $\EuFrak{B}_{\alpha\beta,i}(q)$, where $1\leq i\leq m_{S}$,
is in the space $\S_{6}(\Gamma_{0}(\alpha\beta))$. 
This is obviously the case  
since $\EuFrak{B}_{\alpha\beta,i}(q), 1\leq i\leq m_{S},$ are obtained using an 
exhaustive search which applies items (i)--(iv$\prime$) in 
\autoref{ligozat_theorem}. 

Since the dimension of $\S_{6}(\Gamma_{0}(\alpha\beta))$ is finite, it 
suffices to show that the set $\EuScript{B}_{S}$
is linearly independent. Suppose that $x_{i}\in\mathbb{C}$ and 
$\underset{i=1}{\overset{m_{S}}{\sum}}\,x_{i}\,\EuFrak{B}_{\alpha\beta,i}(q)=0$. 
Then $\underset{i=1}{\overset{m_{S}}{\sum}}\,x_{i}\,\EuFrak{B}_{\alpha\beta,i}(q)
= \underset{n=1}{\overset{\infty}{\sum}}(\,\underset{i=1}{\overset{m_{S}}{\sum}}\,x_{i}\,\EuFrak{b}_{\alpha\beta,i}(n)\,)q^{n} = 0$
which gives the homogeneous system of $m_{S}$ linear equations in $m_{S}$ unknowns:
\begin{equation}  \label{basis-cusp-eqn-sol}
\underset{i=1}{\overset{m_{S}}{\sum}}\EuFrak{b}_{\alpha\beta,i}(n)\,x_{i}= 0,\qquad 
1\leq n\leq m_{S}.
\end{equation}
Two cases arise:
\begin{description}
\item[The smallest degree of $\EuFrak{B}_{\alpha\beta,i}(q)$ is $i$ for each 
$1\leq i\leq m_{S}$] Then the square matrix which corresponds to 
this homogeneous system of $m_{S}$ linear equations is triangular with $1$'s on the 
diagonal. Hence, the determinant of that matrix is $1$ and 
so the unique solution is $x_{i}=0$ for all $1\leq i\leq m_{S}$. 
\item[The smallest degree of $\EuFrak{B}_{\alpha\beta,i}(q)$ is $i$ for  
$1\leq i< m_{S}$] Let $n'$ be the largest positive integer such that 
$1\leq i\leq n'< m_{S}$. Let 
$\EuScript{B}_{S}'=\{\,\EuFrak{B}_{\alpha\beta,i}(q)\,\mid\,1\leq i\leq n'\,\}$ and 
$\EuScript{B}_{S}''=\{\,\EuFrak{B}_{\alpha\beta,i}(q)\,\mid\,n'< i\leq m_{S}\,\}$. Then 
$\EuScript{B}_{S}=\EuScript{B}_{S}'\cup\EuScript{B}_{S}''$ and we may consider 
$\EuScript{B}_{S}$ as an ordered set. By the case above, the set 
$\EuScript{B}_{S}'$ is linearly 
independent. Hence, the linear independence of the set $\EuScript{B}_{S}$ depends 
on that of the set $\EuScript{B}_{S}''$.
Let $A=(\EuFrak{b}_{\alpha\beta,i}(n))$ be the $m_{S}\times m_{S}$ matrix 
in (\hyperref[basis-cusp-eqn-sol]{\ref*{basis-cusp-eqn-sol}}). 
If $\text{det}(A)\neq 0$, then $x_{i}=0$ for all $1\leq i\leq m_{S}$ and we are 
done. Suppose that $\text{det}(A)= 0$. Then for 
some $n'< l\leq m_{S}$ there exists $\EuFrak{B}_{\alpha\beta,l}(q)$ which is causing 
the system of equations to be inconsistent. We substitute $\EuFrak{B}_{\alpha\beta,l}(q)$ with, say 
    $\EuFrak{B}_{\alpha\beta,l}'(q)$, which does not occur in $\EuScript{B}_{S}$ and 
    compute the determinant of the new matrix $A$. 
    Since there are finitely many 
    $\EuFrak{B}_{\alpha\beta,l}(q)$ with $n'< l\leq m_{S}$ that may cause the 
    system of linear equations to be inconsistent and finitely many elements of 
    $\S_{6}(\Gamma_{0}(\alpha\beta))\setminus \EuScript{B}_{S}$, 
    the procedure will terminate with a consistent system of linear equations.  
Hence, we will find a set of elements of $\S_{6}(\Gamma_{0}(\alpha\beta))$ which is linearly independent.
\end{description}
Therefore, the set $\{\,\EuFrak{B}_{\alpha\beta,i}(q)\mid 1\leq i\leq m_{S}\,\}$ 
is linearly independent and hence  
is a basis of $\S_{6}(\Gamma_{0}(\alpha\beta))$.

\item[(c)] Since $\M_{6}(\Gamma_{0}(\alpha\beta))=\E_{6}(\Gamma_{0}(\alpha\beta))\oplus 
\S_{6}(\Gamma_{0}(\alpha\beta))$, the result follows from (a) and (b).
\end{enumerate} 
\end{proof}

If the level $\alpha\beta$ belongs to the class $\N$, then 
\autoref{basisCusp_a_b} (a) is provable by induction on the set of positive 
divisors of $\alpha\beta$; see for example \eN\ \cite{ntienjem2016b}. Note that  each positive divisor of $\alpha\beta$ is in the class $\N$ 
whenever the level $\alpha\beta$ belongs to the class $\N$.  
This nice property does not hold in general if 
the level $\alpha\beta$ belongs to the class $\mathbb{N}\setminus\N$. For 
example $18$ is an element of the class $\mathbb{N}\setminus\N$; however, $6$ 
which is a positive divisor of $18$ does not belong to $\mathbb{N}\setminus\N$. 

The proof of \autoref{basisCusp_a_b}(b) provides us with an effective method to 
determine a basis of the space of cusp forms of weight $0<2k\in\mathbb{N}$
and of level $\alpha\beta$ whenever $\alpha\beta$ belongs to $\mathbb{N}^{*}$.

\subsection{Evaluating the convolution sums $\mathbf{W_{(\boldsymbol{\upalpha,\upbeta})}^{1,3}(n)}$ and 
$\mathbf{W_{(\boldsymbol{\upalpha,\upbeta})}^{3,1}(n)}$}
\label{convolution_alpha_beta-gen}

We recall that it is sufficient to assume that the primitive Dirichlet 
character $\chi$ is not trivial since the case $\chi$ trivial can be 
concluded as an immediate corollary. 

\begin{lemma} \label{convolution-lemma_a_b}
Let $\alpha,\beta\in\mathbb{N}$ be such that $\text{gcd}(\alpha,\beta)=1$. 
Let furthermore $\EuScript{B}_{M}=\EuScript{B}_{E}\cup\EuScript{B}_{S}$ be a 
basis of $\M_{6}(\Gamma_{0}(\alpha\beta))$. Then there exist
$X_{\delta},Z(\chi)_{s}, Y_{j}\in\mathbb{C}, where 1\leq j\leq m_{S}, 
\chi\in\EuScript{C}, s\in D_{\chi}(\alpha\beta)\text{ and } \delta|\alpha\beta$, such that  
\begin{multline}
(\,\alpha\,E_{2}(q^{\alpha}) - E_{2}(q)\,)\, E_{4}(q^{\beta})
 = \sum_{\delta|\alpha\beta}X_{\delta}  
   +  \sum_{n=1}^{\infty}\biggl( 
 - 504\,\sum_{\delta|\alpha\beta}\sigma_{5}(\frac{n}{\delta})X_{\delta}  \\
 + \sum_{\chi\in\EuScript{C}}\,\sum_{s\in D_{\chi}(\alpha\beta)}\,\sigma_{5}(\frac{n}{s})\,
Z(\chi)_{s} 
 + \sum_{j=1}^{m_{S}}\,\EuFrak{b}_{\alpha\beta,j}(n)\,Y_{j}\biggr)\,q^{n}.
\label{convolution_a_b-eqn-0}
\end{multline}
and 
\begin{multline}
E_{4}(q^{\alpha})\,(\, E_{2}(q) - \beta\,E_{2}(q^{\beta})\,) 
= \sum_{\delta|\alpha\beta}X_{\delta}  
+  \sum_{n=1}^{\infty}\biggl( 
 - 504\,\sum_{\delta|\alpha\beta}\sigma_{5}(\frac{n}{\delta})X_{\delta}  \\
+ \sum_{\chi\in\EuScript{C}}\,\sum_{s\in D_{\chi}(\alpha\beta)}\,\sigma_{5}(\frac{n}{s})\,
Z(\chi)_{s} 
+ \sum_{j=1}^{m_{S}}\,\EuFrak{b}_{\alpha\beta,j}(n)\,Y_{j}\biggr)\,q^{n}.
\label{convolution_a_b-eqn-1}
\end{multline}
These also hold for the special cases $\beta=1\land \alpha\neq 1$, $\>\alpha = 1\land \beta\neq 1$ and $\beta=\alpha\neq 1$.
\end{lemma}  
\begin{proof} For the sake of simplicity, we only give the proof for 
	(\hyperref[convolution_a_b-eqn-1]{\ref*{convolution_a_b-eqn-1}}).
	The other cases can be proven in a similar way. 
	
That $E_{4}(q^{\alpha})\,(\, E_{2}(q) - \beta\,E_{2}(q^{\beta})\,) \in
\M_{6}(\Gamma_{0}(\alpha\beta))$ follows from  
\hyperref[evalConvolClass-lema-1]{Lemma \ref*{evalConvolClass-lema-1}}.  
Hence, by 
\autoref{basisCusp_a_b}\,(c), there exist 
$X_{\delta},Z(\chi)_{s}, Y_{j}\in\mathbb{C}, 
1\leq j\leq m_{S}, \chi\in\EuScript{C}, s\in D(\chi)\text{ and }
\delta$ is a divisor of $\alpha\beta$, such that  
\begin{multline*}
E_{4}(q^{\alpha})\,(\, E_{2}(q) - \beta\,E_{2}(q^{\beta})\,)  =
\sum_{\delta|\alpha\beta}X_{\delta}\,E_{6}(q^{\delta}) + \sum_{\chi\in\EuScript{C}}\,
\sum_{s\in D_{\chi}(\alpha\beta)}Z(\chi)_{s}\,E_{6,\chi}(q^{s}) \\
+ \sum_{j=1}^{m_{S}}Y_{j}\,\EuFrak{B}_{\alpha\beta,j}(q)    
= \sum_{\delta|\alpha\beta}X_{\delta} 
+ \sum_{n=1}^{\infty}\biggl(\, 
 - 504\,\sum_{\delta|\alpha\beta}\,\sigma_{5}(\frac{n}{\delta})\,X_{\delta}  \\
 + \sum_{\chi\in\EuScript{C}}\,\sum_{s\in D_{\chi}(\alpha\beta)}
 \,\chi(n)\,\sigma_{5}(\frac{n}{s})\,Z(\chi)_{s} 
  +  \sum_{j=1}^{m_{S}}\,\EuFrak{b}_{\alpha\beta,j}(n)\,Y_{j}\, \biggr)\,q^{n}. 
 \end{multline*} 
 We now consider the special cases $\beta=1\land \alpha\neq 1$,
 $\>\alpha = 1\land \beta\neq 1$ and $\beta=\alpha\neq 1$. We give the proof
 for just one case. 
 We equate the right hand side of 
(\hyperref[convolution_a_b-eqn-0]{\ref*{convolution_a_b-eqn-0}}) 
with that of 
(\hyperref[evalConvolClass-eqn-sp-3]{\ref*{evalConvolClass-eqn-sp-3}})
to obtain
\begin{multline*}
 - 504\,\sum_{\delta|\alpha\beta}X_{\delta}\,\sigma_{5}(\frac{n}{\delta}) + 
  \sum_{\chi\in\EuScript{C}}\biggl(\,\sum_{s\in D_{\chi}(\alpha\beta)}\,\chi(n)\,\sigma_{5}
  (\frac{n}{s})\,Z(\chi)_{s}\,\biggr)  
 + \sum_{j=1}^{m_{S}}Y_{j}\,\EuFrak{b}_{\alpha\beta,j}(n) \\ 
  =  
 24\,\sigma(n) - 240\,\sigma_{3}(\frac{n}{\beta}) \\ 
+ 720\,n\,\sigma_{3}(\frac{n}{\beta}) - 504\,\beta\,\sigma_{5}(\frac{n}{\beta})
+ 24\times 240\,W_{(1,\beta)}^{1,3}(n). 
\end{multline*}
We then take the coefficients of $q^{n}$ such that $n$ is in $D(\alpha\beta)$
and $1\leq n\leq m_{S}$, but as many as the unknown, $X_{1},\ldots,X_{\alpha\beta}$,   
$Z(\chi)_{s}$ for all $\chi\in\EuScript{C}, s\in D(\chi)$,   
and $Y_{1},\ldots,Y_{m_{S}}$, to obtain 
a system of $m_{E}+m_{S}$ linear equations whose unique solution determines 
the values of the unknowns.
Hence, we obtain the result.
\end{proof}

For the following theorem, let for the sake of simplicity 
$X_{\delta},Z(\chi)_{s}$ and
$Y_{j}$ stand for their values obtained in the previous lemma.
\begin{theorem} \label{convolution_a_b}
Let $n$ be a positive integer. Then 
\begin{description}
	\item[Case $\beta=1$ and $\alpha\neq 1$] 
\begin{align}  \label{convolution_a_1-1_3}
 W_{(\alpha,1)}^{1,3}(n)   =\,  & 
    \frac{7}{80\,\alpha}\,(1 + X_{1})\,\sigma_{5}(n) 
 + \frac{1}{24}\,\sigma_{3}(n) - \frac{1}{8\,\alpha}\,n\,\sigma_{3}(n) 
  -   \frac{1}{240}\,\sigma(\frac{n}{\alpha})  \notag \\ &
 + \frac{7}{80\,\alpha}\,\sum_{\substack{{\delta|\alpha}\\{\delta\neq 1} }}\,
  X_{\delta}\,\sigma_{5}(\frac{n}{\delta})  
   - \frac{1}{5760\,\alpha}\,\sum_{j=1}^{m_{S}}\,Y_{j}\,\EuFrak{b}_{\alpha,j}(n)  \notag  \\ &
     - \frac{1}{5760\,\alpha}\,\sum_{\chi\in\EuScript{C}}\,\sum_{s\in 
    	D_{\chi}(\alpha)}Z(\chi)_{s}\,\sigma_{5}(\frac{n}{s})  
\end{align}
	\item[Case $\beta\neq 1$ and $\alpha=1$] 
\begin{align} \label{convolution_1_b-3_1}
W_{(1,\beta)}^{3,1}(n)   = \, & 
\frac{7}{80\,\beta}\,(\,1 - X_{1}\,)\,\sigma_{5}(n) 
+ \frac{1}{24}\,\sigma_{3}(n) - \frac{1}{8\,\beta}\,n\,\sigma_{3}(n) 
- \frac{1}{240}\,\sigma(\frac{n}{\beta})  \notag \\ &
- \frac{7}{80\,\beta}\,\sum_{\substack{{\delta|\beta}\\{\delta\neq 1} }}\,
X_{\delta}\,\sigma_{5}(\frac{n}{\delta})  
+ \frac{1}{5760\,\beta}\,\sum_{j=1}^{m_{S}}\,Y_{j}\,\EuFrak{b}_{\beta,j}(n)  \notag  \\ &
+ \frac{1}{5760\,\beta}\,\sum_{\chi\in\EuScript{C}}\,\sum_{s\in 
	D_{\chi}(\beta)}\,Z(\chi)_{s}\,\sigma_{5}(\frac{n}{s})  
\end{align}
	\item[Case $\beta=\alpha\neq 1$] 
\begin{align} \label{convolution_1_b-1_3}
W_{(1,\beta)}^{1,3}(n)   =\,  & 
\frac{7}{80}\,(\,\beta - X_{\beta}\,)\,\sigma_{5}(\frac{n}{\beta}) 
+ \frac{1}{24}\,\sigma_{3}(\frac{n}{\beta}) - \frac{1}{8}\,n\,\sigma_{3}(\frac{n}{\beta}) - \frac{1}{240}\,\sigma(n) 
  \notag \\ &
- \frac{7}{80}\,\sum_{\substack{{\delta|\beta}\\{\delta\neq\beta} }}\,
X_{\delta}\,\sigma_{5}(\frac{n}{\delta})  
+ \frac{1}{5760}\,\sum_{j=1}^{m_{S}}\,Y_{j}\,\EuFrak{b}_{\beta,j}(n)  \notag  \\ &
+ \frac{1}{5760}\,\sum_{\chi\in\EuScript{C}}\,\sum_{s\in 
	D_{\chi}(\beta)}\,Z(\chi)_{s}\,\sigma_{5}(\frac{n}{s})  
\end{align}
and
\begin{align} \label{convolution_a_1-3_1}
W_{(\alpha,1)}^{3,1}(n)   = \, & 
\frac{7}{80}\,(\,\alpha + X_{\alpha}\,)\,\sigma_{5}(\frac{n}{\alpha}) 
+ \frac{1}{24}\,\sigma_{3}(\frac{n}{\alpha}) - \frac{1}{8}\,n\,\sigma_{3}(\frac{n}{\alpha}) - \frac{1}{240}\,\sigma(n) 
\notag \\ &
+ \frac{7}{80}\,\sum_{\substack{{\delta|\beta}\\{\delta\neq\alpha} }}\,
X_{\delta}\,\sigma_{5}(\frac{n}{\delta})  
- \frac{1}{5760}\,\sum_{j=1}^{m_{S}}\,Y_{j}\,\EuFrak{b}_{\beta,j}(n)  \notag  \\ &
- \frac{1}{5760}\,\sum_{\chi\in\EuScript{C}}\,\sum_{s\in 
	D_{\chi}(\beta)}\,Z(\chi)_{s}\,\sigma_{5}(\frac{n}{s})  
\end{align}
\end{description}
\end{theorem}
\begin{proof}  The proof is straightforward. We only prove the case $\alpha=\beta\neq 1$ 
	and in particular the identity (\hyperref[convolution_a_1-3_1]{\ref*{convolution_a_1-3_1}}).
	
	We equate the right hand side of 
(\hyperref[convolution_a_b-eqn-1]{\ref*{convolution_a_b-eqn-1}}) with that of  
(\hyperref[evalConvolClass-eqn-sp-4]{\ref*{evalConvolClass-eqn-sp-4}}) to yield 
\begin{multline*}
1 - \alpha  
+ \sum_{n=1}^{\infty}\biggl(\, - 24\,\sigma(n) +  240\,\sigma_{3}(\frac{n}{\alpha})   \\ 
-  720\,n\,\sigma_{3}(\frac{n}{\alpha})  
+ 504\,\alpha\,\sigma_{5}(\frac{n}{\alpha})
- 24\times 240\, W_{(\alpha,1)}^{3,1}(n)\,\biggr)q^{n}  =  
\sum_{\delta|\alpha}X_{\delta}  
+  \sum_{n=1}^{\infty}\biggl( 
- 504\,\sum_{\delta|\alpha}\,X_{\delta}\,\sigma_{5}(\frac{n}{\delta})  \\
+ \sum_{\chi\in\EuScript{C}}\,\sum_{s\in D_{\chi}(\alpha)}\,Z(\chi)_{s}\,\sigma_{5}(\frac{n}{s}) 
+ \sum_{j=1}^{m_{S}}\,Y_{j}\,\EuFrak{b}_{\alpha,j}(n)\,\biggr)\,q^{n}.
\end{multline*} 
We then obtain  
 \begin{multline*} 
 24\times 240\, W_{(\alpha,1)}^{3,1}(n) =  - 24\,\sigma(n) + 240\,\sigma_{3}(\frac{n}{\alpha})   
-  720\,n\,\sigma_{3}(\frac{n}{\alpha})  
+ 504\,\alpha\,\sigma_{5}(\frac{n}{\alpha})  \\ 
 + 504\,\sum_{\delta|\beta}\,X_{\delta}\,\sigma_{5}(\frac{n}{\delta}) 
- \sum_{\chi\in\EuScript{C}}\,\sum_{s\in D_{\chi}(\alpha\beta)}\,Z(\chi)_{s} \,\sigma_{5}(\frac{n}{s})   
- \sum_{j=1}^{m_{S}}\,Y_{j}\,\EuFrak{b}_{\beta,j}(n)
\end{multline*}
We then solve for $W_{(1,\beta)}^{3,1}(n)$ to obtain the stated result.
\end{proof}

 \begin{remark}
  We immediately observe that in the identity
  \begin{itemize}
    \item  
    (\hyperref[convolution_a_1-1_3]{\ref*{convolution_a_1-1_3}})\quad the part\quad $\frac{1}{24}\,\sigma_{3}(n) - \frac{1}{8\,\alpha}\,n\,\sigma_{3}(n) 
  -   \frac{1}{240}\,\sigma(\frac{n}{\alpha})$
    \item  
    (\hyperref[convolution_1_b-3_1]{\ref*{convolution_1_b-3_1}})\quad the part\quad  $\frac{1}{24}\,\sigma_{3}(n) - \frac{1}{8\,\beta}\,n\,\sigma_{3}(n) 
- \frac{1}{240}\,\sigma(\frac{n}{\beta})$
    \item  
    (\hyperref[convolution_1_b-1_3]{\ref*{convolution_1_b-1_3}})\quad the part\quad  $\frac{1}{24}\,\sigma_{3}(\frac{n}{\beta}) - \frac{1}{8}\,n\,\sigma_{3}(\frac{n}{\beta}) - \frac{1}{240}\,\sigma(n)$
    \item 
    (\hyperref[convolution_a_1-3_1]{\ref*{convolution_a_1-3_1}})\quad the part\quad  $\frac{1}{24}\,\sigma_{3}(\frac{n}{\alpha}) - \frac{1}{8}\,n\,\sigma_{3}(\frac{n}{\alpha}) - \frac{1}{240}\,\sigma(n)$
  \end{itemize}
  relies each only on $n$, $\alpha$ and $\beta$.  
The basis of the modular space $\M_{6}(\Gamma_{0}(\alpha\beta))$ is not involved. 
\end{remark}

By \hyperref[convolutionSum-lema-sp]{Lemma \ref*{convolutionSum-lema-sp}} and by  \hyperref[convolution-lemma_a_b]{Lemma \ref*{convolution-lemma_a_b}} we obtain the following
\begin{lemma} \label{convolutionSum-coro-sp}
Let $1<\alpha\in\mathbb{N}^{*}$. Then
\begin{align*} 
W_{(\alpha,1)}^{1,3}(n) = W_{(1,\alpha)}^{3,1}(n)
 \quad \text{and} \quad
W_{(1,\alpha)}^{1,3}(n) = W_{(\alpha,1)}^{3,1}(n).
\end{align*} 
\end{lemma}

\begin{theorem} \label{convolutionSum-theor-sp}
	Let $\alpha,\beta\in\mathbb{N}^{*}$ be such that 
	$\gcd(\alpha,\beta)=1$. Then
	\begin{align*} 
	W_{(\alpha,\beta)}^{1,3}(n) = W_{(\beta,\alpha)}^{3,1}(n)
	\quad \text{and} \quad
	W_{(\beta,\alpha)}^{1,3}(n) = W_{(\alpha,\beta)}^{3,1}(n).
	\end{align*} 
\end{theorem}
\begin{proof}  Let $\alpha,\beta\in\mathbb{N}^{*}$ be such that 
$\gcd(\alpha,\beta)=1$. Then we have 
\begin{multline*}
 (\,\alpha\, E_{2}(q^{\alpha}) -  E_{2}(q)\,)\,E_{4}(q^{\beta})  = 
\alpha - 1 
+ \sum_{n=1}^{\infty}\biggl(\, 24\,\sigma(n) - 24\,\alpha\,\sigma(\frac{n}{\alpha}) \\ 
+ 240\,\alpha\,\sigma_{3}(\frac{n}{\beta}) 
- 240\,\sigma_{3}(\frac{n}{\beta})
+ 24\times 240\,W_{(1,\beta)}^{1,3}(n) 
-  24\times 240\,\alpha\, W_{(\alpha,\beta)}^{1,3}(n)\,\biggr)q^{n} \\
 = 1 - \alpha + \sum_{n=1}^{\infty}\biggl(\, - 24\,\sigma(n) + 24\,\alpha\,\sigma(\frac{n}{\alpha}) \\ 
 - 240\,\alpha\,\sigma_{3}(\frac{n}{\beta}) 
 + 240\,\sigma_{3}(\frac{n}{\beta})
 - 24\times 240\,W_{(1,\beta)}^{1,3}(n) 
 +  24\times 240\,\alpha\, W_{(\alpha,\beta)}^{1,3}(n)\,\biggr)q^{n} \\
= E_{4}(q^{\beta})\,(\, E_{2}(q)  - \,\alpha\, E_{2}(q^{\alpha})\,) 
 = 1 - \alpha  + \sum_{n=1}^{\infty}\biggl(\, - 24\,\sigma(n) + 24\,\alpha\,\sigma(\frac{n}{\alpha}) \\ 
- 240\,\alpha\,\sigma_{3}(\frac{n}{\beta}) 
+ 240\,\sigma_{3}(\frac{n}{\beta})
- 24\times 240\,W_{(\beta,1)}^{3,1}(n) 
+  24\times 240\,\alpha\, W_{(\beta,\alpha)}^{3,1}(n)\,\biggr)q^{n} 
\end{multline*}
and 
\begin{multline*}
E_{4}(q^{\alpha})\,(\, E_{2}(q) - \beta\,E_{2}(q^{\beta})\,) = 
1 - \beta  
+ \sum_{n=1}^{\infty}\biggl(\, 24\,\beta\,\sigma(\frac{n}{\beta}) -  24\,\sigma(n)   \\ 
+  240\,\sigma_{3}(\frac{n}{\alpha})  
- 240\,\beta\,\sigma_{3}(\frac{n}{\alpha}) - 24\times 240\, W_{(\alpha,1)}^{3,1}(n)
+ 24\times 240\,\beta\, W_{(\alpha,\beta)}^{3,1}(n)\,\biggr)q^{n} \\
= \beta - 1  
+ \sum_{n=1}^{\infty}\biggl(\, - 24\,\beta\,\sigma(\frac{n}{\beta}) +  24\,\sigma(n)   \\ 
-  240\,\sigma_{3}(\frac{n}{\alpha})  
+ 240\,\beta\,\sigma_{3}(\frac{n}{\alpha}) + 24\times 240\, W_{(\alpha,1)}^{3,1}(n)
- 24\times 240\,\beta\, W_{(\alpha,\beta)}^{3,1}(n)\,\biggr)q^{n} \\
= (\,\beta\,E_{2}(q^{\beta}) -  E_{2}(q)\,)\,E_{4}(q^{\alpha}) 
= \beta - 1  
+ \sum_{n=1}^{\infty}\biggl(\, - 24\,\beta\,\sigma(\frac{n}{\beta}) +  24\,\sigma(n)   \\ 
-  240\,\sigma_{3}(\frac{n}{\alpha})  
+ 240\,\beta\,\sigma_{3}(\frac{n}{\alpha}) + 24\times 240\, W_{(1,\alpha)}^{1,3}(n)
- 24\times 240\,\beta\, W_{(\beta,\alpha)}^{1,3}(n)\,\biggr)q^{n} 
\end{multline*}
Now we consider $W_{(1,\beta)}^{1,3}(n)$ and $W_{(\alpha,1)}^{3,1}(n)$; 
then apply 
\hyperref[convolutionSum-coro-sp]{Corollary \ref*{convolutionSum-coro-sp}}. 
When we then compare both side in each case,  
we obtain the stated result. 
\end{proof}

We now have the prerequisite to determine a formula for the number of representations of a positive integer $n$ by a quadratic form.


\section{Number of Representations of a positive Integer for this Class 
	of Levels}
\label{representations_a_b-c_d}

We discuss in this section the determination of formulae for the number of 
representations of a positive integer by the quadratic forms 
(\hyperref[introduction-eq-1]{\ref*{introduction-eq-1}}) -- 
(\hyperref[introduction-eq-4]{\ref*{introduction-eq-4}}).

\subsection{Representations of a positive Integer by the Quadratic Form  (\hyperref[introduction-eq-1]{\ref*{introduction-eq-1}}) and 
	(\hyperref[introduction-eq-2]{\ref*{introduction-eq-2}})
}\label{representations_a_b}

We determine formulae for the number of representations of a positive integer by 
the quadratic forms (\hyperref[introduction-eq-1]{\ref*{introduction-eq-1}}) and 
(\hyperref[introduction-eq-2]{\ref*{introduction-eq-2}}).

\subsubsection{Formulae for the Number of Representations by 
	(\hyperref[introduction-eq-1]{\ref*{introduction-eq-1}}) and 
	(\hyperref[introduction-eq-2]{\ref*{introduction-eq-2}})
}\label{represent_a_b}

Let $n\in\mathbb{N}$ and let the number of representations
of $n$ by the quadratic form  $\underset{i=1}{\overset{12}{\sum}}x_{i}^{2}$ be the combination of 
\[ r_{4}(n)=\text{card}(\{(x_{1},x_{2}, x_{3},x_{4})\in\mathbb{Z}^{4}~|~ n\, =\, \sum_{i=1}^{4}\,x_{i}^{2}\,\}) \]
and 
\[ r_{8}(n)=\text{card}(\{(x_{1},x_{2},\ldots,x_{7},x_{8})\in\mathbb{Z}^{8}~|~ n\, =\, \sum_{i=1}^{8}\,x_{i}^{2}\,\}). \]
It follows from the definitions that $r_{4}(0) =r_{8}(0) = 1$.  
For all $n\in\mathbb{N}^{*}$ the Jacobi's identities $r_{4}(n)$ and
$r_{8}(n)$ are 
\begin{equation}
r_{4}(n) = 8\,\sigma(n) - 32\,\sigma(\frac{n}{4}) \label{representations-eqn-4-1}
\end{equation}
and 
\begin{equation}
	r_{8}(n) = 16\,\sigma_{3}(n) - 32\,\sigma_{3}(\frac{n}{2}) +  256\,\sigma_{3}(\frac{n}{4}), \label{representations-eqn-4-2}
\end{equation}
respectively, and an arithmetic proof of the identity 
(\hyperref[representations-eqn-4-1]{\ref*{representations-eqn-4-1}}) 
is given by \ksW \cite{williams2001} and that of the identity 
(\hyperref[representations-eqn-4-2]{\ref*{representations-eqn-4-2}}) is 
provided by \gaL\ \cite{lomadze}. 

Now, let the number of representations of $n$ by the quadratic form 
(\hyperref[introduction-eq-1]{\ref*{introduction-eq-1}}) or 
(\hyperref[introduction-eq-2]{\ref*{introduction-eq-2}}) be   
\begin{multline*}
N_{(a_{1},b_{1})}^{4i,4j}(n) 
=\text{card}
(\{(x_{1},x_{2},\ldots,x_{11},x_{12})\in\mathbb{Z}^{12}~|~
n = a_{1}\,\sum_{i=1}^{4}\,x_{i}^{2}  + b_{1}\,\sum_{i=5}^{12}\,x_{i}^{2} \}),
\end{multline*}
where $a_{1},b_{1},i,j\in\mathbb{N}^{*}$ such that $i+j=3$ 
indicates the number of variables used in the representation of $n$. 
In particular we have   
\begin{multline*}
	N_{(a,b)}^{4,8}(n) 
	=\text{card}
	(\{(x_{1},x_{2},\ldots,x_{11},x_{12})\in\mathbb{Z}^{12}~|~
	n = a\,\sum_{i=1}^{4}\,x_{i}^{2}  + b\,\sum_{i=5}^{12}\,x_{i}^{2} \}),
\end{multline*}
and 
\begin{multline*}
N_{(a_{1},b_{1})}^{8,4}(n) 
=\text{card}
(\{(x_{1},x_{2},\ldots,x_{11},x_{12})\in\mathbb{Z}^{12}~|~
n = a_{1}\,\sum_{i=1}^{8}\,x_{i}^{2} +
  b_{1}\,\sum_{i=9}^{12}\,x_{i}^{2} \}),
\end{multline*}
where $a,b,a_{1},b_{1}\in\mathbb{N}^{*}$. 

It immediately follows from the definition of $N_{(a,b)}^{4,8}(n)$ and 
$N_{(a_{1},b_{1})}^{8,4}(n)$ that if 
$a,b,a_{1},b_{1}\in\mathbb{N}^{*}$ are such that $\gcd(a,b)=d>1$ and 
$\gcd(a_{1},b_{1})=d_{1}>1$ for some 
$d,d_{1}\in\mathbb{N}^{*}$, then $N_{(a,b)}^{4,8}(n)= N_{(\frac{a}{d},\frac{b}{d})}^{1,3}(\frac{n}{d})$ and 
$N_{(a_{1},b_{1})}^{8,4}(n)= N_{(\frac{a_{1}}{d_{1}},\frac{b_{1}}{d_{1}})}^{8,4}(\frac{n}{d_{1}})$. Therefore, one can simply assume 
that $a,b,a_{1},b_{1}\in\mathbb{N}^{*}$ are relatively prime. 

We then derive the following result:
\begin{theorem} \label{representations-theor_a_b}
	Let $n\in\mathbb{N}$ and let $a,b,a_{1},b_{1}\in\mathbb{N}^{*}$ be relatively prime. Then  
	\begin{align*}
		N_{(a,b)}^{4,8}(n)  = ~ & 
		8\,\sigma(\frac{n}{a}) - 32\,\sigma(\frac{n}{4a}) + 16\,\sigma_{3}(\frac{n}{b}) - 32\,\sigma_{3}(\frac{n}{2b}) 
		+ 256\sigma_{3}(\frac{n}{4b})   \\ & 
		+ 128\, W_{(a,b)}^{1,3}(n) - 256\, W_{(a,2b)}^{1,3}(n) + 2048\, W_{(a,4b)}^{1,3}(n) - 512\, W_{(4a,b)}^{1,3}(n)      \\ & 
		+ 1024\, W_{(2a,b)}^{1,3}(\frac{n}{2})     
		- 8192\, W_{(a,b)}^{1,3}(\frac{n}{4}) 
	\end{align*}
	and
	\begin{align*}
N_{(a_{1},b_{1})}^{8,4}(n)  = ~ & 
8\,\sigma(\frac{n}{b_{1}}) - 32\,\sigma(\frac{n}{4b_{1}}) 
+ 16\,\sigma_{3}(\frac{n}{a_{1}}) - 32\,\sigma_{3}(\frac{n}{2a_{1}}) 
+ 256\sigma_{3}(\frac{n}{4a_{1}})   \\ & 
+ 128\, W_{(a_{1},b_{1})}^{3,1}(n) - 512\, W_{(a_{1},4b_{1})}^{3,1}(n) - 256\, W_{(2a_{1},b_{1})}^{3,1}(n)  + 1024\, W_{(a_{1},2b_{1})}^{3,1}(\frac{n}{2})       \\ & 
   + 2048\, W_{(4a_{1},b_{1})}^{3,1}(n) - 8192\, W_{(a_{1},b_{1})}^{3,1}(\frac{n}{4}) 
\end{align*}
\end{theorem}
\begin{proof} It suffices to prove the identity $N_{(a,b)}^{4,8}(n)$ 
	since the other can be proved similarly or by using \autoref{convolutionSum-theor-sp}. 
	
	We have 
	\begin{multline*}
		N_{(a,b)}^{4,8}(n)  = \sum_{\substack{
				{(l,m)\in\mathbb{N}^{2}} \\ {a\,l+b\,m=n}
			}}r_{4}(l)r_{8}(m) 
			= r_{4}(\frac{n}{a})r_{8}(0) + r_{4}(0)r_{8}(\frac{n}{b}) 
			+ \sum_{\substack{
					{(l,m)\in\mathbb{N}^{*}\times\mathbb{N}^{*}} \\ {a\,l+b\,m=n}
				}}r_{4}(l)r_{8}(m).
			\end{multline*}
			We make use of (\hyperref[representations-eqn-4-1]{\ref*{representations-eqn-4-1}})
			to obtain 
			\begin{multline*}
				N_{(a,b)}^{4,8}(n)  = 8\,\sigma(\frac{n}{a}) - 32\,\sigma(\frac{n}{4a}) + 16\,\sigma_{3}(\frac{n}{b}) - 32\,\sigma_{3}(\frac{n}{2b}) + 256\sigma_{3}(\frac{n}{4b})  \\
				+ \sum_{\substack{
						{(l,m)\in\mathbb{N}^{*}\times\mathbb{N}^{*}} \\ {al+bm=n}
					}} (8\,\sigma(l) - 32\,\sigma(\frac{l}{4}) )(16\,\sigma_{3}(m) - 32\,\sigma_{3}(\frac{m}{2}) +  256\,\sigma_{3}(\frac{m}{4})). 
				\end{multline*}
				We know that 
				\begin{multline*}
					(8\,\sigma(l) - 32\,\sigma(\frac{l}{4}))(16\,\sigma_{3}(m) - 32\,\sigma_{3}(\frac{m}{2}) +  256\,\sigma_{3}(\frac{m}{4}))  = \\
					128\,\sigma(l)\sigma_{3}(m) - 256\,\sigma(l)\sigma_{3}(\frac{m}{2}) 
					+ 2048\,\sigma(l)\sigma_{3}(\frac{m}{4})  \\
					- 512\,\sigma(\frac{l}{4})\sigma_{3}(m) + 1024\,\sigma(\frac{l}{4})\sigma_{3}(\frac{m}{2}) 
					- 8192\,\sigma_{3}(\frac{l}{4})\sigma_{3}(\frac{m}{4}) 
				\end{multline*}
				In the sequel of this proof, we assume that the evaluation of 
				\begin{equation*}
					W_{(a,b)}^{1,3}(n) = \sum_{\substack{
							{(l,m)\in\mathbb{N}^{*}\times\mathbb{N}^{*}} \\ {al+bm=n}
						}}\sigma(l)\sigma_{3}(m),
					\end{equation*}
					$W_{(a,2b)}^{1,3}(n),\>W_{(a,4b)}^{1,3}(n),\>
					W_{(4a,b)}^{1,3}(n),\>W_{(4a,2b)}^{1,3}(n)$ and $W_{(4a,4b)}^{1,3}(n)$ are known. 
					
					Let $u,v\in\mathbb{N}^{*}$ be given and $f,g:\mathbb{N}\mapsto\mathbb{N}$ be  
					injective functions such that $f(n)=u\cdot n$ and 
					$g(n)=v\cdot n$ for each $n\in\mathbb{N}$.
					
					When we simultaneously apply the functions 
					$f$ and $g$ with $l$ and $m$ as argument, respectively, we derive  
					\begin{equation*}
						W_{(a,b)}^{1,3}(n) = \sum_{\substack{
								{(l,m)\in\mathbb{N}^{*}\times\mathbb{N}^{*}} \\ {al+bm=n}
							}}\sigma(\frac{l}{u})\sigma_{3}(\frac{m}{v}) 
							= \sum_{\substack{
									{(l,m)\in\mathbb{N}^{*}\times\mathbb{N}^{*}} \\ {ua\,l+vb\,m=n}
								}}\sigma(l)\sigma_{3}(m) = W_{(ua,vb)}^{1,3}(n)
							\end{equation*}
								We set $(u,v) = (1,1), (1,2), (1,4), (4,1),(4,2),(4,4)$ accordingly and then  
									finally put all these evaluations together to obtain the stated result 
									for $N_{(a,b)}^{4,8}(n)$. 
								\end{proof}
								
								From this proof, one immediately observe that a formula for the number 
								of representations of a positive integer $n$  by the quadratic forms 
								(\hyperref[introduction-eq-1]{\ref*{introduction-eq-1}}) and (\hyperref[introduction-eq-2]{\ref*{introduction-eq-2}}) depends on the 
								evaluated convolution Sums for some given levels $ab, 2ab$ and $4ab$ with $a,b\in\mathbb{N}^{*}$ relatively prime. 
								
								Based on this observation, we only take into consideration those levels $\alpha\beta$ which are multiple of $4$; 
							    that is $\alpha\beta\equiv 0\pmod{4}$. 
								
								\subsubsection{Determination of all relevant $\mathbf{(a,b)\in\mathbb{N}^{*}\times\mathbb{N}^{*}}$ for $N_{(a,b)}(n)$ for a given $\alpha\beta\in\mathbb{N}^{*}$} 
								\label{determine_a_b}
								We carry out a method to determine all pairs $(a,b)\in\mathbb{N}^{*}\times\mathbb{N}^{*}$ which 
								are necessary for the determination of $N_{(a,b)}(n)$ for a given level 
								$\alpha\beta\in\mathbb{N}^{*}$ such that $\alpha\beta\equiv 0\pmod{4}$ holds. 
								
								Let $\Lambda=\frac{\alpha\beta}{4}=2^{\nu-2}\mho$, 
								$P_{4}=\{p_{0}=2^{\nu-2}\}\cup\underset{j>1}{\bigcup}\{ p_{j}\,\mid\, p_{j}
								\text{ is a prime divisor of } \mho\,\}$ and  
								$\EuScript{P}(P_{4})$ be the power set of $P_{4}$. Then for each $Q\in\EuScript{P}(P_{4})$ 
								we define $\mu(Q)=\underset{p\in Q}{\prod}p$. We set $\mu(Q)=1$ if $Q$ is an 
								empty set. Let now 
								\begin{multline*}
									\Omega_{4}=\{(\mu(Q_{1}),\mu(Q_{2}))~|~\text{ there exist } Q_{1},Q_{2}\in\EuScript{P}(P_{4}) 
									\text{ such that }  \\ 
									\gcd{(\mu(Q_{1}),\mu(Q_{2}))}=1\,\text{ and } 
									\mu(Q_{1})\,\mu(Q_{2})=\Lambda\,\}.
								\end{multline*}
								Observe that $\Omega_{4}\neq\emptyset$ since $(1,\Lambda)\in\Omega_{4}$.
								
								To illustrate our method, suppose that $\alpha\beta=2^{3}\cdot 3\cdot 5$. Then 
								$\Lambda=2\cdot 3\cdot 5$, $P_{4}=\{2,3,5\}$ and 
								$\Omega_{4}=\{(1,30),(2,15),(3,10),(5,6)\}$.
								\begin{proposition} \label{representation-prop-1}
									Suppose that the level $\alpha\beta\in\mathbb{N}^{*}$ and $\alpha\beta\equiv 0\pmod{4}$. 
									Furthermore, suppose that $\Omega_{4}$ is 
									defined as above. Then for all $n\in\mathbb{N}$ the set $\Omega_{4}$ 
									contains all pairs $(a,b)\in\mathbb{N}^{*}\times\mathbb{N}^{*}$ such that $N_{(a,b)}(n)$ can 
									be obtained by applying $W_{(\alpha,\beta)}(n)$ and some other evaluated 
									convolution sums.
								\end{proposition}
								\begin{proof}
									We prove this by induction on the structure of the level $\alpha\,\beta$.
									
									Suppose that $\alpha\beta=2^{\nu}p_{2}$, where $\nu\in\{2, 3\}$ and $p_{2}$ is 
									an odd prime. Then by the above definitions we have $\Lambda=2^{\nu-2}p_{2}$, 
									$P_{4}=\{\,2^{\nu-2},p_{2}\,\}$, 
									$$\EuScript{P}(P_{4})=\{\,\emptyset,\{2^{\nu-2}\},\{p_{2}\},\{\,2^{\nu-2},p_{2}\}\,\},$$
									and $\Omega_{4}=\{\,(1,2^{\nu-2}p_{2}),(2^{\nu-2},p_{2}) \,\}$.
									
									Following the observation made at the end of the proof of 
									\autoref{representations-theor_a_b}, we note that 
									$\alpha\beta=4ab=2^{\nu}p_{2}$. Hence, $ab=2^{\nu-2}p_{2}$ which leads  immediately to $N_{(a,b)}(n)$. 
									
									We show that $\Omega_{4}$ is the largest such set. 
									Assume now that there exist another set, say $\Omega'_{4}$, which results 
									from the above definitions. Then there are two cases.
									\begin{description} 
										\item[Case $\Omega'_{4}\subseteq\Omega_{4}$] There is nothing to show. 
										So, we are done.
										\item[Case $\Omega_{4}\subset\Omega'_{4}$] Let
										$(e,f)\in\Omega_{4}'\setminus\Omega_{4}$.  Since $ef=2^{\nu-2}p_{2}$ and 
										$\gcd{(e,f)}=1$, we must
										have either $(e,f)=(1,2^{\nu-2}p_{2})$ or $(e,f)=(2^{\nu-2},p_{2})$. So, 
										$(e,f)\in\Omega_{4}$. Hence, $\Omega_{4}=\Omega'_{4}$.
									\end{description} 
									Suppose now that $\alpha\beta=2^{\nu}p_{2}p_{3}$, where $\nu\in\{2, 3\}$ and 
									$p_{2},p_{3}$ are distinct odd primes. Then by the induction hypothesis and by
									the above definitions we have essentially  
									$$\Omega_{4}=\{\,(1,2^{\nu-2}p_{2}p_{3}),(2^{\nu-2},p_{2}p_{3}),(2^{\nu-2}p_{2},p_{3}),(2^{\nu-2}p_{3},p_{2})
									\,\}.$$ 
									One notes that 
									$\alpha\beta=4ab=2^{\nu}p_{2}p_{3}$. Hence, $ab=2^{\nu-2}p_{2}p_{3}$ which immediately gives $N_{(a,b)}(n)$. 
									
									Again, we show that $\Omega_{4}$ is the largest such set. 
									Suppose that there exist another set, say $\Omega'_{4}$, which 
									results from the above definitions. Two cases arise.
\begin{description} 
\item[Case $\Omega'_{4}\subseteq\Omega_{4}$] There is nothing to prove. 
So, we are done.
\item[Case $\Omega_{4}\subset\Omega'_{4}$] Let
$(e,f)\in\Omega_{4}'\setminus\Omega_{4}$.  Since $ef=2^{\nu-2}p_{2}p_{3}$ and
$\gcd{(e,f)}=1$, we must
have $(e,f)=(1,2^{\nu-2}p_{2}p_{3})$ or $(e,f)=(2^{\nu-2},p_{2}p_{3})$ or 
$(e,f)=(2^{\nu-2}p_{2},p_{3})$ or $(e,f)=(2^{\nu-2}p_{3},p_{2})$. 
So, $(e,f)\in\Omega_{4}$. Hence, $\Omega_{4}=\Omega'_{4}$.
\end{description} 
\end{proof}
We then deduce in conjunction with \autoref{convolutionSum-theor-sp} the following: 
\begin{corollary} \label{representations-corol_a_b}
Let $n\in\mathbb{N}$, $\alpha\beta\in\mathbb{N}^{*}$ with $\alpha\beta\equiv 0\pmod{4}$ and $\Omega_{4}$ be determined as above. Then for each $(a,b),(a_{1},b_{1})\in\Omega_{4}$ we have  
	\begin{align*}
	N_{(a,b)}^{4,8}(n)\, =\,N_{(b,a)}^{8,4}(n)\, =\, & 
	8\,\sigma(\frac{n}{a}) - 32\,\sigma(\frac{n}{4a}) + 16\,\sigma_{3}(\frac{n}{b}) - 32\,\sigma_{3}(\frac{n}{2b}) 
	+ 256\sigma_{3}(\frac{n}{4b})   \\ & 
        + 128\, W_{(a,b)}^{1,3}(n) - 256\, W_{(a,2b)}^{1,3}(n) + 2048\, W_{(a,4b)}^{1,3}(n)     \\ & 
       - 512\, W_{(4a,b)}^{1,3}(n)  
	+ 1024\, W_{(2a,b)}^{1,3}(\frac{n}{2})     
	- 8192\, W_{(a,b)}^{1,3}(\frac{n}{4}) 
\end{align*}
and
\begin{align*}
	N_{(a_{1},b_{1})}^{8,4}(n)\, =\,N_{(b_{1},a_{1})}^{4,8}(n)\, =\, & 
	8\,\sigma(\frac{n}{b_{1}}) 
	- 32\,\sigma(\frac{n}{4b_{1}}) 
	+ 16\,\sigma_{3}(\frac{n}{a_{1}}) 
	- 32\,\sigma_{3}(\frac{n}{2a_{1}})   \\ & 
	+ 256\sigma_{3}(\frac{n}{4a_{1}}) 
	+ 128\, W_{(a_{1},b_{1})}^{3,1}(n) 
	- 512\, W_{(a_{1},4b_{1})}^{3,1}(n)     \\ & 
	- 256\, W_{(2a_{1},b_{1})}^{3,1}(n)  
	 + 1024\, W_{(a_{1},2b_{1})}^{3,1}(\frac{n}{2})  
	+ 2048\, W_{(4a_{1},b_{1})}^{3,1}(n)     \\ & 
	- 8192\, W_{(a_{1},b_{1})}^{3,1}(\frac{n}{4}) 
\end{align*}
\end{corollary}
From this corollary it follows that if $a=b=a_{1}=b_{1}$, then $N_{(a,a)}^{4,8}(n) = N_{(a,a)}^{8,4}(n)$, especially 
$N_{(1,1)}^{4,8}(n) = N_{(1,1)}^{8,4}(n)$.							

								
\subsection{Representations of a positive Integer by the Quadratic Form (\hyperref[introduction-eq-3]{\ref*{introduction-eq-3}}) and (\hyperref[introduction-eq-4]{\ref*{introduction-eq-4}})
}\label{representations_c_d}
								
We now determine formulae for the number of representations of a positive 
integer by the octonary quadratic forms (\hyperref[introduction-eq-3]{\ref*{introduction-eq-3}}) and (\hyperref[introduction-eq-4]{\ref*{introduction-eq-4}}).
								
\subsubsection{Formulae for the Number of Representations by 
(\hyperref[introduction-eq-3]{\ref*{introduction-eq-3}}) and (\hyperref[introduction-eq-4]{\ref*{introduction-eq-4}})
								}\label{represent_c_d}
								
Let $n\in\mathbb{N}$ and let $s_{4}(n)$ and $s_{8}(n)$denote the number of representations
of $n$ by the quaternary quadratic form  $\underset{i=1}{\overset{2}{\sum}}\,(\,x_{2i-1}^{2}+ x_{2i-1}x_{2i} + x_{2i}^{2}\,)$, that is,   
\begin{multline*}
s_{4}(n)=\text{card}(\{(x_{1},x_{2},x_{3},x_{4})\in\mathbb{Z}^{4}~|~  
n = \,\sum_{i=1}^{2}\,(\,x_{2i-1}^{2}+ x_{2i-1}x_{2i} + x_{2i}^{2}\,)  
\})
\end{multline*}
and the quadratic form  $\underset{i=1}{\overset{4}{\sum}}\,(\,x_{2i-1}^{2}+ x_{2i-1}x_{2i} + x_{2i}^{2}\,)$, that means,   
\begin{multline*}
s_{8}(n)=\text{card}(\{(x_{1},x_{2},\ldots,x_{7},x_{8})\in\mathbb{Z}^{8}~|~  
n = \,\sum_{i=1}^{4}\,(\,x_{2i-1}^{2}+ x_{2i-1}x_{2i} + x_{2i}^{2}\,)\, \}),
\end{multline*}
respectively.
It is obvious that $s_{4}(0) = s_{8}(0) =1$. \jgH\ et al.\ \cite{huardetal}, \gaL\ \cite{lomadze} and \ksW\ \cite[Thrm 17.3, p.\ 225]{williams2011} have 
proved that for all $n\in\mathbb{N}^{*}$ 
\begin{equation}
s_{4}(n) = 12\,\sigma(n) - 36\,\sigma(\frac{n}{3}). \label{representations-eqn-c_d-1}
\end{equation}
 \gaL\ \cite{lomadze} has proved that for all $n\in\mathbb{N}^{*}$ 
\begin{equation}
s_{8}(n) = 24\,\sigma_{3}(n) + 216\,\sigma_{3}(\frac{n}{3}). \label{representations-eqn-c_d-2}
\end{equation}
Now, let the number of representations of $n$ by the quadratic form 
(\hyperref[introduction-eq-3]{\ref*{introduction-eq-3}}) be 
\begin{multline*}
R_{(c,d)}^{4,8}(n) =\text{card}
(\{(x_{1},x_{2},\ldots,x_{11},x_{12})\in\mathbb{Z}^{12}~|~
n = c\,\sum_{i=1}^{2}\,(\,x_{2i-1}^{2}+ x_{2i-1}x_{2i} + x_{2i}^{2}\,)    \\  
+ d\,\sum_{i=3}^{6}\,(\,x_{2i-1}^{2}+ x_{2i-1}x_{2i} + x_{2i}^{2}\,)\, \}),
\end{multline*}
and that of the quadratic form 
(\hyperref[introduction-eq-4]{\ref*{introduction-eq-4}}) be 
\begin{multline*}
R_{(c_{1},d_{1})}^{8,4}(n) =\text{card}
(\{(x_{1},x_{2},\ldots,x_{11},x_{12})\in\mathbb{Z}^{12}~|~
n = c_{1}\,\sum_{i=1}^{4}\,(\,x_{2i-1}^{2}+ x_{2i-1}x_{2i} + x_{2i}^{2}\,)   \\
  + d_{1}\,\sum_{i=5}^{6}\,(\,x_{2i-1}^{2}+ x_{2i-1}x_{2i} + x_{2i}^{2}\,)\, \}),
\end{multline*}							
where $c,d,c_{1},d_{1}\in\mathbb{N}^{*}$.
								
From these definitions suppose that 
$c,d,c_{1},d_{1}\in\mathbb{N}^{*}$ are such that $\gcd(c,d)= e>1$ and $\gcd(c{1},d_{1})= e_{1}>1$for some 
$e,e_{1}\in\mathbb{N}^{*}$. Then $R_{(c,d)}^{4,8}(n)= R_{(\frac{c}{e},\frac{d}{e})}^{4,8}(\frac{n}{e})$ and $R_{(c_{1},d_{1})}^{8,4}(n)= R_{(\frac{c_{1}}{e_{1}},\frac{d_{1}}{e_{1}})}^{8,4}(\frac{n}{e_{1}})$, respectively. Hence, one can simply assume that 
$c,d,c_{1},d_{1}\in\mathbb{N}^{*}$ are relatively prime.
								
We infer the following 
\begin{theorem} \label{representations-theor-c_d}
Let $n\in\mathbb{N}$ and $c,d,c_{1},d_{1}\in\mathbb{N}^{*}$ be relatively prime. Then  
\begin{align*}
R_{(c,d)}^{4,8}(n) =  & 
12\,\sigma(\frac{n}{c}) - 36\,\sigma(\frac{n}{3c}) + 24\,\sigma_{3}(\frac{n}{d}) +
216\,\sigma_{3}(\frac{n}{3d}) + 288\, W_{(c,d)}^{1,3}(n) \\ & 
+ 2592\, W_{(c,3d)}^{1,3}(n) 
- 864\,W_{(3c,d)}^{1,3}(n) -  7776\,W_{(c,d)}^{1,3}(\frac{n}{3}) 
\end{align*}
and 
\begin{align*}
R_{(c_{1},d_{1})}^{8,4}(n) =  & 
24\,\sigma_{3}(\frac{n}{c_{1}}) + 216\,\sigma_{3}(\frac{n}{3c_{1}}) + 12\,\sigma(\frac{n}{d_{1}}) -
36\,\sigma(\frac{n}{3d_{1}}) + 288\, W_{(c_{1},d_{1})}^{3,1}(n) \\ & 
- 864\, W_{(c_{1},3d_{1})}^{3,1}(n) 
+ 2592\,W_{(3c_{1},d_{1})}^{3,1}(n) - 7776\,W_{(c_{1},d_{1})}^{3,1}(\frac{n}{3}) 
\end{align*}
								\end{theorem}
\begin{proof} It suffices to prove the identity $R_{(c_{1},d_{1})}^{8,4}(n)$ since 
	the other can be proved similarly or by applying \autoref{convolutionSum-theor-sp}. 
	
	We have 
	\begin{equation*}
          \begin{split}
	R_{(c_{1},d_{1})}^{8,4}(n) & = \sum_{\substack{
			{(l,m)\in\mathbb{N}^{2}} \\ {c_{1}\,l+d_{1}\,m=n}
	}}s_{8}(l)s_{4}(m) \\
	& = s_{8}(\frac{n}{c_{1}})s_{4}(0) + s_{8}(0)s_{4}(\frac{n}{d_{1}}) 
	+ \sum_{\substack{
			{(l,m)\in\mathbb{N}^{*}\times\mathbb{N}^{*}} \\ {c_{1}\,l+d_{1}\,m=n}
	}}s_{8}(l)s_{4}(m). \\
          \end{split}
	\end{equation*}
	We make use of (\hyperref[representations-eqn-4-1]{\ref*{representations-eqn-4-1}}) 
	and (\hyperref[representations-eqn-4-2]{\ref*{representations-eqn-4-2}})
	to obtain 
	\begin{multline*}
	R_{(c_{1},d_{1})}^{8,4}(n)  = 24\,\sigma_{3}(\frac{n}{c_{1}}) + 216\,\sigma_{3}(\frac{n}{3c_{1}}) + 12\,\sigma(\frac{n}{d_{1}}) - 36\,\sigma(\frac{n}{3d_{1}})   \\
	+ \sum_{\substack{
			{(l,m)\in\mathbb{N}^{*}\times\mathbb{N}^{*}} \\ {c_{1}l+d_{1}m=n}
	}} (24\,\sigma_{3}(l) + 216\,\sigma_{3}(\frac{l}{3}) )(12\,\sigma(m) - 36\,\sigma(\frac{m}{3})). 
	\end{multline*}
	We know that 
	\begin{multline*}
	(24\,\sigma_{3}(l) + 216\,\sigma_{3}(\frac{l}{3}) )(12\,\sigma(m) - 36\,\sigma(\frac{m}{3})) = \\
	288\,\sigma_{3}(l)\sigma(m) - 864\,\sigma_{3}(l)\sigma(\frac{m}{3}) 
	+ 2592\,\sigma_{3}(\frac{l}{3})\sigma(m)  
- 7776\,\sigma(\frac{l}{3})\sigma_{3}(\frac{m}{3}) 
	\end{multline*}
	In the sequel of this proof, we assume that the evaluation of 
	\begin{equation*}
	W_{(c_{1},d_{1})}^{3,1}(n) = \sum_{\substack{
			{(l,m)\in\mathbb{N}^{*}\times\mathbb{N}^{*}} \\ {c_{1}l+d_{1}m=n}
	}}\sigma_{3}(l)\sigma(m),
	\end{equation*}
	$W_{(c_{1},3d_{1})}^{3,1}(n),\>W_{(3c_{1},d_{1})}^{3,1}(n)\>$
	and $W_{(3c_{1},3d_{1})}^{3,1}(n)$ are known. 
	
	Let $u,v\in\mathbb{N}^{*}$ be given and $f,g:\mathbb{N}\mapsto\mathbb{N}$ be  
	injective functions such that $f(n)=u\cdot n$ and 
	$g(n)=v\cdot n$ for each $n\in\mathbb{N}$.
	
	When we simultaneously apply the functions 
	$f$ and $g$ with $l$ and $m$ as argument, respectively, we derive  
	\begin{equation*}
	W_{(c_{1},d_{1})}^{3,1}(n) = \sum_{\substack{
			{(l,m)\in\mathbb{N}^{*}\times\mathbb{N}^{*}} \\ {c_{1}l+d_{1}m=n}
	}}\sigma_{3}(\frac{l}{u})\sigma(\frac{m}{v}) 
	= \sum_{\substack{
			{(l,m)\in\mathbb{N}^{*}\times\mathbb{N}^{*}} \\ {uc_{1}\,l+vd_{1}\,m=n}
	}}\sigma_{3}(l)\sigma(m) = W_{(uc_{1},vd_{1})}^{1,3}(n)
	\end{equation*}
	We set $(u,v) = (1,1), (1,3), (3,1),(3,3)$ accordingly and then  
	finally put all these evaluations together to obtain the stated result 
	for $R_{(c_{1},d_{1})}^{8,4}(n)$. 
\end{proof}
																	
	From this proof, we note that a formula for the number 
	of representations of a positive integer $n$  by the  quadratic form 
	(\hyperref[introduction-eq-3]{\ref*{introduction-eq-3}}) and (\hyperref[introduction-eq-4]{\ref*{introduction-eq-4}}) depends on the 
		evaluated convolution Sums for some given levels $c_{1}d_{1}$ and $3c_{1}d_{1}$ with $c_{1},d_{1}\in\mathbb{N}^{*}$ relatively prime. 
																				
	As a consequence, we do consider only the levels $\alpha\beta$   
	 which are divisible by $3$; that is $\alpha\beta\equiv 0\pmod{3}$. 
																		
\subsubsection{Determination of all relevant $\mathbf{(c,d)\in\mathbb{N}^{*}\times\mathbb{N}^{*}}$ for $R_{(c,d)}(n)$ for a given level  $\alpha\beta\in\mathbb{N}^{*}$} 
\label{determine_c_d}
	The following method determine all pairs $(c,d)\in\mathbb{N}^{*}\times\mathbb{N}^{*}$ necessary for
the determination of $R_{(c,d)}(n)$ for a given $\alpha\beta\in\mathbb{N}^{*}$ 
belonging to the above class. The following method is quasi similar to the one 
used in \hyperref[determine_a_b]{Subsection \ref*{determine_a_b}}. 
																		
Let $\Delta=\frac{\alpha\beta}{3}=\frac{2^{\nu}\mho}{3}$. Let 
$P_{3}=\{p_{0}=2^{\nu}\}\cup\underset{j>2}{\bigcup}\{ p_{j}\,\mid\,p_{j}\text{ is a prime divisor of } \mho\,\}$. 
Let $\EuScript{P}(P_{3})$ be the power set of $P_{3}$. Then for each 
$Q\in\EuScript{P}(P_{3})$ 
we define $\mu(Q)=\underset{p\in Q}{\prod}p$. We set $\mu(Q)=1$ if	$Q$ is an empty set.
Let now $\Omega_{3}$ be defined in a similar way as $\Omega_{4}$ in 
\hyperref[determine_a_b]{Subsection \ref*{determine_a_b}}; 
	however, with $\Delta$ instead of $\Lambda$, i.e., 
\begin{multline*}
\Omega_{3}=\{(\mu(Q_{1}),\mu(Q_{2}))~|~\text{ there exist } Q_{1},Q_{2}\in\EuScript{P}(P_{3}) 
\text{ such that }  \\ 
\gcd{(\mu(Q_{1}),\mu(Q_{2}))}=1\,\text{ and } 
\mu(Q_{1})\,\mu(Q_{2})=\Delta\,\}.
\end{multline*}
Note that $\Omega_{3}\neq\emptyset$ since $(1,\Delta)\in\Omega_{3}$.
																		
As an example, suppose again that $\alpha\beta=2^{3}\cdot 3\cdot 5$. Then 
$\Delta=2^{3}\cdot 5$, $P_{3}=\{2^{3},5\}$ and $\Omega_{3}=\{(1,40),(5,8)\}$.
	\begin{proposition} \label{representation-prop-2}
		Suppose that the level $\alpha\beta\in\mathbb{N}^{*}$ and $\alpha\beta\equiv 0\pmod{3}$. Suppose in addition that $\Omega_{3}$ 
		be defined as above. Then for all $n\in\mathbb{N}$ the set $\Omega_{3}$ 
		contains all pairs $(c,d)\in\mathbb{N}^{*}\times\mathbb{N}^{*}$ such that $R_{(c,d)}(n)$ can be obtained by applying $W_{(\alpha,\beta)}(n)$ and some other evaluated convolution sums.
  \end{proposition}
\begin{proof} 
Similar to the proof of 
\hyperref[representation-prop-1]{Proposition \ref*{representation-prop-1}}.
\end{proof}
																		
We then infer in collaboration with \autoref{convolutionSum-theor-sp} the following: 
\begin{corollary} \label{representations-corol_c_d}
	Let $n\in\mathbb{N}$, $\alpha\beta\in\mathbb{N}^{*}$ with $\alpha\beta\equiv 0\pmod{3}$ and $\Omega_{3}$ be determined as above. Then for each $(c,d),(c_{1},d_{1})\in\Omega_{3}$ we obtain  
	\begin{align*}
	R_{(c,d)}^{4,8}(n)\, =\,R_{(d,c)}^{8,4}(n)\, =\, & 
	12\,\sigma(\frac{n}{c}) - 36\,\sigma(\frac{n}{3c}) + 24\,\sigma_{3}(\frac{n}{d}) +
	216\,\sigma_{3}(\frac{n}{3d}) + 288\, W_{(c,d)}^{1,3}(n) \\ & 
	+ 2592\, W_{(c,3d)}^{1,3}(n) 
	- 864\,W_{(3c,d)}^{1,3}(n) -  7776\,W_{(c,d)}^{1,3}(\frac{n}{3}) 
\end{align*}
and 
\begin{align*}
	R_{(c_{1},d_{1})}^{8,4}(n)\, =\,R_{(d_{1},c_{1})}^{4,8}(n)\, =\, & 
	24\,\sigma_{3}(\frac{n}{c_{1}}) 
	+ 216\,\sigma_{3}(\frac{n}{3c_{1}}) 
	+ 12\,\sigma(\frac{n}{d_{1}}) 
	- 36\,\sigma(\frac{n}{3d_{1}}) \\ & 
	+ 288\, W_{(c_{1},d_{1})}^{3,1}(n) 
	- 864\, W_{(c_{1},3d_{1})}^{3,1}(n) 
	+ 2592\,W_{(3c_{1},d_{1})}^{3,1}(n) \\ &  
	- 7776\,W_{(c_{1},d_{1})}^{3,1}(\frac{n}{3}) 
\end{align*}
\end{corollary}
It immediately follows that if $c=d=c_{1}=d_{1}$, then $R_{(c,c)}^{4,8}(n) = R_{(c,c)}^{8,4}(n)$, especially 
$R_{(1,1)}^{4,8}(n) = R_{(1,1)}^{8,4}(n)$.


\section{Evaluation of the Convolution Sums for some Levels in $\N$}
\label{convolution_33_40_56}

In this section, we give explicit formulae for the convolution sums 
\begin{itemize}
\item $W_{(\alpha,\beta)}^{1,3}(n)$  when 
$\alpha\beta=7,8,14,15,20,21$; and 
\item $W_{(\alpha,\beta)}^{3,1}(n)$  when 
$\alpha\beta=3,4,6,7,8,12,14,15,20,21$; 
\end{itemize}

When we apply \autoref{basis-cusp-eqn}, we conclude that for example 
\begin{align}
\M_{6}(\Gamma_{0}(3))\subset\M_{6}(\Gamma_{0}(6))\subset \M_{6}(\Gamma_{0}(12))  \label{eqn-3-12} \\ 
\M_{6}(\Gamma_{0}(4))\subset\M_{6}(\Gamma_{0}(8))\subset\M_{6}(\Gamma_{0}(16))\subset
\M_{6}(\Gamma_{0}(32)).  \label{eqn-4-32} 
\end{align}
This implies the same inclusion relation for the bases, the space of Eisenstein 
forms of weight $6$ and the spaces of cusp forms of weight $6$.

\subsection{Bases of $\E_{6}(\boldsymbol{\Upgamma}_{0}(\boldsymbol{\upalpha\upbeta}))$ and
	$\S_{6}(\boldsymbol{\Upgamma}_{0}(\boldsymbol{\upalpha\upbeta}))$ for
	these values of $\upalpha\upbeta$
}
\label{convolution_33_40_56-gen}

We apply the dimension formulae in \tM 's book \cite[Thrm 2.5.2,~p.~60]{miyake1989} or 
\cite[Prop.~6.1, p.~91]{wstein} and 
(\hyperref[dimension-Eisenstein-special]{\ref*{dimension-Eisenstein-special}})
(\hyperref[dimension-Eisenstein-special]{\ref*{dimension-Eisenstein-special}})
to deduce that 
\begin{alignat*}{3}   
\text{dim}(\E_{6}(\Gamma_{0}(\alpha\beta)))=d(\alpha\beta), \quad
\text{dim}(\S_{6}(\Gamma_{0}(\alpha\beta)))=1 \quad \text{when $2<\alpha\beta<6$ holds,} \\
\text{dim}(\S_{6}(\Gamma_{0}(6)))=3, \quad
\text{dim}(\S_{6}(\Gamma_{0}(7)))=3, \quad
\text{dim}(\S_{6}(\Gamma_{0}(8)))=3, \\
\text{dim}(\S_{6}(\Gamma_{0}(12)))=7, \quad
\text{dim}(\S_{6}(\Gamma_{0}(14)))=8, \quad
\text{dim}(\S_{6}(\Gamma_{0}(15)))=8, \\
\text{dim}(\S_{6}(\Gamma_{0}(20)))=12,  \quad
\text{dim}(\S_{6}(\Gamma_{0}(21)))=12.
\end{alignat*}

We apply \autoref{ligozat_theorem} as mentioned in 
\hyperref[convolution_alpha_beta-bases]{Subsection \ref*{convolution_alpha_beta-bases}} 
to determine as many elements of
\begin{alignat*}{5}
  \S_{6}(\Gamma_{0}(8)), \quad \S_{6}(\Gamma_{0}(14)), \quad
  \S_{6}(\Gamma_{0}(15)), \quad 
  \S_{6}(\Gamma_{0}(20)) \quad \text{and} \quad \S_{6}(\Gamma_{0}(21))  
\end{alignat*}
as possible. Then we apply \hyperref[basis-remark]{Remark \ref*{basis-remark}} 
\textbf{(r2)} when selecting basis elements of a given space of cusp forms as 
stated in the proof of \autoref{basisCusp_a_b} (b).  
\begin{corollary} \label{basisCusp_33_40_56}
	\begin{enumerate}
		\item[\textbf{(a)}] Let $\kappa=12,14,15,20,21$ and $\delta\in D(\kappa)$. Then the set \[\EuScript{B}_{E,\kappa}=\{\,M(q^{t})\,\mid ~ t|\kappa\,\}\]
		is a basis of $\E_{6}(\Gamma_{0}(\kappa))$. 		
		\item[\textbf{(b)}] Let $i\in\mathbb{N}^{*}$ satisfies $1\leq i\leq m_{\kappa}$, where $m_{\kappa}=7,8,12$. 
		
		Let $\delta_{1}\in D(12)$ and 
		$(r(i,\delta_{1}))_{i,\delta_{1}}$ be the 
		\autoref{convolutionSums-12-table} of the powers of $\eta(\delta_{1} z)$. 
		
		Let $\delta_{2}\in D(14)$ and 
		$(r(j,\delta_{2}))_{j,\delta_{2}}$ be the 
		\autoref{convolutionSums-14-table} of the powers of $\eta(\delta_{2} z)$. 
		
		Let $\delta_{3}\in D(15)$ and 
		$(r(k,\delta_{3}))_{k,\delta_{3}}$ be the 
		\autoref{convolutionSums-15-table} of the powers of $\eta(\delta_{3} z)$. 
		Let $\delta_{4}\in D(20)$ and 
		$(r(j,\delta_{4}))_{j,\delta_{4}}$ be the 
		\autoref{convolutionSums-20-table} of the powers of $\eta(\delta_{4} z)$. 
		
		Let $\delta_{5}\in D(21)$ and 
		$(r(k,\delta_{5}))_{k,\delta_{5}}$ be the 
		\autoref{convolutionSums-21-table} of the powers of $\eta(\delta_{5} z)$. 
		
		Let furthermore  
		\begin{gather*}
		\EuFrak{B}_{\kappa,i}(q)=\underset{\delta\,|\,\kappa}{\prod}\eta^{r(i,\delta)}(\delta\,z)  
		\end{gather*} 
		be selected elements of 
		$\S_{6}(\Gamma_{0}(\kappa))$.
		
		Then the set 
		\begin{gather*}
		\EuScript{B}_{S,\kappa}=\{\,\EuFrak{B}_{\kappa,i}(q)\,\mid ~ 1\leq i\leq m_{\kappa}\,\}
		\end{gather*}
		is a basis 
		of $\S_{6}(\Gamma_{0}(\kappa))$. 
		\item[\textbf{(c)}] The set
\[\EuScript{B}_{M,\kappa}=\EuScript{B}_{E,\kappa}\cup\EuScript{B}_{S,\kappa}\] 
		constitutes a basis of $\M_{6}(\Gamma_{0}(\kappa))$ for $\kappa=12,14,15,20,21$.
	\end{enumerate}
\end{corollary}
By \hyperref[basis-remark]{Remark \ref*{basis-remark}} \textbf{(r1)}, 
$\EuFrak{B}_{\kappa,i}(q)$
can be expressed in the form 
$\underset{n=1}{\overset{\infty}{\sum}}\EuFrak{b}_{\kappa,i}(n)q^{n}$,
where $\kappa=12,14,15,20,21$. 

\begin{proof} 
	It follows immediately from \autoref{basisCusp_a_b}. 
\end{proof}

\subsection{Evaluation of $\mathbf{W_{(\boldsymbol{\upalpha,\upbeta})}^{2i-1,2j-1}(n)}$ when $\boldsymbol{\upalpha\upbeta}\mathbf{=3,4,6,7,8,12,14,15,20,21}$} 
\label{convolSum-w_33_40_56}

We consider in the following two corollaries to discuss the case 
$(\,\alpha\,E_{2}(q^{\alpha}) - E_{2}(q)\,)\, E_{4}(q^{\beta})$ and the case 
$E_{4}(q^{\beta})\,(\,E_{2}(q) - \alpha\,E_{2}(q^{\alpha})\,)$ for 
some values of $\alpha\beta$ since the other cases can be handled similarly.

The first corollary deals with the case $(\,\alpha\,E_{2}(q^{\alpha}) - E_{2}(q)\,)\, E_{4}(q^{\beta})$.  It is sufficient to consider only the special cases $1\neq\alpha = \beta$ and $\alpha> 1$ whenever $\beta=1$.
\begin{corollary} \label{lema-w_7-21}
	We have  
	\begin{multline}  
	(\,7\, E_{2}(q^{7}) - E_{2}(q)\,)\, E_{4}(q^{7}) 
	=  6 + \sum_{n=1}^{\infty}\biggl(\, 
- \frac{1}{2451}\,\sigma_{5}(n) 
+ \frac{14707}{2451}\,\sigma_{5}(\frac{n}{7})    \\
+ \frac{19440}{817}\,\EuFrak{b}_{21,1}(n) 
+ \frac{5760}{19}\,\EuFrak{b}_{21,2}(n) 
+ \frac{843120}{817}\,\EuFrak{b}_{21,3}(n) 
	\, \biggr)\,q^{n}, 
	\label{convolSum-eqn-1_7-1_3}
	\end{multline}
\begin{multline}  
(\,7\, E_{2}(q^{7}) - E_{2}(q)\,)\, E_{4}(q) 
=  6 + \sum_{n=1}^{\infty}\biggl(\, 
- \frac{2101}{2451}\,\sigma_{5}(n)
+ \frac{16807}{2451}\,\sigma_{5}(\frac{n}{7})     \\
+ \frac{843120}{817}\,\EuFrak{b}_{21,1}(n) 
+ \frac{282240}{19}\,\EuFrak{b}_{21,2}(n) 
+ \frac{46675440}{817}\,\EuFrak{b}_{21,3}(n)
\, \biggr)\,q^{n}, 
\label{convolSum-eqn-7_1-1_3}
\end{multline}
	\begin{multline}
	(\,8\, E_{2}(q^{8}) -  E_{2}(q)\,)\,E_{4}(q^{8}) 
	= 7 + \sum_{n=1}^{\infty}\biggl(\, 14
- \frac{1}{5376}\,\sigma_{5}(n)
- \frac{5}{1792}\,\sigma_{5}(\frac{n}{2})    \\
+ \frac{7}{48}\,\sigma_{5}(\frac{n}{4}) 
+ \frac{20}{7}\,\sigma_{5}(\frac{n}{8}) 
+ \frac{765}{32}\,\EuFrak{b}_{32,1}(n) 
+ \frac{135}{2}\,\EuFrak{b}_{32,2}(n) 
+ 360\,\EuFrak{b}_{32,3}(n) 
	\,\biggr)\,q^{n},
	\label{convolSum-eqn-1_8-1_3}
	\end{multline}
	\begin{multline}
(\,8\, E_{2}(q^{8}) - E_{2}(q)\,)\,E_{4}(q) 
=  7 + \sum_{n=1}^{\infty}\biggl(\, 
- \frac{37}{42}\,\sigma_{5}(n)
+ \frac{5}{14}\,\sigma_{5}(\frac{n}{2})    \\
+ \frac{10}{7}\,\sigma_{5}(\frac{n}{4}) 
+ \frac{128}{21}\,\sigma_{5}(\frac{n}{8}) 
+ 1260\,\EuFrak{b}_{32,1}(n)
+ 6480\,\EuFrak{b}_{32,2}(n)
+ 23040\,\EuFrak{b}_{32,3}(n)
\,\biggr)\,q^{n},
\label{convolSum-eqn-8_1-1_3}
\end{multline}
\begin{multline}
(\,14\, E_{2}(q^{14}) -  E_{2}(q)\,)\,E_{4}(q^{14}) 
=  13 + \sum_{n=1}^{\infty}\biggl(\, 
 - \frac{1}{51471}\,\sigma_{5}(n)
- \frac{20}{51471}\,\sigma_{5}(\frac{n}{2})    \\
+ \frac{2101}{7353}\,\sigma_{5}(\frac{n}{7})
+ \frac{42020}{7353}\,\sigma_{5}(\frac{n}{14})
+ \frac{19600}{817}\,\EuFrak{b}_{14,1}(n)
+ \frac{254400}{817}\,\EuFrak{b}_{14,2}(n)     \\
+ \frac{916880}{817}\,\EuFrak{b}_{14,3}(n)
+ \frac{255680}{817}\,\EuFrak{b}_{14,4}(n) 
+ \frac{523520}{817}\,\EuFrak{b}_{14,5}(n) 
+ \frac{4369920}{817}\,\EuFrak{b}_{14,6}(n)     \\
- \frac{1437600}{817}\,\EuFrak{b}_{14,7}(n) 
- \frac{1280}{817}\,\EuFrak{b}_{14,8}(n) 
\,\biggr)\,q^{n},
\label{convolSum-eqn-1_14-1_3}
\end{multline}
\begin{multline}
(\,14\,E_{2}(q^{14}) -  E_{2}(q)\,)\,E_{4}(q) 
= 13 + \sum_{n=1}^{\infty}\biggl(\, 
- \frac{6853}{7353}\,\sigma_{5}(n)
+ \frac{1600}{7353}\,\sigma_{5}(\frac{n}{2})    \\
+ \frac{24010}{7353}\,\sigma_{5}(\frac{n}{7})  
+ \frac{76832}{7353}\,\sigma_{5}(\frac{n}{14}) 
+ \frac{2184880}{817}\,\EuFrak{b}_{14,1}(n)
+ \frac{36980160}{817}\,\EuFrak{b}_{14,2}(n)     \\
+ \frac{179693360}{817}\,\EuFrak{b}_{14,3}(n)
+ \frac{153292160}{817}\,\EuFrak{b}_{14,4}(n)
+ \frac{102126080}{817}\,\EuFrak{b}_{14,5}(n)     \\
+ \frac{1505817600}{817}\,\EuFrak{b}_{14,6}(n) 
- \frac{282011520}{817}\,\EuFrak{b}_{14,7}(n)
+ \frac{716800}{817}\,\EuFrak{b}_{14,8}(n)
\,\biggr)\,q^{n},
\label{convolSum-eqn-14_1-1_3}
\end{multline}
\begin{multline}
(\,15\, E_{2}(q^{15}) -  E_{2}(q)\,)\,E_{4}(q^{15}) 
=  14 + \sum_{n=1}^{\infty}\biggl(\,
\frac{274}{8934471}\,\sigma_{5}(n)
+ \frac{5457}{992719}\,\sigma_{5}(\frac{n}{3})    \\
- \frac{8090239}{8934471}\,\sigma_{5}(\frac{n}{5})
+ \frac{14791494}{992719}\,\sigma_{5}(\frac{n}{15})
+ \frac{3405800}{141817}\,\EuFrak{b}_{15,1}(n)    \\
+ \frac{1158952}{10909}\,\EuFrak{b}_{15,2}(n)
- \frac{35264}{10909}\,\EuFrak{b}_{15,3}(14n)
+ \frac{2061176}{10909}\,\EuFrak{b}_{15,4}(n)    \\
+ \frac{177635528}{141817}\,\EuFrak{b}_{15,5}(n)
+ \frac{2367936}{10909}\,\EuFrak{b}_{15,6}(n)
+ \frac{34294464}{10909}\,\EuFrak{b}_{15,7}(n)    \\
+ \frac{1203968}{10909}\,\EuFrak{b}_{15,8}(n),
\,\biggr)\,q^{n},
\label{convolSum-eqn-1_15-1_3}
\end{multline}
\begin{multline}
(\,15\,E_{2}(q^{15}) -  E_{2}(q)\,)\,E_{4}(q) 
= 14 + \sum_{n=1}^{\infty}\biggl(\, 
- \frac{917834}{992719}\,\sigma_{5}(n)
- \frac{1460295}{992719}\,\sigma_{5}(\frac{n}{3})    \\
- \frac{104862815}{992719}\,\sigma_{5}(\frac{n}{5}) 
+ \frac{121139010}{992719}\,\sigma_{5}(\frac{n}{15}) 
+ \frac{413824680}{141817}\,\EuFrak{b}_{15,1}(n)    \\
+ \frac{756726120}{10909}\,\EuFrak{b}_{15,2}(n) 
- \frac{93504960}{10909}\,\EuFrak{b}_{15,3}(n) 
+ \frac{263947320}{10909}\,\EuFrak{b}_{15,4}(n)    \\
+ \frac{90853616520}{141817}\,\EuFrak{b}_{15,5}(n) 
+ \frac{283798080}{10909}\,\EuFrak{b}_{15,6}(n) 
+ \frac{3149876160}{10909}\,\EuFrak{b}_{15,7}(n)    \\
- \frac{339972480}{10909}\,\EuFrak{b}_{15,8}(n)
\,\biggr)\,q^{n},
\label{convolSum-eqn-15_1-1_3}
\end{multline}
\begin{multline}
	(\,20\, E_{2}(q^{20}) -  E_{2}(q)\,)\,E_{4}(q^{20}) =
= 19 + \sum_{n=1}^{\infty}\biggl(\, 
\frac{1567}{117445608}\,\sigma_{5}(n)    \\
- \frac{110791}{117445608}\,\sigma_{5}(\frac{n}{2})
- \frac{10004}{233027}\,\sigma_{5}(\frac{n}{4})  
- \frac{702215}{234891216}\,\sigma_{5}(\frac{n}{5})      \\
- \frac{10264633}{234891216}\,\sigma_{5}(\frac{n}{10})  
+ \frac{93420884}{4893567}\,\sigma_{5}(\frac{n}{20})     
+ \frac{5594215}{233027}\,\EuFrak{b}_{20,1}(n)    \\
+ \frac{16718864}{233027}\,\EuFrak{b}_{20,2}(n)
+ \frac{89883520}{233027}\,\EuFrak{b}_{20,3}(n)
+ \frac{32106736}{233027}\,\EuFrak{b}_{20,4}(n)     \\
- \frac{527968775}{466054}\,\EuFrak{b}_{20,5}(n)
+ \frac{86133984}{233027}\,\EuFrak{b}_{20,6}(n)
+ \frac{1416189312}{233027}\,\EuFrak{b}_{20,7}(n)    \\  
+ \frac{4970832}{7517}\,\EuFrak{b}_{20,8}(n)   
+ \frac{976297728}{233027}\,\EuFrak{b}_{20,9}(n)
+ \frac{618500736}{233027}\,\EuFrak{b}_{20,10}(n)    \\
- \frac{64669696}{7517}\,\EuFrak{b}_{20,11}(n)
- \frac{4227164544}{233027}\,\EuFrak{b}_{20,12}(n)
\,\biggr)\,q^{n},
 \label{convolSum-eqn-1_20-1_3}
	\end{multline}
	\begin{multline}
	(\,20\,E_{2}(q^{20}) -  E_{2}(q)\,)\,E_{4}(q) =
= 19 + \sum_{n=1}^{\infty}\biggl(\, 
- \frac{53619409}{58722804}\,\sigma_{5}(n)     \\
- \frac{188714723}{58722804}\,\sigma_{5}(\frac{n}{2}) 
+ \frac{2103232}{1631189}\,\sigma_{5}(\frac{n}{4})  
+ \frac{9257815}{8388972}\,\sigma_{5}(\frac{n}{5})      \\
+ \frac{398439023}{58722804}\,\sigma_{5}(\frac{n}{10})  
+ \frac{68258944}{4893567}\,\sigma_{5}(\frac{n}{20})   
+ \frac{960956950}{233027}\,\EuFrak{b}_{20,1}(n)    \\
+ \frac{7006131104}{233027}\,\EuFrak{b}_{20,2}(n)
+ \frac{31247296000}{233027}\,\EuFrak{b}_{20,3}(n)
+ \frac{31145902336}{233027}\,\EuFrak{b}_{20,4}(n)     \\
+ \frac{15378042050}{233027}\,\EuFrak{b}_{20,5}(n)
+ \frac{32291496384}{233027}\,\EuFrak{b}_{20,6}(n)
+ \frac{542908645632}{233027}\,\EuFrak{b}_{20,7}(n)    \\
+ \frac{309602112}{7517}\,\EuFrak{b}_{20,8}(n)  
+ \frac{653138738688}{233027}\,\EuFrak{b}_{20,9}(n)
- \frac{678376410624}{233027}\,\EuFrak{b}_{20,10}(n)    \\ 
- \frac{13917769216}{7517}\,\EuFrak{b}_{20,11}(n)
- \frac{1138525192704}{233027}\,\EuFrak{b}_{20,12}(n)
\,\biggr)\,q^{n},
\label{convolSum-eqn-20_1-1_3}
\end{multline}
\begin{multline}
(\,21\, E_{2}(q^{21}) -  E_{2}(q)\,)\, E_{4}(q^{21}) =
20 + \sum_{n=1}^{\infty}\biggl(\, 
- \frac{1}{223041}\,\sigma_{5}(n)14
- \frac{350}{31863}\,\sigma_{5}(\frac{n}{7})    \\
+  \frac{1821}{91}\,\sigma_{5}(\frac{n}{21})    
+  \frac{254880}{10621}\,\EuFrak{b}_{21,1}(n)    
+  \frac{77040}{247}\,\EuFrak{b}_{21,2}(n)        \\
+  \frac{11969280}{10621}\,\EuFrak{b}_{21,3}(n)    
+ \frac{4320}{13}\,\EuFrak{b}_{21,4}(n)    
+  \frac{8640}{13}\,\EuFrak{b}_{21,5}(n)        \\
+  \frac{38880}{13}\,\EuFrak{b}_{21,6}(n)    
+  \frac{13200}{13}\,\EuFrak{b}_{21,7}(n)    
+  \frac{11520}{13}\,\EuFrak{b}_{21,8}(n)        \\
+  \frac{132480}{13}\,\EuFrak{b}_{21,9}(n)    
- \frac{50400}{13}\,\EuFrak{b}_{21,10}(n)    
- \frac{89280}{13}\,\EuFrak{b}_{21,11}(n)     
+ \frac{4320}{13}\,\EuFrak{b}_{21,12}(n)    
\,\biggr)\,q^{n},
 \label{convolSum-eqn-1_21-1-3}
\end{multline}
\begin{multline}
(\,21\, E_{2}(q^{21}) -  E_{2}(q)\,)\,E_{4}(q) =
20 + \sum_{n=1}^{\infty}\biggl(\, 
- \frac{10121}{10621}\,\sigma_{5}(n)
+ \frac{24010}{10621}\,\sigma_{5}(\frac{n}{7})  \\
+ \frac{243}{13}\,\sigma_{5}(\frac{n}{21}) 
+ \frac{46134720}{10621}\,\EuFrak{b}_{21,1}(n)
+ \frac{18925200}{247}\,\EuFrak{b}_{21,2}14(n)     \\
+ \frac{4185951840}{10621}\,\EuFrak{b}_{21,3}(n)
+ \frac{5352480}{13}\,\EuFrak{b}_{21,4}(n)
+ \frac{4415040}{13}\,\EuFrak{b}_{21,5}(n14)   \\
+ \frac{19867680}{13}\,\EuFrak{b}_{21,6}(n)
+ \frac{25900560}{13}\,\EuFrak{b}_{21,7}(n)
- \frac{1451520}{13}\,\EuFrak{b}_{21,8}(n)   \\
+ \frac{82373760}{13}\,\EuFrak{b}_{21,9}(n)
- \frac{18416160}{13}\,\EuFrak{b}_{21,10}(n)
- \frac{9979200}{13}\,\EuFrak{b}_{21,11}(n)   \\
- \frac{4082400}{13}\,\EuFrak{b}_{21,12}(n)
\,\biggr)\,q^{n},
\label{convolSum-eqn-21_1-1_3}
\end{multline}
\begin{multline}
(\,5\, E_{2}(q^{5}) -  E_{2}(q)\,)\,E_{4}(q^{3}) =
= 4 + \sum_{n=1}^{\infty}\biggl(\, 
- \frac{71366}{8934471}\,\sigma_{5}(n)
- \frac{812655}{992719}\,\sigma_{5}(\frac{n}{3})   \\
- \frac{107126590}{8934471}\,\sigma_{5}(\frac{n}{5}) 
+ \frac{16694415}{992719}\,\sigma_{5}(\frac{n}{15})
+ \frac{2832680}{141817}\,\EuFrak{b}_{21,1}(n)   \\
+ \frac{5125480}{10909}\,\EuFrak{b}_{21,2}(n)   
- \frac{10139840}{10909}\,\EuFrak{b}_{21,3}(n)
- \frac{17525320}{10909}\,\EuFrak{b}_{21,4}(n)   \\
- \frac{1113650920}{141817}\,\EuFrak{b}_{21,5}(n)  
+ \frac{31878720}{10909}\,\EuFrak{b}_{21,6}(n)
+ \frac{740168640}{10909}\,\EuFrak{b}_{21,7}(n)   \\
- \frac{4481920}{10909}\,\EuFrak{b}_{21,8}(n)
\,\biggr)\,q^{n},
	\label{convolSum-eqn-5_3-1_3}
\end{multline}
\begin{multline}
(\,5\, E_{2}(q^{5}) -  E_{2}(q)\,)\,E_{4}(q^{4}) =
= 4 + \sum_{n=1}^{\infty}\biggl(\, 
 - \frac{292153}{117445608}\,\sigma_{5}(n)
- \frac{4512719}{117445608}\,\sigma_{5}(\frac{n}{2})    \\
- \frac{10268}{699081}\,\sigma_{5}(\frac{n}{4}) 
+ \frac{1690315}{117445608}\,\sigma_{5}(\frac{n}{5})
+ \frac{25485149}{117445608}\,\sigma_{5}(\frac{n}{10})    \\
+ \frac{6238012}{1631189}\,\sigma_{5}(\frac{n}{20}) 
+ \frac{5300495}{233027}\,\EuFrak{b}_{20,1}(n) 
+ \frac{2624176}{233027}\,\EuFrak{b}_{20,2}(n)     \\
+ \frac{14691200}{233027}\,\EuFrak{b}_{20,3}(n) 
- \frac{196596016}{233027}\,\EuFrak{b}_{20,4}(n) 
+ \frac{150021475}{233027}\,\EuFrak{b}_{20,5}(n)     \\
+ \frac{645547296}{233027}\,\EuFrak{b}_{20,6}(n) 
+ \frac{1108142208}{233027}\,\EuFrak{b}_{20,7}(n) 
- \frac{87462672}{7517}\,\EuFrak{b}_{20,8}(n)     \\
- \frac{4875587328}{233027}\,\EuFrak{b}_{20,9}(n) 
- \frac{1996649856}{233027}\,\EuFrak{b}_{20,10}(n) 
+ \frac{126800896}{7517}\,\EuFrak{b}_{20,11}(n)     \\
+ \frac{18330282624}{233027}\,\EuFrak{b}_{20,12}(n) 
\,\biggr)\,q^{n},
\label{convolSum-eqn-5_4-1_3}
\end{multline}
\begin{multline}
(\,7\, E_{2}(q^{7}) -  E_{2}(q)\,)\,E_{4}(q^{2}) =
= 6 + \sum_{n=1}^{\infty}\biggl(\, 
- \frac{2101}{51471}\,\sigma_{5}(n)
- \frac{42020}{51471}\,\sigma_{5}(\frac{n}{2})    \\
+ \frac{2401}{7353}\,\sigma_{5}(\frac{n}{7})  
+ \frac{48020}{7353}\,\sigma_{5}(\frac{n}{14})  
+ \frac{2800}{817}\,\EuFrak{b}_{14,1}(n)    \\
+ \frac{372480}{817}\,\EuFrak{b}_{14,2}(n)
+ \frac{2820080}{817}\,\EuFrak{b}_{14,3}(n)
- \frac{1644160}{817}\,\EuFrak{b}_{14,4}(n)    \\
+ \frac{8142080}{817}\,\EuFrak{b}_{14,5}(n)
- \frac{4753920}{817}\,\EuFrak{b}_{14,6}(n)
+ \frac{36097440}{817}\,\EuFrak{b}_{14,7}(n)    \\
+ \frac{34958080}{817}\,\EuFrak{b}_{14,8}(n)
\,\biggr)\,q^{n},
\label{convolSum-eqn-7_2-1_3}
\end{multline}
\begin{multline}
(\,7\, E_{2}(q^{7}) -  E_{2}(q)\,)\,E_{4}(q^{3}) =
= 6 + \sum_{n=1}^{\infty}\biggl(\, 
- \frac{2101}{223041}\,\sigma_{5}(n)
- \frac{63030}{74347}\,\sigma_{5}(\frac{n}{3})    \\
+ \frac{2401}{31863}\,\sigma_{5}(\frac{n}{7}) 
+ \frac{72030}{10621}\,\sigma_{5}(\frac{n}{21}) 
- \frac{4930000}{10621}\,\EuFrak{b}_{21,1}(n)    \\
- \frac{2344560}{817}\,\EuFrak{b}_{21,2}(n) 
- \frac{10694880}{817}\,\EuFrak{b}_{21,3}(n) 
+ \frac{5131200}{817}\,\EuFrak{b}_{21,4}(n)     \\
+ \frac{901920}{43}\,\EuFrak{b}_{21,5}(n) 
- \frac{48257040}{817}\,\EuFrak{b}_{21,6}(n) 
+ \frac{460703520}{10621}\,\EuFrak{b}_{21,7}(n)     \\
+ \frac{137650560}{817}\,\EuFrak{b}_{21,8}(n) 
- \frac{199923120}{817}\,\EuFrak{b}_{21,9}(n) 
+ \frac{55816800}{817}\,\EuFrak{b}_{21,10}(n)     \\
+ \frac{439823520}{817}\,\EuFrak{b}_{21,11}(n) 
+ \frac{394960}{817}\,\EuFrak{b}_{21,12}(n) 
\,\biggr)\,q^{n},
\label{convolSum-eqn-7_3-1_3}
\end{multline}
\end{corollary}
\begin{proof} 
	These identities follow immediately 
	when one sets for example $(\alpha,\beta)=(1,7)$, $(1,8)$, 
	$(1,14)$, $(1,20)$, $(1,21)$ in  
	\hyperref[convolution-lemma_a_b]{Lemma \ref*{convolution-lemma_a_b}}.
	In case of $\alpha\beta=20$, we take all $n$ belonging to the set $\{\,1,2,\ldots,16, 17,20\,\}$ 
	to obtain a system of $18$ linear equations with unknowns $X_{\delta}$ and $Y_{j}$, 
	where $\delta\in D(20)$ and $1\leq i\leq 12$.
\end{proof}

The second corollary handles the case $E_{4}(q^{\alpha})(\,E_{2}(q) - \beta\,E_{2}(q^{\beta})\,)$.  We only consider the special cases $1\neq\alpha = \beta$ and $\alpha> 1$ whenever $\beta=1$.
\begin{corollary} \label{lema-w_12-21}
It holds that  
\begin{multline}  
\, E_{4}(q)\,(\,E_{2}(q) - 12\,E_{2}(q^{12})\,)
	=  -11 + \sum_{n=1}^{\infty}\biggl(\, 
\frac{1153}{1274}\,\sigma_{5}(n)
- \frac{279}{1274}\,\sigma_{5}(\frac{n}{2})
- \frac{789}{1274}\,\sigma_{5}(\frac{n}{3})  \\
- \frac{640}{637}\,\sigma_{5}(\frac{n}{4})
- \frac{2451}{1274}\,\sigma_{5}(\frac{n}{6})
- \frac{5184}{637}\,\sigma_{5}(\frac{n}{12})
- \frac{200916}{91}\,\EuFrak{b}_{12,1}(n)  
- \frac{362664}{13}\,\EuFrak{b}_{12,2}(n)    \\
- \frac{85968}{7}\,\EuFrak{b}_{12,3}(n)  
- \frac{23880384}{91}\,\EuFrak{b}_{12,4}(n)  
+ 3456\,\EuFrak{b}_{12,5}(n)  
+ \frac{841536}{7}\,\EuFrak{b}_{12,6}(n)  \\
- 290304\,\EuFrak{b}_{12,7}(n)  
\,\biggr)\,q^{n},
\label{convolSum-eqn-1_12_3-1}
\end{multline}
\begin{multline}  
\, E_{4}(q^{12})\,(\,E_{2}(q) - 12\,E_{2}(q^{12})\,)
=  -11 + \sum_{n=1}^{\infty}\biggl(\, 
- \frac{1}{61152}\,\sigma_{5}(n)
+ \frac{11}{20384}\,\sigma_{5}(\frac{n}{2})   \\
+ \frac{61}{20384}\,\sigma_{5}(\frac{n}{3})   
+ \frac{20}{1911}\,\sigma_{5}(\frac{n}{4})  
+ \frac{899}{20384}\,\sigma_{5}(\frac{n}{6}) 
- \frac{7044}{637}\,\sigma_{5}(\frac{n}{12})  \\
- \frac{8739}{364}\,\EuFrak{b}_{12,1}(n)  
- \frac{39321}{182}\,\EuFrak{b}_{12,2}(n)  
+ \frac{837}{7}\,\EuFrak{b}_{12,3}(n)  
- \frac{80484}{91}\,\EuFrak{b}_{12,4}(n)  \\
- 504\,\EuFrak{b}_{12,5}(n)  
- \frac{324}{7}\,\EuFrak{b}_{12,6}(n)  
- 864\,\EuFrak{b}_{12,7}(n)  
\,\biggr)\,q^{n},
\label{convolSum-eqn-12_1_3-1}
\end{multline}
\begin{multline}  
\, E_{4}(q)\,(\,E_{2}(q) - 15\,E_{2}(q^{15})\,)
=  -14 + \sum_{n=1}^{\infty}\biggl(\, 
\frac{917834}{992719}\,\sigma_{5}(n)
+ \frac{1460295}{992719}\,\sigma_{5}(\frac{n}{3})   \\
+ \frac{104862815}{992719}\,\sigma_{5}(\frac{n}{5})    
- \frac{121139010}{992719}\,\sigma_{5}(\frac{n}{15})
- \frac{413824680}{141817}\,\EuFrak{b}_{15,1}(n)    \\
- \frac{756726120}{10909}\,\EuFrak{b}_{15,2}(n)   
+ \frac{93504960}{10909}\,\EuFrak{b}_{15,3}(n) 
- \frac{263947320}{10909}\,\EuFrak{b}_{15,4}(n)    \\
- \frac{90853616520}{141817}\,\EuFrak{b}_{15,5}(n)   
- \frac{283798080}{10909}\,\EuFrak{b}_{15,6}(n) 
- \frac{3149876160}{10909}\,\EuFrak{b}_{15,7}(n) \\
+ \frac{339972480}{10909}\,\EuFrak{b}_{15,8}(n) 
\,\biggr)\,q^{n},
\label{convolSum-eqn-1_15_3-1}
\end{multline}
\begin{multline}  
\, E_{4}(q^{15})\,(\,E_{2}(q) - 15\,E_{2}(q^{15})\,)
=  -14 + \sum_{n=1}^{\infty}\biggl(\, 
 - \frac{274}{8934471}\,\sigma_{5}(n)
- \frac{5457}{992719}\,\sigma_{5}(\frac{n}{3})  \\
+ \frac{8090239}{8934471}\,\sigma_{5}(\frac{n}{5}) 
- \frac{14791494}{992719}\,\sigma_{5}(\frac{n}{15}) 
- \frac{3405800}{141817}\,\EuFrak{b}_{15,1}(n)    \\
- \frac{1158952}{10909}\,\EuFrak{b}_{15,12}(n) 
+ \frac{35264}{10909}\,\EuFrak{b}_{15,3}(n) 
- \frac{2061176}{10909}\,\EuFrak{b}_{15,4}(n)    \\
- \frac{177635528}{141817}\,\EuFrak{b}_{15,5}(n) 
- \frac{2367936}{10909}\,\EuFrak{b}_{15,6}(n) 
- \frac{34294464}{10909}\,\EuFrak{b}_{15,7}(n)   \\
- \frac{1203968}{10909}\,\EuFrak{b}_{15,8}(n) 
\,\biggr)\,q^{n},
\label{convolSum-eqn-15_1_3-1}
\end{multline}
\begin{multline}  
\, E_{4}(q)\,(\,E_{2}(q) - 20\,E_{2}(q^{20})\,)
=  -19 + \sum_{n=1}^{\infty}\biggl(\, 
\frac{53619409}{58722804}\,\sigma_{5}(n)
+ \frac{188714723}{58722804}\,\sigma_{5}(\frac{n}{2})  \\
- \frac{2103232}{1631189}\,\sigma_{5}(\frac{n}{4})  
- \frac{9257815}{8388972}\,\sigma_{5}(\frac{n}{5})
- \frac{398439023}{58722804}\,\sigma_{5}(\frac{n}{10})  \\
- \frac{68258944}{4893567}\,\sigma_{5}(\frac{n}{20})
- \frac{960956950}{233027}\,\EuFrak{b}_{20,1}(n) 
- \frac{7006131104}{233027}\,\EuFrak{b}_{20,12}(n)   \\
- \frac{31247296000}{233027}\,\EuFrak{b}_{20,3}(n)
- \frac{31145902336}{233027}\,\EuFrak{b}_{20,4}(n)
- \frac{15378042050}{233027}\,\EuFrak{b}_{20,5}(n)   \\
- \frac{32291496384}{233027}\,\EuFrak{b}_{20,6}(n)
- \frac{542908645632}{233027}\,\EuFrak{b}_{20,7}(n)
- \frac{309602112}{7517}\,\EuFrak{b}_{20,8}(n)   \\
- \frac{653138738688}{233027}\,\EuFrak{b}_{20,9}(n)   
+ \frac{678376410624}{233027}\,\EuFrak{b}_{20,10}(n)
+ \frac{13917769216}{7517}\,\EuFrak{b}_{20,11}(n)   \\
+ \frac{1138525192704}{233027}\,\EuFrak{b}_{20,12}(n)
\,\biggr)\,q^{n},
\label{convolSum-eqn-1_20_3-1}
\end{multline}
\begin{multline}  
\, E_{4}(q^{20})\,(\,E_{2}(q) - 20\,E_{2}(q^{20})\,)
=  -19 + \sum_{n=1}^{\infty}\biggl(\, 
- \frac{1567}{117445608}\,\sigma_{5}(n) 
+ \frac{110791}{117445608}\,\sigma_{5}(\frac{n}{2})  \\
+ \frac{10004}{233027}\,\sigma_{5}(\frac{n}{4}) 
+ \frac{702215}{234891216}\,\sigma_{5}(\frac{n}{5}) 
+ \frac{10264633}{234891216}\,\sigma_{5}(\frac{n}{10})   \\
- \frac{93420884}{4893567}\,\sigma_{5}(\frac{n}{20}) 
- \frac{5594215}{233027}\,\EuFrak{b}_{20,1}(n) 
- \frac{16718864}{233027}\,\EuFrak{b}_{20,2}(n)   \\
- \frac{89883520}{233027}\,\EuFrak{b}_{20,3}(n) 
- \frac{32106736}{233027}\,\EuFrak{b}_{20,4}(n) 
+ \frac{527968775}{466054}\,\EuFrak{b}_{20,5}(n) \\
- \frac{86133984}{233027}\,\EuFrak{b}_{20,6}(n) 
- \frac{1416189312}{233027}\,\EuFrak{b}_{20,7}(n) 
- \frac{4970832}{7517}\,\EuFrak{b}_{20,8}(n)   \\
- \frac{976297728}{233027}\,\EuFrak{b}_{20,9}(n) 
- \frac{618500736}{233027}\,\EuFrak{b}_{20,10}(n) 
+ \frac{64669696}{7517}\,\EuFrak{b}_{20,11}(n)   \\
+ \frac{4227164544}{233027}\,\EuFrak{b}_{20,12}(n) 
\,\biggr)\,q^{n},
\label{convolSum-eqn-20_1_3-1}
\end{multline}
\begin{multline}  
\, E_{4}(q)\,(\,E_{2}(q) - 21\,E_{2}(q^{21})\,)
=  -20 + \sum_{n=1}^{\infty}\biggl(\, 
\frac{10121}{10621}\,\sigma_{5}(n)  
- \frac{4050}{10621}\,\sigma_{5}(\frac{n}{3})  \\
- \frac{24010}{10621}\,\sigma_{5}(\frac{n}{7})  
- \frac{194481}{10621}\,\sigma_{5}(\frac{n}{21})  
+ \frac{75600}{10621}\,\EuFrak{b}_{21,1}(n)   \\
- \frac{27075600}{817}\,\EuFrak{b}_{21,2}(n)     
- \frac{233275680}{817}\,\EuFrak{b}_{21,3}(n)    
- \frac{400645440}{817}\,\EuFrak{b}_{21,4}(n)   \\
- \frac{171387360}{817}\,\EuFrak{b}_{21,5}(n)    
- \frac{997680240}{817}\,\EuFrak{b}_{21,6}(n)     
- \frac{23349501360}{10621}\,\EuFrak{b}_{21,7}(n)  \\   
+ \frac{430513920}{817}\,\EuFrak{b}_{21,8}(n)    
- \frac{4488019920}{817}\,\EuFrak{b}_{21,9}(n)     
+ \frac{561042720}{817}\,\EuFrak{b}_{21,10}(n)    \\ 
- \frac{162254880}{817}\,\EuFrak{b}_{21,11}(n)   
- \frac{3554640}{817}\,\EuFrak{b}_{21,12}(n)   
\,\biggr)\,q^{n},
\label{convolSum-eqn-1_21_3-1}
\end{multline}
\begin{multline}  
\, E_{4}(q^{21})\,(\,E_{2}(q) - 21\,E_{2}(q^{21})\,)
=  -20 + \sum_{n=1}^{\infty}\biggl(\, 
 \frac{1}{223041}\,\sigma_{5}(n) 
+ \frac{30}{74347}\,\sigma_{5}(\frac{n}{3})  \\
+ \frac{350}{31863}\,\sigma_{5}(\frac{n}{37}) 
- \frac{212541}{10621}\,\sigma_{5}(\frac{n}{21}) 
- \frac{80}{10621}\,\EuFrak{b}_{21,1}(n)   \\
- \frac{58800}{817}\,\EuFrak{b}_{21,2}(n) 
- \frac{430560}{817}\,\EuFrak{b}_{21,3}(n) 
- \frac{624000}{817}\,\EuFrak{b}_{21,4}(n) \\
+ \frac{2400}{43}\,\EuFrak{b}_{21,5}(n) 
- \frac{1026960}{817}\,\EuFrak{b}_{21,6}(n) 
- \frac{22742160}{10621}\,\EuFrak{b}_{21,7}(n)   \\
+ \frac{1159680}{817}\,\EuFrak{b}_{21,8}(n) 
- \frac{4078800}{817}\,\EuFrak{b}_{21,9}(n) 
- \frac{126240}{817}\,\EuFrak{b}_{21,10}(n)   \\
+ \frac{1260000}{817}\,\EuFrak{b}_{21,11}(n) 
- \frac{19600}{817}\,\EuFrak{b}_{21,12}(n) 
\,\biggr)\,q^{n}.
\label{convolSum-eqn-21_1_3-1}
\end{multline}
\end{corollary}
\begin{proof} Similar to that of 
	\hyperref[lema-w_7-21]{Corollary \ref*{lema-w_7-21}}.
\end{proof}

We now state and prove our main result of this section. 
\begin{corollary} \label{convolutionSum-w_3-21}
Let $n$ be a positive integer. Then 
\begin{align}
W_{(1,3)}^{3,1}(n)\, = W_{(3,1)}^{1,3}(n)\, = \,&
\frac{1}{104}\,\sigma_{5}(n) 
+ \frac{81}{1040}\,\sigma_{5}(\frac{n}{3}) 
+ \frac{1}{24}\,\sigma_{3}(n)    
- \frac{1}{24}\,n\,\sigma_{3}(n)           \notag \\  &  
- \frac{1}{240}\,\sigma(\frac{n}{3}) 
- \frac{1}{104}\,\EuFrak{b}_{9,1}(n),
\label{convolutionSum-w_1_3}  
\end{align} 
\begin{align}
W_{(3,1)}^{3,1}(n)\, = W_{(1,3)}^{1,3}(n)\, = \,&
\frac{1}{1040}\,\sigma_{5}(n)  
+ \frac{9}{104}\,\sigma_{5}(\frac{n}{3}) 
+ \frac{1}{24}\,\sigma_{3}(\frac{n}{3})   
- \frac{1}{8}\,n\,\sigma_{3}(\frac{n}{3})      \notag \\  &   
- \frac{1}{240}\,\sigma(n) 
+ \frac{1}{312}\,\EuFrak{b}_{9,1}(n),
\label{convolutionSum-w_3_1}  
\end{align} 
\begin{align}
W_{(1,4)}^{3,1}(n)\, = W_{(4,1)}^{1,3}(n)\, = \,&
\frac{1}{192}\,\sigma_{5}(n)
+ \frac{1}{64}\,\sigma_{5}(\frac{n}{2})
+ \frac{1}{15}\,\sigma_{5}(\frac{n}{4})
+ \frac{1}{24}\,\sigma_{3}(n)    
- \frac{1}{32}\,n\,\sigma_{3}(n)           \notag \\  &  
- \frac{1}{240}\,\sigma(\frac{n}{4}) 
- \frac{1}{64}\,\EuFrak{b}_{32,1}(n),
\label{convolutionSum-w_1_4}  
\end{align} 
\begin{align}
W_{(4,1)}^{3,1}(n)\, = W_{(1,4)}^{1,3}(n)\, = \,&
\frac{1}{3840}\,\sigma_{5}(n)  
+ \frac{1}{256}\,\sigma_{5}(\frac{n}{2})   
+ \frac{1}{12}\,\sigma_{5}(\frac{n}{4}) 
+ \frac{1}{24}\,\sigma_{3}(\frac{n}{4})   
- \frac{1}{8}\,n\,\sigma_{3}(\frac{n}{4})      \notag \\  &   
- \frac{1}{240}\,\sigma(n) 
+ \frac{1}{256}\,\EuFrak{b}_{32,1}(n),
\label{convolutionSum-w_4_1}  
\end{align} 
\begin{align}
W_{(1,6)}^{3,1}(n)\, = W_{(6,1)}^{1,3}(n)\, = \,&
\frac{5}{2184}\,\sigma_{5}(n)  
+ \frac{2}{273}\,\sigma_{5}(\frac{n}{2})   
+ \frac{27}{1456}\,\sigma_{5}(\frac{n}{3})   
+ \frac{27}{455}\,\sigma_{5}(\frac{n}{6})      \notag \\  &   
+ \frac{1}{24}\,\sigma_{3}(n)     
- \frac{1}{48}\,n\,\sigma_{3}(n)      
- \frac{1}{240}\,\sigma(\frac{n}{6}) 
- \frac{101}{4368}\,\EuFrak{b}_{12,1}(n)    \notag \\  &  
- \frac{121}{546}\,\EuFrak{b}_{12,2}(n)
+ \frac{3}{14}\,\EuFrak{b}_{12,3}(n),
\label{convolutionSum-w_1_6}  
\end{align} 
\begin{align}
W_{(6,1)}^{3,1}(n)\, = W_{(1,6)}^{1,3}(n)\, = \,&
\frac{1}{21840}\,\sigma_{5}(n) 
+ \frac{1}{1092}\,\sigma_{5}(\frac{n}{2})
+ \frac{3}{728}\,\sigma_{5}(\frac{n}{3})
+ \frac{15}{182}\,\sigma_{5}(\frac{n}{6})    \notag \\  &  
+ \frac{1}{24}\,\sigma_{3}(\frac{n}{6})  
- \frac{1}{8}\,n\,\sigma_{3}(\frac{n}{6})   
- \frac{1}{240}\,\sigma(n) 
+ \frac{3}{728}\,\EuFrak{b}_{12,1}(n)      \notag \\  & 
+ \frac{19}{546}\,\EuFrak{b}_{12,2}(n)
- \frac{1}{28}\,\EuFrak{b}_{12,3}(n),
\label{convolutionSum-w_6_1}  
\end{align} 
\begin{align}
W_{(2,3)}^{3,1}(n)\, = W_{(3,2)}^{1,3}(n)\, = \,&
 \frac{1}{2184}\,\sigma_{5}(n) 
+ \frac{5}{546}\,\sigma_{5}(\frac{n}{2})
+ \frac{27}{7280}\,\sigma_{5}(\frac{n}{3})
+ \frac{27}{364}\,\sigma_{5}(\frac{n}{6})      \notag \\  & 
+ \frac{1}{24}\,\sigma_{3}(\frac{n}{2})   
- \frac{1}{24}\,n\,\sigma_{3}(\frac{n}{2})   
- \frac{1}{240}\,\sigma(\frac{n}{3}) 
- \frac{1}{2184}\,\EuFrak{b}_{12,1}(n)      \notag \\  & 
+ \frac{4}{273}\,\EuFrak{b}_{12,2}(n)
- \frac{3}{28}\,\EuFrak{b}_{12,3}(n),
\label{convolutionSum-w_2_3}  
\end{align} 
\begin{align}
W_{(3,2)}^{3,1}(n)\, = W_{(2,3)}^{1,3}(n)\, = \,&
\frac{1}{4368}\,\sigma_{5}(n) 
+ \frac{1}{1365}\,\sigma_{5}(\frac{n}{2})
+ \frac{15}{728}\,\sigma_{5}(\frac{n}{3})
+ \frac{6}{91}\,\sigma_{5}(\frac{n}{6})    \notag \\  & 
+ \frac{1}{24}\,\sigma_{3}(\frac{n}{3})   
- \frac{1}{16}\,n\,\sigma_{3}(\frac{n}{3})   
- \frac{1}{240}\,\sigma(\frac{n}{2}) 
- \frac{1}{4368}\,\EuFrak{b}_{12,1}(n)      \notag \\  & 
- \frac{1}{182}\,\EuFrak{b}_{12,2}(n)
+ \frac{1}{14}\,\EuFrak{b}_{12,3}(n),
\label{convolutionSum-w_3_2}  
\end{align} 
\begin{align}
W_{(1,12)}^{3,1}(n)\, = W_{(12,1)}^{1,3}(n)\, = \,&
\frac{121}{174720}\,\sigma_{5}(n) 
+ \frac{93}{58240}\,\sigma_{5}(\frac{n}{2})
+ \frac{263}{58240}\,\sigma_{5}(\frac{n}{3})
+ \frac{2}{273}\,\sigma_{5}(\frac{n}{4})   \notag \\  & 
+ \frac{817}{58240}\,\sigma_{5}(\frac{n}{6})
+ \frac{27}{455}\,\sigma_{5}(\frac{n}{12})
+ \frac{1}{24}\,\sigma_{3}(n)       
- \frac{1}{96}\,n\,\sigma_{3}(n)           \notag \\  &  
- \frac{1}{240}\,\sigma(\frac{n}{12}) 
- \frac{5581}{174720}\,\EuFrak{b}_{12,1}(n)
- \frac{1679}{4160}\,\EuFrak{b}_{12,2}(n)      \notag \\  & 
- \frac{199}{1120}\,\EuFrak{b}_{12,3}(n)  
- \frac{41459}{10920}\,\EuFrak{b}_{12,4}(n)
+ \frac{1}{20}\,\EuFrak{b}_{12,5}(n)
+ \frac{487}{280}\,\EuFrak{b}_{12,6}(n)      \notag \\  & 
- \frac{21}{5}\,\EuFrak{b}_{12,7}(n)
\label{convolutionSum-w_1_12}  
\end{align} 
\begin{align}
W_{(12,1)}^{3,1}(n)\, = W_{(1,12)}^{1,3}(n)\, = \,&
- \frac{1}{118560}\,\sigma_{5}(n) 
+ \frac{3}{55328}\,\sigma_{5}(\frac{n}{2})   
+ \frac{17}{63232}\,\sigma_{5}(\frac{n}{3})   
+ \frac{1}{1092}\,\sigma_{5}(\frac{n}{4})      \notag \\  &    
+ \frac{1705}{442624}\,\sigma_{5}(\frac{n}{6})   
+ \frac{15}{182}\,\sigma_{5}(\frac{n}{12})   
+ \frac{1}{24}\,\sigma_{3}(\frac{n}{12})   
- \frac{1}{8}\,n\,\sigma_{3}(\frac{n}{12})        \notag \\  & 
- \frac{1}{240}\,\sigma(n)    
+ \frac{33}{7904}\,\EuFrak{b}_{12,1}(n)
+ \frac{55}{1456}\,\EuFrak{b}_{12,2}(n)
- \frac{93}{4864}\,\EuFrak{b}_{12,3}(n)      \notag \\  & 
+ \frac{2545}{10374}\,\EuFrak{b}_{12,4}(n) 
+ \frac{1}{38}\,\EuFrak{b}_{12,5}(n)
- \frac{3}{133}\,\EuFrak{b}_{12,6}(n)
+ \frac{15}{38}\,\EuFrak{b}_{12,7}(n),
\label{convolutionSum-w_12_1}  
\end{align} 
\begin{align}
W_{(3,4)}^{3,1}(n)\, = W_{(4,3)}^{1,3}(n)\, = \,&
 \frac{1}{13440}\,\sigma_{5}(n) 
+ \frac{9}{58240}\,\sigma_{5}(\frac{n}{2}) 
+ \frac{23}{4480}\,\sigma_{5}(\frac{n}{3}) 
+ \frac{1}{1365}\,\sigma_{5}(\frac{n}{4})     \notag \\  & 
+ \frac{901}{58240}\,\sigma_{5}(\frac{n}{6}) 
+ \frac{6}{91}\,\sigma_{5}(\frac{n}{12})   
+ \frac{1}{24}\,\sigma_{3}(\frac{n}{4})   
- \frac{1}{24}\,n\,\sigma_{3}(\frac{n}{4})     \notag \\  & 
- \frac{1}{240}\,\sigma(\frac{n}{3})   
- \frac{1}{13440}\,\EuFrak{b}_{12,1}(n)
- \frac{89}{29120}\,\EuFrak{b}_{12,2}(n)    \notag \\  & 
+ \frac{33}{1120}\,\EuFrak{b}_{12,3}(n)  
+ \frac{71}{3640}\,\EuFrak{b}_{12,4}(n)  
- \frac{7}{20}\,\EuFrak{b}_{12,5}(n)      \notag \\  & 
+ \frac{111}{280}\,\EuFrak{b}_{12,6}(n)
- \frac{3}{5}\,\EuFrak{b}_{12,7}(n)
\label{convolutionSum-w_3-4}  
\end{align} 
\begin{align}
W_{(4,3)}^{3,1}(n)\, = W_{(3,4)}^{1,3}(n)\, = \,&
  - \frac{1}{698880}\,\sigma_{5}(n) 
+ \frac{107}{232960}\,\sigma_{5}(\frac{n}{2}) 
+ \frac{61}{232960}\,\sigma_{5}(\frac{n}{3})     \notag \\  & 
+ \frac{5}{546}\,\sigma_{5}(\frac{n}{4}) 
+ \frac{803}{232960}\,\sigma_{5}(\frac{n}{6}) 
+ \frac{27}{364}\,\sigma_{5}(\frac{n}{12})     \notag \\  & 
+ \frac{1}{24}\,\sigma_{3}(\frac{n}{3})   
- \frac{1}{32}\,n\,\sigma_{3}(\frac{n}{3})   
- \frac{1}{240}\,\sigma(\frac{n}{4}) 
+ \frac{1}{698880}\,\EuFrak{b}_{12,1}(n)    \notag \\  & 
- \frac{47}{116480}\,\EuFrak{b}_{12,2}(n)
+ \frac{19}{4480}\,\EuFrak{b}_{12,3}(n)
+ \frac{5099}{43680}\,\EuFrak{b}_{12,4}(n)    \notag \\  & 
- \frac{1}{80}\,\EuFrak{b}_{12,5}(n)
- \frac{127}{1120}\,\EuFrak{b}_{12,6}(n)
+ \frac{21}{20}\,\EuFrak{b}_{12,7}(n),
\label{convolutionSum-w_4_3}  
\end{align} 
\begin{align}
W_{(1,7)}^{1,3}(n)\, = W_{(7,1)}^{3,1}(n)\, = \,&
\frac{7}{196080}\,\sigma_{5}(n) 
+ \frac{1715}{19608}\,\sigma_{5}(\frac{n}{7}) 
+ \frac{1}{24}\,\sigma_{3}(\frac{n}{7})   
- \frac{1}{8}\,n\,\sigma_{3}(\frac{n}{7})      \notag \\  &   
- \frac{1}{240}\,\sigma(n) 
+ \frac{27}{6536}\,\EuFrak{b}_{21,1}(n)  
+ \frac{1}{19}\,\EuFrak{b}_{21,2}(n)  
+ \frac{1171}{6536}\,\EuFrak{b}_{21,3}(n),  
\label{convolutionSum-w_1_7}  
\end{align} 
\begin{align}
W_{(7,1)}^{1,3}(n) \, =\,W_{(1,7)}^{3,1}(n) \, =\, &
\frac{35}{19608}\,\sigma_{5}(n) 
+ \frac{16807}{196080}\,\sigma_{5}(\frac{n}{7})
+ \frac{1}{24}\,\sigma_{3}(n)    
- \frac{1}{56}\,n\,\sigma_{3}(n)           \notag \\  &  
- \frac{1}{240}\,\sigma(\frac{n}{7}) 
- \frac{1171}{45752}\,\EuFrak{b}_{21,1}(n)  
- \frac{7}{19}\,\EuFrak{b}_{21,2}(n) 
- \frac{9261}{6536}\,\EuFrak{b}_{21,3}(n), 
\label{convolutionSum-w_7_1}  
\end{align}
\begin{align}
W_{(1,8)}^{1,3}(n) \,  = \,W_{(8,1)}^{3,1}(n) \, =\, &
\frac{1}{61440}\,\sigma_{5}(n)
+ \frac{1}{4096}\,\sigma_{5}(\frac{n}{2})
+ \frac{1}{256}\,\sigma_{5}(\frac{n}{4})
+ \frac{1}{12}\,\sigma_{5}(\frac{n}{8})    \notag \\  & 
+ \frac{1}{24}\,\sigma_{3}(\frac{n}{8})    
- \frac{1}{8}\,n\,\sigma_{3}(\frac{n}{8})  
- \frac{1}{240}\,\sigma(n) 
+ \frac{17}{4096}\,\EuFrak{b}_{32,1}(n)   \notag \\  & 
+ \frac{3}{256}\,\EuFrak{b}_{32,2}(n) 
+ \frac{1}{16}\,\EuFrak{b}_{32,3}(n)
\label{convolutionSum-w_1_8}  
\end{align}
\begin{align}
W_{(8,1)}^{1,3}(n) \, = \,W_{(1,8)}^{3,1}(n) \, =\, & 
\frac{1}{768}\,\sigma_{5}(n)
+ \frac{1}{256}\,\sigma_{5}(\frac{n}{2})
+ \frac{1}{64}\,\sigma_{5}(\frac{n}{4})
+ \frac{1}{15}\,\sigma_{3}(\frac{n}{8})  \notag \\  &  
+ \frac{1}{24}\,\sigma_{3}(n)    
- \frac{1}{64}\,n\,\sigma_{3}(n)  
- \frac{1}{240}\,\sigma(\frac{n}{8}) 
- \frac{7}{256}\,\EuFrak{b}_{32,1}(n)   \notag \\  & 
- \frac{9}{64}\,\EuFrak{b}_{32,2}(n)  
- \frac{1}{2}\,\EuFrak{b}_{32,3}(n)
	\label{convolutionSum-w_8_1}  
\end{align} 
\begin{align}
W_{(1,14)}^{1,3}(n) \, = \,W_{(14,1)}^{3,1}(n) \, =\, & 
\frac{1}{588240}\,\sigma_{5}(n)
+ \frac{1}{29412}\,\sigma_{5}(\frac{n}{2})
+ \frac{245}{58824}\,\sigma_{5}(\frac{n}{7})
+ \frac{1225}{14706}\,\sigma_{5}(\frac{n}{14})   \notag \\  & 
+ \frac{1}{24}\,\sigma_{3}(\frac{n}{14})    
- \frac{1}{8}\,n\,\sigma_{3}(\frac{n}{14})  
- \frac{1}{240}\,\sigma(n) 
+ \frac{245}{58824}\,\EuFrak{b}_{14,1}(n)   \notag \\  & 
+ \frac{265}{4902}\,\EuFrak{b}_{14,2}(n) 
+ \frac{11461}{58824}\,\EuFrak{b}_{14,3}(n) 
+ \frac{799}{14706}\,\EuFrak{b}_{14,4}(n)  \notag \\  & 
+ \frac{818}{7353}\,\EuFrak{b}_{14,5}(n)  
+ \frac{2276}{2451}\,\EuFrak{b}_{14,6}(n) 
- \frac{2995}{9804}\,\EuFrak{b}_{14,7}(n)   \notag \\  & 
- \frac{2}{7353}\,\EuFrak{b}_{14,8}(n), 
\label{convolutionSum-w_1_14}  
\end{align}
\begin{align}
W_{(14,1)}^{1,3}(n)\, = \,W_{(1,14)}^{3,1}(n) \, =\, & 
\frac{25}{58824}\,\sigma_{5}(n)
+ \frac{10}{7353}\,\sigma_{5}(\frac{n}{2})
+ \frac{2401}{117648}\,\sigma_{5}(\frac{n}{7})
+ \frac{2401}{36765}\,\sigma_{5}(\frac{n}{14})  \notag \\  & 
+ \frac{1}{24}\,\sigma_{3}(n)    
- \frac{1}{112}\,n\,\sigma_{3}(n)  
- \frac{1}{240}\,\sigma(\frac{n}{14}) 
- \frac{27311}{823536}\,\EuFrak{b}_{14,1}(n)   \notag \\  & 
- \frac{5503}{9804}\,\EuFrak{b}_{14,2}(n) 
- \frac{320881}{117648}\,\EuFrak{b}_{14,3}(n) 
- \frac{34217}{14706}\,\EuFrak{b}_{14,4}(n)  \notag \\  & 
- \frac{11398}{7353}\,\EuFrak{b}_{14,5}(n)  
- \frac{56020}{2451}\,\EuFrak{b}_{14,6}(n) 
+ \frac{20983}{4902}\,\EuFrak{b}_{14,7}(n)   \notag \\  & 
- \frac{80}{7353}\,\EuFrak{b}_{14,8}(n), 
\label{convolutionSum-w_14_1}  
\end{align}
\begin{align}
W_{(1,15)}^{1,3}(n)\, = \,W_{(15,1)}^{3,1}(n) \, =\, & 
- \frac{137}{51054120}\,\sigma_{5}(n)
- \frac{5457}{11345360}\,\sigma_{5}(\frac{n}{3})
+ \frac{8090239}{102108240}\,\sigma_{5}(\frac{n}{5})   \notag \\  & 
+ \frac{99291}{11345360}\,\sigma_{5}(\frac{n}{15})
+ \frac{1}{24}\,\sigma_{3}(\frac{n}{15})    
- \frac{1}{8}\,n\,\sigma_{3}(\frac{n}{15})  
- \frac{1}{240}\,\sigma(n)   \notag \\  & 
+ \frac{85145}{20421648}\,\EuFrak{b}_{15,1}(n) 
+ \frac{144869}{7854480}\,\EuFrak{b}_{15,2}(n) 
- \frac{551}{981810}\,\EuFrak{b}_{15,3}(n)   \notag \\  & 
+ \frac{257647}{7854480}\,\EuFrak{b}_{15,4}(n)  
+ \frac{22204441}{102108240}\,\EuFrak{b}_{15,5}(n) 
+ \frac{4111}{109090}\,\EuFrak{b}_{15,6}(n)   \notag \\  & 
+ \frac{59539}{109090}\,\EuFrak{b}_{15,7}(n) 
+ \frac{9406}{490905}\,\EuFrak{b}_{15,8}(n),
\label{convolutionSum-w_1_15}  
\end{align}
\begin{align}
W_{(15,1)}^{1,3}(n)\,  = \,W_{(1,15)}^{3,1}(n) \, =\, & 	
 \frac{14977}{34036080}\,\sigma_{5}(n)
- \frac{97353}{11345360}\,\sigma_{5}(\frac{n}{3})
- \frac{20972563}{34036080}\,\sigma_{5}(\frac{n}{5})   \notag \\  & 
+ \frac{4037967}{5672680}\,\sigma_{5}(\frac{n}{15})
+ \frac{1}{24}\,\sigma_{3}(n)    
- \frac{1}{120}\,n\,\sigma_{3}(n)  
- \frac{1}{240}\,\sigma(\frac{n}{15})    \notag \\  & 
- \frac{383171}{11345360}\,\EuFrak{b}_{15,1}(n) 
- \frac{2102017}{2618160}\,\EuFrak{b}_{15,2}(n) 
+ \frac{32467}{327270}\,\EuFrak{b}_{15,3}(n)    \notag \\  & 
- \frac{733187}{2618160}\,\EuFrak{b}_{15,4}(n) 
- \frac{84123719}{11345360}\,\EuFrak{b}_{15,5}(n) 
- \frac{32847}{109090}\,\EuFrak{b}_{15,6}(n)    \notag \\  & 
- \frac{364569}{109090}\,\EuFrak{b}_{15,7}(n) 
+ \frac{59023}{163635}\,\EuFrak{b}_{15,8}(n), 
\label{convolutionSum-w_15_1}  
\end{align}
\begin{align}
W_{(1,20)}^{1,3}(n)\,  = \,W_{(20,1)}^{3,1}(n) \, =\, & 	
- \frac{1567}{1342235520}\,\sigma_{5}(n)
+ \frac{110791}{1342235520}\,\sigma_{5}(\frac{n}{2})
+ \frac{17507}{4660540}\,\sigma_{5}(\frac{n}{4})  \notag \\ &
+ \frac{140443}{536894208}\,\sigma_{5}(\frac{n}{5})
+ \frac{10264633}{2684471040}\,\sigma_{5}(\frac{n}{10}) \notag \\ &
+ \frac{556307}{6990810}\,\sigma_{5}(\frac{n}{20}) 
+ \frac{1}{24}\,\sigma_{3}(\frac{n}{20})
- \frac{1}{8}\,n\,\sigma_{3}(\frac{n}{20})
- \frac{1}{240}\,\sigma(n)   \notag \\ &  
+ \frac{1118843}{268447104}\,\EuFrak{b}_{20,1}(n) 
+ \frac{1044929}{83889720}\,\EuFrak{b}_{20,2}(n)   \notag \\ &
+ \frac{140443}{2097243}\,\EuFrak{b}_{20,3}(n)   
+ \frac{2006671}{83889720}\,\EuFrak{b}_{20,4}(n)   \notag \\ &
- \frac{105593755}{536894208}\,\EuFrak{b}_{20,5}(n)  
+ \frac{897229}{13981620}\,\EuFrak{b}_{20,6}(n)    \notag \\ &
+ \frac{1229331}{1165135}\,\EuFrak{b}_{20,7}(n)  
+ \frac{103559}{902040}\,\EuFrak{b}_{20,8}(n)   \notag \\ &
+ \frac{2542442}{3495405}\,\EuFrak{b}_{20,9}(n)  
+ \frac{536893}{1165135}\,\EuFrak{b}_{20,10}(n)   \notag \\ &
- \frac{505232}{338265}\,\EuFrak{b}_{20,11}(n)  
- \frac{11008241}{3495405}\,\EuFrak{b}_{20,12}(n),  
\label{convolutionSum-w_1_20}  
\end{align}
\begin{align}
W_{(20,1)}^{1,3}(n)\,  = \,W_{(1,20)}^{3,1}(n) \, =\, & 
\frac{1020679}{2684471040}\,\sigma_{5}(n)
- \frac{188714723}{13422355200}\,\sigma_{5}(\frac{n}{2})
+ \frac{32863}{5825675}\,\sigma_{5}(\frac{n}{4})  \notag \\ &
+ \frac{12960941}{2684471040}\,\sigma_{5}(\frac{n}{5})
+ \frac{398439023}{13422355200}\,\sigma_{5}(\frac{n}{10})   \notag \\ &
+ \frac{1066546}{17477025}\,\sigma_{5}(\frac{n}{20}) 
 + \frac{1}{24}\,\sigma_{3}(n)
 - \frac{1}{160}\,n\,\sigma_{3}(n) 
-   \frac{1}{240}\,\sigma(\frac{n}{20})  \notag \\ &
- \frac{19219139}{536894208}\,\EuFrak{b}_{20,1}(n)  
- \frac{218941597}{838897200}\,\EuFrak{b}_{20,2}(n) \notag \\ &  
- \frac{2441195}{2097243}\,\EuFrak{b}_{20,3}(n)    
- \frac{121663681}{104862150}\,\EuFrak{b}_{20,4}(n)  \notag \\ &
- \frac{307560841}{536894208}\,\EuFrak{b}_{20,5}(n)  
- \frac{168184877}{139816200}\,\EuFrak{b}_{20,6}(n)    \notag \\ &
- \frac{235637433}{11651350}\,\EuFrak{b}_{20,7}(n)   
- \frac{1612511}{4510200}\,\EuFrak{b}_{20,8}(n)  \notag \\ &
- \frac{425220533}{17477025}\,\EuFrak{b}_{20,9}(n)   
+ \frac{147217103}{5825675}\,\EuFrak{b}_{20,10}(n)  \notag \\ &
+ \frac{27183143}{1691325}\,\EuFrak{b}_{20,11}(n)  
+ \frac{741227339}{17477025}\,\EuFrak{b}_{20,12}(n),  
\label{convolutionSum-w_20_1}
\end{align}
\begin{align}
W_{(1,21)}^{1,3}(n)\,  = \,W_{(21,1)}^{3,1}(n) \, =\, & 
\frac{1}{2549040}\,\sigma_{5}(n) 
+ \frac{245}{254904}\,\sigma_{5}(\frac{n}{7}) 
+ \frac{9}{104}\,\sigma_{5}(\frac{n}{21})   \notag \\  & 
+ \frac{1}{24}\,\sigma_{3}(\frac{n}{21})    
- \frac{1}{8}\,n\,\sigma_{3}(\frac{n}{21})  
- \frac{1}{240}\,\sigma(n)    \notag \\  & 
+ \frac{177}{42484}\,\EuFrak{b}_{21,1}(n) 
+ \frac{107}{1976}\,\EuFrak{b}_{21,2}(n) 
+ \frac{2078}{10621}\,\EuFrak{b}_{21,3}(n)   \notag \\  & 
+ \frac{3}{52}\,\EuFrak{b}_{21,4}(n) 
+ \frac{3}{26}\,\EuFrak{b}_{21,5}(n) 
+ \frac{27}{52}\,\EuFrak{b}_{21,6}(n)   \notag \\  & 
+ \frac{55}{312}\,\EuFrak{b}_{21,7}(n) 
+ \frac{2}{13}\,\EuFrak{b}_{21,8}(n) 
+ \frac{23}{13}\,\EuFrak{b}_{21,9}(n)   \notag \\  & 
- \frac{35}{52}\,\EuFrak{b}_{21,10}(n) 
- \frac{31}{26}\,\EuFrak{b}_{21,11}(n) 
+ \frac{3}{52}\,\EuFrak{b}_{21,12}(n),
\label{convolutionSum-w_1_21}  
\end{align} 
\begin{align}
W_{(21,1)}^{1,3}(n)\,  =\,W_{(1,21)}^{3,1}(n) \, =\,  & 
\frac{25}{127452}\,\sigma_{5}(n)
+ \frac{2401}{254904}\,\sigma_{5}(\frac{n}{7}) 
+ \frac{81}{1040}\,\sigma_{5}(\frac{n}{21})     \notag \\  & 
+ \frac{1}{24}\,\sigma_{3}(n)    
- \frac{1}{168}\,n\,\sigma_{3}(n)  
- \frac{1}{240}\,\sigma(\frac{n}{21})    
- \frac{16019}{446082}\,\EuFrak{b}_{21,1}(n)   \notag \\  &  
- \frac{3755}{5928}\,\EuFrak{b}_{21,2}(n) 
- \frac{415273}{127452}\,\EuFrak{b}_{21,3}(n)  
- \frac{177}{52}\,\EuFrak{b}_{21,4}(n)     \notag \\  & 
- \frac{73}{26}\,\EuFrak{b}_{21,5}(n) 
- \frac{657}{52}\,\EuFrak{b}_{21,6}(n)   
- \frac{1713}{104}\,\EuFrak{b}_{21,7}(n)     \notag \\  & 
+ \frac{12}{13}\,\EuFrak{b}_{21,8}(n) 
- \frac{681}{13}\,\EuFrak{b}_{21,9}(n) 
+ \frac{609}{52}\,\EuFrak{b}_{21,10}(n)     \notag \\  & 
+ \frac{165}{26}\,\EuFrak{b}_{21,11}(n) 
+ \frac{135}{52}\,\EuFrak{b}_{21,12}(n),
\label{convolutionSum-w_21_1}  
\end{align} 
\begin{align}
W_{(5,3)}^{1,3}(n)\,  =\,W_{(3,5)}^{3,1}(n) \, =\,  & 
\frac{5363}{102108240}\,\sigma_{5}(n)
+ \frac{33831}{11345360}\,\sigma_{5}(\frac{n}{3}) 
- \frac{10712659}{51054120}\,\sigma_{5}(\frac{n}{5})    \notag \\  & 
+ \frac{3338883}{11345360}\,\sigma_{5}(\frac{n}{15}) 
+ \frac{1}{24}\,\sigma_{3}(\frac{n}{3})    
- \frac{1}{40}\,n\,\sigma_{3}(\frac{n}{3})  
- \frac{1}{240}\,\sigma(\frac{n}{5})    \notag \\  & 
- \frac{5363}{102108240}\,\EuFrak{b}_{15,1}(n) 
- \frac{128137}{7854480}\,\EuFrak{b}_{15,2}(n) 
+ \frac{31687}{981810}\,\EuFrak{b}_{15,3}(n)    \notag \\  & 
+ \frac{438133}{7854480}\,\EuFrak{b}_{15,4}(n) 
+ \frac{27841273}{102108240}\,\EuFrak{b}_{15,5}(n) 
- \frac{11069}{109090}\,\EuFrak{b}_{15,6}(n)    \notag \\  & 
- \frac{257003}{109090}\,\EuFrak{b}_{15,7}(n) 
+ \frac{7003}{490905}\,\EuFrak{b}_{15,8}(n), 
\label{convolutionSum-w_5_3}  
\end{align} 
\begin{align}
W_{(5,4)}^{1,3}(n)\,  =\,W_{(4,5)}^{3,1}(n) \, =\,  & 
\frac{4591}{536894208}\,\sigma_{5}(n)
+ \frac{1460777}{13422355200}\,\sigma_{5}(\frac{n}{2})
+ \frac{573583}{34954050}\,\sigma_{5}(\frac{n}{4})    \notag \\  & 
+ \frac{338063}{1342235520}\,\sigma_{5}(\frac{n}{5})
+ \frac{25485149}{6711177600}\,\sigma_{5}(\frac{n}{10}) \notag \\ &
+ \frac{1559503}{23302700}\,\sigma_{5}(\frac{n}{20})   
+ \frac{1}{24}\,\sigma_{3}(\frac{n}{4})    
- \frac{1}{40}\,n\,\sigma_{3}(\frac{n}{4})  
- \frac{1}{240}\,\sigma(\frac{n}{5})    \notag \\  & 
- \frac{4591}{536894208}\,\EuFrak{b}_{20,1}(n)  
- \frac{164011}{419448600}\,\EuFrak{b}_{20,2}(n)   \notag \\ &
- \frac{4591}{2097243}\,\EuFrak{b}_{20,3}(n)     
+ \frac{12287251}{419448600}\,\EuFrak{b}_{20,4}(n) \notag \\ &  
- \frac{6000859}{268447104}\,\EuFrak{b}_{20,5}(n)  
- \frac{6724451}{69908100}\,\EuFrak{b}_{20,6}(n)      \notag \\  & 
- \frac{961929}{5825675}\,\EuFrak{b}_{20,7}(n)  
+ \frac{1822139}{4510200}\,\EuFrak{b}_{20,8}(n)   \notag \\ &
+ \frac{12696842}{17477025}\,\EuFrak{b}_{20,9}(n)     
+ \frac{1733203}{5825675}\,\EuFrak{b}_{20,10}(n)   \notag \\ &
- \frac{990632}{1691325}\,\EuFrak{b}_{20,11}(n)  
- \frac{47735111}{17477025}\,\EuFrak{b}_{20,12}(n),  
\label{convolutionSum-w_5_4}  
\end{align} 
\begin{align}
W_{(2,7)}^{1,3}(n)\,  =\,W_{(7,2)}^{3,1}(n) \, =\,  & 
\frac {1}{117648}\,\sigma_{5}(n)
+ \frac {1}{36765}\,\sigma_{5}(\frac{n}{2}) 
+ \frac {1225}{58824}\,\sigma_{5}(\frac{n}{7})    \notag \\  & 
+ \frac {490}{7353}\,\sigma_{5}(\frac{n}{14})
+ \frac{1}{24}\,\sigma_{3}(\frac{n}{7})    
- \frac{1}{16}\,n\,\sigma_{3}(\frac{n}{7})  
- \frac{1}{240}\,\sigma(\frac{n}{2})    \notag \\  & 
- \frac {1}{117648}\,\EuFrak{b}_{14,1}(n) 
+ \frac {37}{9804}\,\EuFrak{b}_{14,2}(n) 
+ \frac {2455}{117648}\,\EuFrak{b}_{14,3}(n)    \notag \\  & 
+ \frac {149}{14706}\,\EuFrak{b}_{14,4}(n) 
- \frac {326}{7353}\,\EuFrak{b}_{14,5}(n) 
- \frac {140}{2451}\,\EuFrak{b}_{14,6}(n)    \notag \\  & 
+ \frac {1253}{4902}\,\EuFrak{b}_{14,7}(n) 
+ \frac {3920}{7353}\,\EuFrak{b}_{14,8}(n), 
\label{convolutionSum-w_2_7}  
\end{align} 
\begin{align}
W_{(7,2)}^{1,3}(n)\,  =\,W_{(2,7)}^{3,1}(n) \, =\,  & 
\frac{5}{58824}\,\sigma_{5}(n)
+ \frac{25}{14706}\,\sigma_{5}(\frac{n}{2}) 
+ \frac{2401}{588240}\,\sigma_{5}(\frac{n}{7})    \notag \\  & 
+ \frac{2401}{29412}\,\sigma_{5}(\frac{n}{14}) 
+ \frac{1}{24}\,\sigma_{3}(\frac{n}{2})    
- \frac{1}{56}\,n\,\sigma_{3}(\frac{n}{2})  
- \frac{1}{240}\,\sigma(\frac{n}{7})    \notag \\  & 
- \frac{5}{58824}\,\EuFrak{b}_{14,1}(n)  
- \frac{194}{17157}\,\EuFrak{b}_{14,2}(n) 
- \frac{35251}{411768}\,\EuFrak{b}_{14,3}(n)    \notag \\  & 
+ \frac{367}{7353}\,\EuFrak{b}_{14,4}(n) 
- \frac{12722}{51471}\,\EuFrak{b}_{14,5}(n) 
+ \frac{2476}{17157}\,\EuFrak{b}_{14,6}(n)    \notag \\  & 
- \frac{75203}{68628}\,\EuFrak{b}_{14,7}(n) 
- \frac{54622}{51471}\,\EuFrak{b}_{14,8}(n), 
\label{convolutionSum-w_7_2}  
\end{align} 
\begin{align}
W_{(3,7)}^{1,3}(n)\,  =\,W_{(7,3)}^{3,1}(n) \, =\,  & 
\frac {1}{254904}\,\sigma_{5}(n)
+ \frac {27}{80}\,\sigma_{5}(\frac{n}{3})
+ \frac {1225}{127452}\,\sigma_{5}(\frac{n}{7})   \notag \\  & 
- \frac {27}{104}\,\sigma_{5}(\frac{n}{21})
+ \frac{1}{24}\,\sigma_{3}(\frac{n}{7})    
- \frac{1}{24}\,n\,\sigma_{3}(\frac{n}{7})  
- \frac{1}{240}\,\sigma(\frac{n}{3})    \notag \\  & 
- \frac {1}{254904}\,\EuFrak{b}_{21,1}(n) 
- \frac {1}{5928}\,\EuFrak{b}_{21,2}(n) 
- \frac {85435}{254904}\,\EuFrak{b}_{21,3}(n)    \notag \\  & 
- \frac {35}{52}\,\EuFrak{b}_{21,4}(n) 
- \frac {9}{26}\,\EuFrak{b}_{21,5}(n) 
- \frac {1163}{104}\,\EuFrak{b}_{21,6}(n) 
- \frac {263}{104}\,\EuFrak{b}_{21,7}(n)    \notag \\  & 
+ \frac {7}{13}\,\EuFrak{b}_{21,8}(n) 
- \frac {7699}{52}\,\EuFrak{b}_{21,9}(n) 
+ \frac {27}{52}\,\EuFrak{b}_{21,10}(n) 
+ \frac {41}{26}\,\EuFrak{b}_{21,11}(n)    \notag \\  & 
- \frac {14337}{26}\,\EuFrak{b}_{21,12}(n), 
\label{convolutionSum-w_3_7}  
\end{align} 
\begin{align}
W_{(7,3)}^{1,3}(n)\,  =\,W_{(3,7)}^{3,1}(n) \, =\,  & 
\frac{5}{254904}\,\sigma_{5}(n)
+ \frac{75}{42484}\,\sigma_{5}(\frac{n}{3}) 
+ \frac{2401}{2549040}\,\sigma_{5}(\frac{n}{7})    \notag \\  & 
+ \frac{7203}{84968}\,\sigma_{5}(\frac{n}{21}) 
+ \frac{1}{24}\,\sigma_{3}(\frac{n}{3})    
- \frac{1}{56}\,n\,\sigma_{3}(\frac{n}{3})  
- \frac{1}{240}\,\sigma(\frac{n}{7})    \notag \\  & 
+ \frac{16019}{1338246}\,\EuFrak{b}_{21,1}(n)  
+ \frac{9769}{137256}\,\EuFrak{b}_{21,2}(n)  
+ \frac{1061}{3268}\,\EuFrak{b}_{21,3}(n)     \notag \\  & 
- \frac{5345}{34314}\,\EuFrak{b}_{21,4}(n)  
- \frac{1879}{3612}\,\EuFrak{b}_{21,5}(n)  
+ \frac{201071}{137256}\,\EuFrak{b}_{21,6}(n)     \notag \\  & 
- \frac{319933}{297388}\,\EuFrak{b}_{21,7}(n)  
- \frac{71693}{17157}\,\EuFrak{b}_{21,8}(n)  
+ \frac{277671}{45752}\,\EuFrak{b}_{21,9}(n)    \notag \\  &  
- \frac{116285}{68628}\,\EuFrak{b}_{21,10}(n)  
- \frac{305433}{22876}\,\EuFrak{b}_{21,11}(n)  
- \frac{4937}{411768}\,\EuFrak{b}_{21,12}(n).  
\label{convolutionSum-w_7_3}  
\end{align} 
\end{corollary}
\begin{proof} 
	These identities follow from \autoref{convolution_a_b} when we set  
	for example $(\alpha,\beta)=(1,14)$, $(3,4)$, $(1,20)$, $(5,4)$, $(1,21)$, $(3,7)$.
\end{proof}


\section{Evaluation of the convolution sums $W_{(\alpha,\beta)}^{2i-1,2j-1}(n)$ 
	for the Levels $\alpha\beta=9,16,18,27,32\in\mathbb{N}^{*}\setminus\N$}
\label{convolution_9_32}

Let $i,j\in\mathbb{N}^{*}$ be such that $i+j=3$. 
Then we give explicit formulae for the convolution sums 
$W_{(\alpha,\beta)}^{2i-1,2j-1}(n)$ for $\alpha\beta=9$, $16$, $18$, $27$, $32$.
These levels belong to $\mathbb{N}^{*}\setminus\N$. Hence, the 
primitive Dirichlet characters are non-trivial.

\subsection{Bases for $\E_{6}(\boldsymbol{\Upgamma}_{0}(\boldsymbol{\upalpha\upbeta}))$ and $\S_{6}(\boldsymbol{\Upgamma}_{0}(\boldsymbol{\upalpha\upbeta}))$ when $\boldsymbol{\upalpha\upbeta=18,27,32}$}  
\label{convolution_9_32-gen}

By the inclusion relation (\hyperref[basis-cusp-eqn]{\ref*{basis-cusp-eqn}}), it is sufficient to consider the bases only for the levels $16$, $18$ $18$, $27$,  and $32$.
 
The dimension formulae for the space of 
cusp forms as given in T.~Miyake's book  
\cite[Thrm 2.5.2,~p.~60]{miyake1989} and W.~A.~Stein's book  
\cite[Prop.\ 6.1, p.\ 91]{wstein} and 
(\hyperref[dimension-Eisenstein]{\ref*{dimension-Eisenstein}})
are applied to compute 
\begin{alignat*}{2}   
\text{dim}(\E_{6}(\Gamma_{0}(9)))=4, \quad 
\text{dim}(\S_{6}(\Gamma_{0}(9)))=3,  \\
\text{dim}(\E_{6}(\Gamma_{0}(16)))=6, \quad
\text{dim}(\S_{6}(\Gamma_{0}(16)))=7, \\
\text{dim}(\E_{6}(\Gamma_{0}(18)))=8, \quad
\text{dim}(\S_{6}(\Gamma_{0}(18)))=11, \\
\text{dim}(\E_{6}(\Gamma_{0}(27)))=6, \quad
\text{dim}(\S_{6}(\Gamma_{0}(27)))=12, \\
\text{dim}(\E_{6}(\Gamma_{0}(32)))=8 \quad \text{ and } \quad
\text{dim}(\S_{6}(\Gamma_{0}(32)))=16.
\end{alignat*}

We use \autoref{ligozat_theorem} to determine many eta quotients which are 
elements of the spaces $\S_{6}(\Gamma_{0}(18))$,  $\S_{6}(\Gamma_{0}(27))$
and $\S_{6}(\Gamma_{0}(32))$, respectively.
 
Let $D(18)$, $27$ and $D(32)$ denote the sets 
of positive divisors of $18$, $27$ and $32$, respectively.

\begin{corollary} \label{basisCusp5_9}
\begin{enumerate}
\item[\textbf{(a)}] Let $\chi(n)=\legendre{-4}{n}$ and $\psi(n)=\legendre{-3}{n}$ be 
primitive Dirichlet characters such that $\chi$ is not an annihilator of 
$\E_{6}(\Gamma_{0}(9))$ and $\psi$ is not an 
annihilator of $\E_{6}(\Gamma_{0}(16))$. Then the sets 
\begin{align*}
\EuScript{B}_{E,18}=\{E_{6}(q^{t})\mid t|18\}\cup\{\,E_{6,\legendre{-4}{n}}(q^{s})\mid 
s=1,3\,\},\\  
\EuScript{B}_{E,27}=\{E_{6}(q^{t})\mid t|27\}\cup\{\,E_{6,\legendre{-4}{n}}(q^{s})\mid 
s=1,3\} \text{ and }  \\
\EuScript{B}_{E,32}=\{E_{6}(q^{t})\mid t|32\}\cup\{\,E_{6,\legendre{-3}{n}}(q^{s})\mid 
s=1,2\,\} 
\end{align*} 
constitute bases of $\E_{6}(\Gamma_{0}(18))$, $\E_{6}(\Gamma_{0}(27))$  
and $\E_{6}(\Gamma_{0}(32))$, 
respectively.
\item[\textbf{(b)}] Let $1\leq i\leq 11$, $1\leq j\leq 12$ 
$1\leq k\leq 16$ 
be positive integers.

Let $\delta_{1}\in D(18)$ and 
$(r(i,\delta_{1}))_{i,\delta_{1}}$ be the  
\autoref{convolutionSums-18-table} of the powers of $\eta(\delta_{1}\,z)$. 

Let $\delta_{2}\in D(27)$ and 
$(r(j,\delta_{2}))_{j,\delta_{2}}$ be the  
\autoref{convolutionSums-27-table} of the powers of $\eta(\delta_{2}\,z)$. 

Let $\delta_{3}\in D(32)$ and 
$(r(k,\delta_{3}))_{k,\delta_{3}}$ be the  
\autoref{convolutionSums-32-table} of the powers of $\eta(\delta_{3}\,z)$. 

Let furthermore  
\begin{align*}
\EuFrak{B}_{18,i}(q)=\underset{\delta_{1}|18}{\prod}\eta^{r(i,\delta_{1})}(\delta_{1}
\,z), \quad   
\EuFrak{B}_{27,j}(q)=\underset{\delta_{2}|27}{\prod}\eta^{r(j,\delta_{2})}(\delta_{2}
\,z) 
 \text{ and }    \\
\EuFrak{B}_{32,k}(q)=\underset{\delta_{3}|32}{\prod}\eta^{r(k,\delta_{3})}(\delta_{3}
\,z)    
\end{align*} 
be selected elements of 
$\S_{6}(\Gamma_{0}(18))$, $\S_{6}(\Gamma_{0}(27))$  
and $\S_{6}(\Gamma_{0}(32))$, 
respectively. 

The sets 
\begin{align*}
\EuScript{B}_{S,18}=\{ \EuFrak{B}_{18,i}(q)\mid ~1\leq i\leq 11\}, \quad  
\EuScript{B}_{S,27}=\{ \EuFrak{B}_{27,j}(q)\mid ~1\leq j\leq 12\} \text{ and }  \\
\EuScript{B}_{S,32}=\{ \EuFrak{B}_{32,k}(q)\mid ~1\leq k\leq 16\}
\end{align*} 
are bases of $\S_{6}(\Gamma_{0}(18))$, $\S_{6}(\Gamma_{0}(27))$ and  $\S_{6}(\Gamma_{0}(32))$,
  respectively.
\item[\textbf{(c)}] The sets 
\begin{align*}
\EuScript{B}_{M,18}=\EuScript{B}_{E,18}\cup\EuScript{B}_{S,18},  
\EuScript{B}_{M,27}=\EuScript{B}_{E,27}\cup\EuScript{B}_{S,27} \text{ and } 
\EuScript{B}_{M,32}=\EuScript{B}_{E,32}\cup\EuScript{B}_{S,32}
\end{align*} 
constitute bases of $\M_{6}(\Gamma_{0}(18))$ and $\M_{6}(\Gamma_{0}(27))$, 
and  $\M_{6}(\Gamma_{0}(32))$,  
respectively. 
\end{enumerate}
\end{corollary}

By \hyperref[basis-remark]{Remark \ref*{basis-remark}} (r1), each 
$\EuFrak{B}_{\alpha\beta,i}(q)$ is expressible in the form 
$\underset{n=1}{\overset{\infty}{\sum}}\EuFrak{b}_{\alpha\beta,i}(n)q^{n}$.

Note that the basis elements $\EuFrak{B}_{27,12}(q)$ and $\EuFrak{B}_{32,16}(q)$
are elements of the cusp spaces $\S_{2}(\Gamma_{0}(27))$ and  
$\S_{2}(\Gamma_{0}(32))$, respectively.

\begin{proof} It holds that $18=3^{2}\times 2$. Since $\gcd(4,3)=1$, it holds that the 
 primitive Dirichlet character $\chi(n)=\legendre{-4}{n}$ is not an annihilator of 
 $\E_{6}(\Gamma_{0}(3^{2}))$. Hence, 
 $\chi(n)=\legendre{-4}{n}$ is not an annihilator of $\E_{6}(\Gamma_{0}(9))$,  
 $\E_{6}(\Gamma_{0}(18))$ and $\E_{6}(\Gamma_{0}(27))$.  Similarly, since 
 $\gcd(3,4)=1$, the primitive Dirichlet character $\psi(n)=\legendre{-3}{n}$ 
 is not an annihilator of the space $\E_{6}(\Gamma_{0}(2^{4}))$. Therefore, $\psi(n)=\legendre{-3}{n}$ is not an annihilator of $\E_{6}(\Gamma_{0}(16))$ 
 and $\E_{6}(\Gamma_{0}(32))$.
 
We only give the proof for $\EuScript{B}_{M,18}=\EuScript{B}_{E,18}\cup
\EuScript{B}_{S,18}$ since the other cases are proved similarly. 
 
\begin{enumerate}
\item[\textbf{(a)}]
Suppose that $x_{\delta},z_{1},z_{3}\in\mathbb{C}$ with $\delta|18$. Let   
$$\underset{\delta|18}{\sum} x_{\delta}\,E_{6}(q^{\delta}) + 
z_{1}\,E_{6,\legendre{-4}{n}}(q) + z_{3}\,E_{6,\legendre{-4}{n}}(q^{3})=0.$$ 
We observe that 
\begin{equation} \label{base-5_9-kronecker}
\legendre{-4}{n}= \begin{cases}
 -1 & \text{ if } n\equiv 3\pmod{4}, \\
 0 & \text{ if } \text{gcd}(4,n) \neq 1, \\
 1 & \text{ if } n\equiv 1\pmod{4}. 
	\end{cases}
\end{equation}
and recall that for all $0\neq a\in\mathbb{Z}$ it holds that $\legendre{a}{0}=0$. 
Since the conductor of the Dirichlet character $\legendre{-4}{n}$ is $4$, we infer from 
(\hyperref[Eisenstein-gen]{\ref*{Eisenstein-gen}})  
that $C_{0}=0$. 
We then deduce  
\begin{equation*}
\underset{\delta|18}{\sum} x_{\delta} 
 + \underset{i=1}{\overset{\infty}
{\sum}}\biggl( - 504\underset{\delta|18}{\sum}\sigma_{5}(\frac{n}{\delta}) x_{\delta} + 
\legendre{-4}{n}\sigma_{5}(n)z_{1} + \legendre{-4}{n}\sigma_{5}(\frac{n}
{3})z_{3}\biggr)q^{n}=0.  
\end{equation*}
Then we equate the coefficients of $q^{n}$ for $n\in D(18)$ plus for example  
$n=5,7$ 
to obtain a system of $8$ linear equations 
whose unique solution is $x_{\delta}=z_{1}=z_{3}=0$ with $\delta\in D(18)$. So, the set 
$\EuScript{B}_{E}$ is linearly independent. 
Hence, the set $\EuScript{B}_{E}$ is a basis of 
$\E_{6}(\Gamma_{0}(18))$.
\item[\textbf{(b)}] 
Suppose that $x_{i}\in\mathbb{C}$ with $1\leq i\leq 11$. Let  
$\underset{i=1}{\overset{11}{\sum}}x_{i}\,\EuFrak{B}_{18,i}(q)=0$. Then   
\begin{equation*}
\underset{i=1}{\overset{11}{\sum}}x_{i}\underset{n=1}{\overset{\infty}{\sum}}\,\EuFrak{b}_{18,i}(n)q^{n}
= \underset{n=1}{\overset{\infty}{\sum}}\biggl(\,\underset{i=1}{\overset{11}{\sum}}\,\EuFrak{b}_{18,i}(n)\,x_{i}\,\biggr)q^{n} = 0.
\end{equation*}
So, we equate the coefficients of $q^{n}$ for $1\leq n\leq 11$ to obtain 
a system of $11$ linear equations 
whose unique solution is $x_{i}=0$ for all  $1\leq i\leq 11$. 
It follows that the set $\EuScript{B}_{S}$ is linearly independent. 
Hence, the set $\EuScript{B}_{S}$ is a basis of 
$S_{6}(\Gamma_{0}(18))$.
\item[\textbf{(c)}]
Since $\M_{6}(\Gamma_{0}(18))=\E_{6}(\Gamma_{0}(18))\oplus 
\S_{6}(\Gamma_{0}(18))$, the result follows from (a) and (b).
\end{enumerate}
\end{proof}

\subsection{Evaluation of $\mathbf{W_{(\boldsymbol{\upalpha,\upbeta})}^{2i-1,2j-1}(n)}$ for  $\boldsymbol{\upalpha\upbeta=9,16,18,27,32}$} 
\label{convolSum-w_45_50}  

In this section, the evaluation of the convolution sum $W_{(\alpha,\beta)}^{2i-1,2j-1}(n)$ is discussed for  
$\alpha\beta=9$, $16$, $18$, $27$ and $32$. 

In the following corollary we consider 
$(\,\alpha\, E_{2}(q^{\alpha}) - E_{2}(q)\,)\,E_{4}(q^{\beta})$ and
$E_{4}(q^{\alpha})\,(\,E_{2}(q) - \beta\, E_{2}(q^{\beta})\,)$ 
simultaneously. We will also make use of \autoref{convolutionSum-theor-sp}.
\begin{corollary} \label{lema_9_32}
It holds that 
\begin{multline}
(\,9\, E_{2}(q^{9}) - E_{2}(q)\,)\,E_{4}(q^{9}) = 
  8 + \sum_{n=1}^{\infty}\biggl(\,
- \frac{1}{7371}\,\sigma_{5}(n)
- \frac{80}{7371}\,\sigma_{5}(\frac{n}{3}) \\
+ \frac{729}{91}\,\sigma_{5}(\frac{n}{9}) 
+ \frac{2800}{117}\,\EuFrak{b}_{18,1}(n)
+ \frac{640}{3}\,\EuFrak{b}_{18,2}(n)
+ \frac{6480}{13}\,\EuFrak{b}_{18,3}(n)
\,\biggr)\,q^{n},
\label{convolSum-eqn-1_9}
\end{multline}
\begin{multline}
E_{4}(q)\,(\,E_{2}(q) - 9\, E_{2}(q^{9})\,) = 
- 8 + \sum_{n=1}^{\infty}\biggl(\,
\frac{81}{91}\,\sigma_{5}(n)  
- \frac{80}{91}\,\sigma_{5}(\frac{n}{3})  \\ 
- \frac{729}{91}\,\sigma_{5}(\frac{n}{9})  
- \frac{19440}{13}\,\EuFrak{b}_{18,1}(n)  
-17280\,\EuFrak{b}_{18,2}(n)
- \frac{680400}{13}\,\EuFrak{b}_{18,3}(n)
\,\biggr)\,q^{n},
\label{convolSum-eqn-9_1}
\end{multline}
\begin{multline}
(\,16\, E_{2}(q^{16}) - E_{2}(q)\,)\,E_{4}(q^{16})  = 
 15 + \sum_{n=1}^{\infty}\biggl(\, 
- \frac{1}{86016}\,\sigma_{5}(n)
- \frac{5}{28672}\,\sigma_{5}(\frac{n}{2}) \\
- \frac{5}{1792}\,\sigma_{5}(\frac{n}{4}) 
- \frac{5}{112}\,\sigma_{5}(\frac{n}{8}) 
+ \frac{316}{21}\,\sigma_{5}(\frac{n}{16}) 
+ \frac{12285}{512}\,\EuFrak{b}_{32,1}(n)
+ \frac{2295}{32}\,\EuFrak{b}_{32,2}(n)  \\
+ \frac{765}{2}\,\EuFrak{b}_{32,3}(n)
+ \frac{315}{2}\,\EuFrak{b}_{32,4}(n)
+ 360\,\EuFrak{b}_{32,5}(n)
+ 1080\,\EuFrak{b}_{32,6}(n)
+ 1440\,\EuFrak{b}_{32,7}(n)
\,\biggr)\,q^{n},
\label{convolSum-eqn-1_16}
\end{multline}
\begin{multline}
E_{4}(q)\,(\,E_{2}(q) - 16\, E_{2}(q^{16})\,)  = 
 - 15 + \sum_{n=1}^{\infty}\biggl(\, 
 \frac{79}{84}\,\sigma_{5}(n)  
  - \frac{5}{28}\,\sigma_{5}(\frac{n}{2})  \\
  - \frac{5}{7}\,\sigma_{5}(\frac{n}{4}) 
- \frac{20}{7}\,\sigma_{5}(\frac{n}{8})  
  - \frac{256}{21}\,\sigma_{5}(\frac{n}{16})  
-3150\,\EuFrak{b}_{32,1}(n)
-22680\,\EuFrak{b}_{32,2}(n)   \\
-92160\,\EuFrak{b}_{32,3}(n)
-105120\,\EuFrak{b}_{32,4}(n)
-322560\,\EuFrak{b}_{32,5}(n)
-414720\,\EuFrak{b}_{32,6}(n) \\
-184320\,\EuFrak{b}_{32,7}(n)
\,\biggr)\,q^{n},
\label{convolSum-eqn-16_1}
\end{multline}
\begin{multline}
(\,18\, E_{2}(q^{18}) - E_{2}(q)\,)\,E_{4}(q^{18})  = 
  17 + \sum_{n=1}^{\infty}\biggl(\, 
- \frac{1}{154791}\,\sigma_{5}(n)
- \frac{20}{154791}\,\sigma_{5}(\frac{n}{2}) \\
- \frac{80}{154791}\,\sigma_{5}(\frac{n}{3}) 
- \frac{1600}{154791}\,\sigma_{5}(\frac{n}{6}) 
- \frac{30}{637}\,\sigma_{5}(\frac{n}{9}) 
+ \frac{10866}{637}\,\sigma_{5}(\frac{n}{18})  \\
+ \frac{2792}{351}\,\EuFrak{b}_{18,1}(n)
+ \frac{293744}{2457}\,\EuFrak{b}_{18,2}(n)
+ \frac{97960}{819}\,\EuFrak{b}_{18,3}(n)
+ \frac{61504}{63}\,\EuFrak{b}_{18,4}(n) \\
- \frac{408}{7}\,\EuFrak{b}_{18,5}(n)
+ \frac{2016400}{819}\,\EuFrak{b}_{18,6}(n)
- \frac{1216}{21}\,\EuFrak{b}_{18,7}(n)
- \frac{10048}{21}\,\EuFrak{b}_{18,8}(n) \\
+ \frac{6368}{7}\,\EuFrak{b}_{18,9}(n)
+ \frac{39104}{21}\,\EuFrak{b}_{18,10}(n)
+ \frac{3032}{189}\,\EuFrak{b}_{18,11}(n)
\,\biggr)\,q^{n},
\label{convolSum-eqn-1_18}
\end{multline}
\begin{multline}
E_{4}(q)\,(\,E_{2}(q) - 18\, E_{2}(q^{18})\,)  = 
 - 17 + \sum_{n=1}^{\infty}\biggl(\, 
\frac{1811}{1911}\,\sigma_{5}(n)  
  - \frac{320}{1911}\,\sigma_{5}(\frac{n}{2})   \\
 - \frac{800}{1911}\,\sigma_{5}(\frac{n}{3})   
 - \frac{2560}{1911}\,\sigma_{5}(\frac{n}{6})   
 - \frac{2430}{637}\,\sigma_{5}(\frac{n}{9})  
  - \frac{7776}{637}\,\sigma_{5}(\frac{n}{18})  \\  
- \frac{21648}{13}\,\EuFrak{b}_{18,1}(n)  
 - \frac{3354816}{91}\,\EuFrak{b}_{18,2}(n)  
 - \frac{5815440}{91}\,\EuFrak{b}_{18,3}(n)  
 - \frac{2562048}{7}\,\EuFrak{b}_{18,4}(n)  \\
 - \frac{1018656}{7}\,\EuFrak{b}_{18,5}(n)  
 - \frac{126129600}{91}\,\EuFrak{b}_{18,6}(n)  
 - \frac{1396224}{7}\,\EuFrak{b}_{18,7}(n) 
  + \frac{1029888}{7}\,\EuFrak{b}_{18,8}(n)  \\
 + \frac{2871936}{7}\,\EuFrak{b}_{18,9}(n)  
 - \frac{3573504}{7}\,\EuFrak{b}_{18,10}(n)   
- \frac{13728}{7}\,\EuFrak{b}_{18,11}(n)
\,\biggr)\,q^{n}, 
\label{convolSum-eqn-18_1}
\end{multline}
\begin{multline}
(\,2\, E_{2}(q^{2}) - E_{2}(q)\,)\,E_{4}(q^{9})  = 
 1 + \sum_{n=1}^{\infty}\biggl(\, 
- \frac{11}{154791}\,\sigma_{5}(n)
+ \frac{32}{154791}\,\sigma_{5}(\frac{n}{2})   \\
- \frac{880}{154791}\,\sigma_{5}(\frac{n}{3})
+ \frac{2560}{154791}\,\sigma_{5}(\frac{n}{6})
- \frac{330}{637}\,\sigma_{5}(\frac{n}{9})
+ \frac{960}{637}\,\sigma_{5}(\frac{n}{18})   \\
+ \frac{14176}{351}\,\EuFrak{b}_{18,1}(n)  
+ \frac{651712}{2457}\,\EuFrak{b}_{18,2}(n)  
+ \frac{214880}{819}\,\EuFrak{b}_{18,3}(n)  
+ \frac{512}{63}\,\EuFrak{b}_{18,4}(n)  \\
+ \frac{1056}{7}\,\EuFrak{b}_{18,5}(n)  
+ \frac{6080}{819}\,\EuFrak{b}_{18,6}(n)  
+ \frac{36352}{21}\,\EuFrak{b}_{18,7}(n)  
- \frac{68864}{21}\,\EuFrak{b}_{18,8}(n) \\  
- \frac{35456}{7}\,\EuFrak{b}_{18,9}(n)  
+ \frac{116992}{21}\,\EuFrak{b}_{18,10}(n)  
- \frac{3104}{189}\,\EuFrak{b}_{18,11}(n) 
\,\biggr)\,q^{n}, 
\label{convolSum-eqn-2_9}
\end{multline}
\begin{multline}
E_{4}(q^{2})\,(\,E_{2}(q) - 9\,E_{2}(q^{9})\,)  = 
- 8 + \sum_{n=1}^{\infty}\biggl(\, 
- \frac{27}{637}\,\sigma_{5}(n)
- \frac{540}{637}\,\sigma_{5}(\frac{n}{2})   \\
+ \frac{80}{1911}\,\sigma_{5}(\frac{n}{3}) 
+ \frac{1600}{1911}\,\sigma_{5}(\frac{n}{6})
+ \frac{243}{637}\,\sigma_{5}(\frac{n}{9})
+ \frac{4860}{637}\,\sigma_{5}(\frac{n}{18})   \\
- \frac{7944}{13}\,\EuFrak{b}_{18,1}(n)  
- \frac{255408}{91}\,\EuFrak{b}_{18,2}(n)  
+ \frac{896040}{91}\,\EuFrak{b}_{18,3}(n)  
- \frac{80064}{7}\,\EuFrak{b}_{18,4}(n)  \\
+ \frac{44712}{7}\,\EuFrak{b}_{18,5}(n)  
- \frac{3429360}{91}\,\EuFrak{b}_{18,6}(n)  
+ \frac{1728}{7}\,\EuFrak{b}_{18,7}(n)  
+ \frac{281664}{7}\,\EuFrak{b}_{18,8}(n)  \\
+ \frac{1552608}{7}\,\EuFrak{b}_{18,9}(n)  
- \frac{1086912}{7}\,\EuFrak{b}_{18,10}(n)  
+ \frac{4296}{7}\,\EuFrak{b}_{18,11}(n) 
\,\biggr)\,q^{n}, 
\label{convolSum-eqn-9_2}
\end{multline}
\begin{multline}
(\,27\, E_{2}(q^{27}) - E_{2}(q)\,)\,E_{4}(q^{27})  = 
 26 + \sum_{n=1}^{\infty}\biggl(\, 
- \frac{1}{597051}\,\sigma_{5}(n)
- \frac{80}{597051}\,\sigma_{5}(\frac{n}{3}) \\
- \frac{80}{7371}\,\sigma_{5}(\frac{n}{9})
+ \frac{2367}{91}\,\sigma_{5}(\frac{n}{27})
- \frac{17968400}{9477}\,\EuFrak{b}_{27,1}(n)
- \frac{2746880}{243}\,\EuFrak{b}_{27,2}(n) \\
- \frac{5880080}{351}\,\EuFrak{b}_{27,3}(n)
+ \frac{104320}{9}\,\EuFrak{b}_{27,4}(n)
+ \frac{640}{3}\,\EuFrak{b}_{27,5}(n)
+ \frac{2560}{3}\,\EuFrak{b}_{27,6}(n) \\
- 50560\,\EuFrak{b}_{27,7}(n)
+ 1920\,\EuFrak{b}_{27,8}(n)
+ \frac{25200}{13}\,\EuFrak{b}_{27,9}(n)  \\
- 460800\,\EuFrak{b}_{27,10}(n)  
+ 1920\,\EuFrak{b}_{27,12}(n)
\,\biggr)\,q^{n}, 
\label{convolSum-eqn-1_27}
\end{multline}
\begin{multline}
E_{4}(q)\,(\,E_{2}(q) - 27\, E_{2}(q^{27})\,)  = 
- 26 + \sum_{n=1}^{\infty}\biggl(\, 
\frac{263}{273}\,\sigma_{5}(n)   
     - \frac{80}{273}\,\sigma_{5}(\frac{n}{3}) \\  
      - \frac{240}{91}\,\sigma_{5}(\frac{n}{9})  
      - \frac{2187}{91}\,\sigma_{5}(\frac{n}{27})   
 + \frac{54512400}{13}\,\EuFrak{b}_{27,1}(n) 
 + 25113600\,\EuFrak{b}_{27,2}(n)  \\
  + \frac{487185840}{13}\,\EuFrak{b}_{27,3}(n)  
  -25453440\,\EuFrak{b}_{27,4}(n)  
  -466560\,\EuFrak{b}_{27,5}(n)  \\
  -1451520\,\EuFrak{b}_{27,6}(n)  
 + 111818880\,\EuFrak{b}_{27,7}(n)  
  -1399680\,\EuFrak{b}_{27,8}(n)   \\
 - \frac{56628720}{13}\,\EuFrak{b}_{27,9}(n)  
 + 1018967040\,\EuFrak{b}_{27,10}(n)  
 -4199040\,\EuFrak{b}_{27,12}(n)
\,\biggr)\,q^{n}, 
\label{convolSum-eqn-27_1}
\end{multline}
\begin{multline}
(\,32\, E_{2}(q^{32}) - E_{2}(q)\,)\,E_{4}(q^{32})  = 
 31 + \sum_{n=1}^{\infty}\biggl(\, 
- \frac{1}{1376256}\,\sigma_{5}(n)
- \frac{5}{458752}\,\sigma_{5}(\frac{n}{2}) \\
- \frac{5}{28672}\,\sigma_{5}(\frac{n}{4})
- \frac{5}{1792}\,\sigma_{5}(\frac{n}{8})
- \frac{5}{112}\,\sigma_{5}(\frac{n}{16})
+ \frac{652}{21}\,\sigma_{5}(\frac{n}{32}) \\
- \frac{11599875}{8192}\,\EuFrak{b}_{32,1}(n)
+ \frac{36855}{512}\,\EuFrak{b}_{32,2}(n)
- \frac{540675}{32}\,\EuFrak{b}_{32,3}(n)
+ \frac{5355}{32}\,\EuFrak{b}_{32,4}(n) \\
+ \frac{23805}{2}\,\EuFrak{b}_{32,5}(n)
+ \frac{2295}{2}\,\EuFrak{b}_{32,6}(n)
- 90630\,\EuFrak{b}_{32,7}(n)
+ \frac{675}{2}\,\EuFrak{b}_{32,8}(n) \\
+ 23400\,\EuFrak{b}_{32,9}(n)
+ 1080\,\EuFrak{b}_{32,10}(n)
+ 1440\,\EuFrak{b}_{32,11}(n)
+ 2520\,\EuFrak{b}_{32,12}(n) \\
+ 96480\,\EuFrak{b}_{32,13}(n)
+ 4320\,\EuFrak{b}_{32,14}(n)
+ 11520\,\EuFrak{b}_{32,15}(n)
+ 1440\,\EuFrak{b}_{32,16}(n)
\,\biggr)\,q^{n},
\label{convolSum-eqn-1_32}
\end{multline}
\begin{multline}
E_{4}(q)\,(\,E_{2}(q) - 32\, E_{2}(q^{32})\,)  = 
- 31 + \sum_{n=1}^{\infty}\biggl(\, 
 \frac{163}{168}\,\sigma_{5}(n)  
    - \frac{5}{56}\,\sigma_{5}(\frac{n}{2})   \\ 
     - \frac{5}{14}\,\sigma_{5}(\frac{n}{4}) 
     - \frac{10}{7}\,\sigma_{5}(\frac{n}{8})  
     - \frac{40}{7}\,\sigma_{5}(\frac{n}{16})  
     - \frac{512}{21}\,\sigma_{5}(\frac{n}{32})     \\ 
+ 5891265\,\EuFrak{b}_{32,1}(n)  
  -56700\,\EuFrak{b}_{32,2}(n)  
 + 70536960\,\EuFrak{b}_{32,3}(n)  
  -367920\,\EuFrak{b}_{32,4}(n)     \\ 
  -48337920\,\EuFrak{b}_{32,5}(n)  
  -1658880\,\EuFrak{b}_{32,6}(n)  
 + 376197120\,\EuFrak{b}_{32,7}(n)  
  -1684800\,\EuFrak{b}_{32,8}(n)     \\ 
  -97228800\,\EuFrak{b}_{32,9}(n)
-5806080\,\EuFrak{b}_{32,10}(n)  
  -9216000\,\EuFrak{b}_{32,11}(n)   
 -6727680\,\EuFrak{b}_{32,12}(n)     \\ 
  -400711680\,\EuFrak{b}_{32,13}(n)   
 -3317760\,\EuFrak{b}_{32,14}(n)  
  -47185920\,\EuFrak{b}_{32,15}(n)  \\
  -5898240\,\EuFrak{b}_{32,16}(n)
\,\biggr)\,q^{n}.
\label{convolSum-eqn-32_1}
\end{multline}
\end{corollary}
\begin{proof} We give the proof for the case where $\alpha=1$ and $\beta=18$. 
The proof for the other cases can be done similarly. 

This follows immediately when one sets $\alpha=1$ and $\beta=18$ in 
\hyperref[convolution-lemma_a_b]{Lemma \ref*{convolution-lemma_a_b}}.
However, we briefly show the proof 
for $(18\,\,E_{2}(q^{18}) - E_{2}(q) )\,E_{4}(q^{18})$ as an example. One obtains 
\begin{align}
(18\,E_{2}(q^{18}) - E_{2}(q) )\,E_{4}(q^{18})  & =   
\sum_{\delta|18}\,x_{\delta} E_{6}(q^{\delta}) + z_{1}\,E_{6,\legendre{-4}{n}}(q) 
+ z_{2}\,E_{6,\legendre{-4}{n}}(q^{2})  \notag \\ & 
  + \sum_{j=1}^{11}\,y_{j}\,\EuFrak{B}_{18,j}(q)  \notag \\ &  
  = \sum_{\delta|18}\,x_{\delta} 
  + C_{0}\,z_{1} + C_{0}\,z_{3}
  + \sum_{i=1}^{\infty}\biggl( 
  \sum_{\delta|18}\,- 504\,\sigma_{5}(\frac{n}{\delta})\,x_{\delta}\notag \\ &
  + \legendre{-4}{n}\,\sigma_{5}(n)\,z_{1}  
  + \legendre{-4}{n}\,\sigma_{5}(\frac{n}{2})\,z_{2} 
  \quad + \sum_{j=1}^{11}\,y_{j}\,\EuFrak{b}_{18,k}(n)\,\biggr)\,q^{n}. 
      \label{convolution_1_18-eqn-0}
\end{align} 
We apply the primitive Dirichlet character  
\begin{equation*} \label{base-3_16-kronecker}
\legendre{-4}{n}= \begin{cases}
 -1 & \text{ if } n\equiv 3\pmod{4}, \\
 0 & \text{ if } \text{gcd}(4,n) \neq 1, \\
 1 & \text{ if } n\equiv 1\pmod{4}. 
	\end{cases}
\end{equation*}

Since the conductor of the Dirichlet character $\legendre{-4}{n}$ is greater than zero, from (\hyperref[Eisenstein-gen]{\ref*{Eisenstein-gen}})
we have $C_{0}=0$.   
Now when we equate the right hand side of 
(\hyperref[convolution_1_18-eqn-0]{\ref*{convolution_1_18-eqn-0}})
with that of  
(\hyperref[evalConvolClass-eqn-11]{\ref*{evalConvolClass-eqn-11}}), and when we take the 
coefficients of $q^{n}$ for which $1\leq n\leq 11$ and 
$n=12$, $13$, $14$, $15$, $16$, $17$, $18$, $36$ for example, we 
obtain a system of $19$ linear equations with a unique solution. 
Hence, we obtain the stated result.
\end{proof} 
 
Now we state and prove our main result of this subsection. 
\begin{corollary} \label{convolutionSum-w_9_32}
Let $n$ be a positive integer. Then 
\begin{align}
 W_{(1,9)}^{1,3}(n) \, = W_{(9,1)}^{3,1}(n) \,=\, & 
\frac{1}{84240}\,\sigma_{5}(n) 
+ \frac{1}{1053}\,\sigma_{5}(\frac{n}{3}) 
+ \frac{9}{104}\,\sigma_{5}(\frac{n}{9})        \notag \\ &
+ \frac{1}{24}\,\sigma_{3}(\frac{n}{9})    
- \frac{1}{8}\,n\,\sigma_{3}(\frac{n}{9})  
- \frac{1}{240}\,\sigma(n)    \notag \\  & 
+ \frac{35}{8424}\,\EuFrak{b}_{18,1}(n) 
+ \frac{1}{27}\,\EuFrak{b}_{18,2}(n) 
+ \frac{9}{104}\,\EuFrak{b}_{18,3}(n) 
\label{convolutionSum-w_1_9}
\end{align}
\begin{align}
 W_{(1,9)}^{3,1}(n) \, =\, W_{(9,1)}^{1,3}(n) \, =\, &
\frac{1}{936}\,\sigma_{5}(n)   
 + \frac{1}{117}\,\sigma_{5}(\frac{n}{3})   
 + \frac{81}{1040}\,\sigma_{5}(\frac{n}{9})        \notag \\ &  
+ \frac{1}{24}\,\sigma_{3}(n)    
- \frac{1}{72}\,n\,\sigma_{3}(n)
- \frac{1}{240}\,\sigma(\frac{n}{9})     \notag \\  & 
- \frac{3}{104}\,\EuFrak{b}_{18,1}(n)   
- \frac{1}{3}\,\EuFrak{b}_{18,2}(n) 
- \frac{105}{104}\,\EuFrak{b}_{18,3}(n), 
\label{convolutionSum-w_9_1}
\end{align}
\begin{align}
 W_{(1,16)}^{1,3}(n) \, =\, W_{(16,1)}^{3,1}(n) \, =\, &
\frac{1}{983040}\,\sigma_{5}(n) 
+ \frac{1}{65536}\,\sigma_{5}(\frac{n}{2}) 
+ \frac{1}{4096}\,\sigma_{5}(\frac{n}{4})    \notag \\  & 
+ \frac{1}{256}\,\sigma_{5}(\frac{n}{8}) 
+ \frac{1}{12}\,\sigma_{5}(\frac{n}{16}) 
+ \frac{1}{24}\,\sigma_{3}(\frac{n}{16})    
- \frac{1}{8}\,n\,\sigma_{3}(\frac{n}{16})     \notag \\  & 
- \frac{1}{240}\,\sigma(n) 
+ \frac{273}{65536}\,\EuFrak{b}_{32,1}(n)
+ \frac{51}{4096}\,\EuFrak{b}_{32,2}(n)
+ \frac{17}{256}\,\EuFrak{b}_{32,3}(n)   \notag \\  & 
+ \frac{7}{256}\,\EuFrak{b}_{32,4}(n)
+ \frac{1}{16}\,\EuFrak{b}_{32,5}(n) 
+ \frac{3}{16}\,\EuFrak{b}_{32,6}(n)
+ \frac{1}{4}\,\EuFrak{b}_{32,7}(n), 
\label{convolutionSum-w_1_16}
\end{align}
\begin{align}
 W_{(1,16)}^{3,1}(n) \, = \, W_{(16,1)}^{1,3}(n) \, = \, & 
\frac{1}{3072}\,\sigma_{5}(n)  
  + \frac{1}{1024}\,\sigma_{5}(\frac{n}{2})   
 + \frac{1}{256}\,\sigma_{5}(\frac{n}{4}) 
  + \frac{1}{64}\,\sigma_{5}(\frac{n}{8})     \notag \\  &   
 + \frac{1}{15}\,\sigma_{5}(\frac{n}{16}) 
+ \frac{1}{24}\,\sigma_{3}(n)    
- \frac{1}{128}\,n\,\sigma_{3}(n)
- \frac{1}{240}\,\sigma(\frac{n}{16})     \notag \\  & 
- \frac{35}{1024}\,\EuFrak{b}_{32,1}(n) 
  - \frac{63}{256}\,\EuFrak{b}_{32,2}(n) 
  -\,\EuFrak{b}_{32,3}(n)
- \frac{73}{64}\,\EuFrak{b}_{32,4}(n)     \notag \\  & 
- \frac{7}{2}\,\EuFrak{b}_{32,5}(n)  
- \frac{9}{2}\,\EuFrak{b}_{32,6}(n)
-2\,\EuFrak{b}_{32,7}(n)
\label{convolutionSum-w_16_1}
\end{align}
\begin{align}
 W_{(1,18)}^{1,3}(n) \, = \, W_{(18,1)}^{3,1}(n) \, = \, & 
\frac{1}{1769040}\,\sigma_{5}(n) 
+ \frac{1}{88452}\,\sigma_{5}(\frac{n}{2}) 
+ \frac{1}{22113}\,\sigma_{5}(\frac{n}{3})   \notag \\  & 
+ \frac{20}{22113}\,\sigma_{5}(\frac{n}{6}) 
+ \frac{3}{728}\,\sigma_{5}(\frac{n}{9}) 
+ \frac{15}{182}\,\sigma_{5}(\frac{n}{18}) 
+ \frac{1}{24}\,\sigma_{3}(\frac{n}{18})      \notag \\  & 
- \frac{1}{8}\,n\,\sigma_{3}(\frac{n}{18}) 
- \frac{1}{240}\,\sigma(n) 
+ \frac{349}{252720}\,\EuFrak{b}_{18,1}(n)     \notag \\  & 
+ \frac{18359}{884520}\,\EuFrak{b}_{18,2}(n) 
+ \frac{2449}{117936}\,\EuFrak{b}_{18,3}(n)
+ \frac{961}{5670}\,\EuFrak{b}_{18,4}(n)     \notag \\  & 
- \frac{17}{1680}\,\EuFrak{b}_{18,5}(n)
+ \frac{25205}{58968}\,\EuFrak{b}_{18,6}(n)
- \frac{19}{1890}\,\EuFrak{b}_{18,7}(n)      \notag \\  & 
- \frac{157}{1890}\,\EuFrak{b}_{18,8}(n)
+ \frac{199}{1260}\,\EuFrak{b}_{18,9}(n)
+ \frac{611}{1890}\,\EuFrak{b}_{18,10}(n)      \notag \\  & 
+ \frac{379}{136080}\,\EuFrak{b}_{18,11}(n)
\label{convolutionSum-w_1_18}
\end{align}
\begin{align}
 W_{(1,18)}^{3,1}(n)\,  =\, W_{(18,1)}^{1,3}(n) \, = \, &
 \frac{5}{19656}\,\sigma_{5}(n)   
 + \frac{2}{2457}\,\sigma_{5}(\frac{n}{2})   
 + \frac{5}{2457}\,\sigma_{5}(\frac{n}{3}) 
 + \frac{16}{2457}\,\sigma_{5}(\frac{n}{6})   \notag \\  &      
 + \frac{27}{1456}\,\sigma_{5}(\frac{n}{9})  
 +  \frac{27}{455}\,\sigma_{5}(\frac{n}{18}) 
+ \frac{1}{24}\,\sigma_{3}(n)    
- \frac{1}{144}\,n\,\sigma_{3}(n)   \notag \\  & 
- \frac{1}{240}\,\sigma(\frac{n}{18}) 
- \frac{451}{28080}\,\EuFrak{b}_{18,1}(n)   
 - \frac{17473}{49140}\,\EuFrak{b}_{18,2}(n)       \notag \\  &   
 - \frac{8077}{13104}\,\EuFrak{b}_{18,3}(n)  
 - \frac{1112}{315}\,\EuFrak{b}_{18,4}(n)   
  - \frac{393}{280}\,\EuFrak{b}_{18,5}(n)       \notag \\  &   
 - \frac{43795}{3276}\,\EuFrak{b}_{18,6}(n)   
 - \frac{202}{105}\,\EuFrak{b}_{18,7}(n)     
 + \frac{149}{105}\,\EuFrak{b}_{18,8}(n)      \notag \\  &   
 + \frac{277}{70}\,\EuFrak{b}_{18,9}(n)   
 - \frac{517}{105}\,\EuFrak{b}_{18,10}(n)   
 - \frac{143}{7560}\,\EuFrak{b}_{18,11}(n), 
\label{convolutionSum-w_18_1}
\end{align}
\begin{align}
 W_{(2,9)}^{1,3}(n)\,  =\, W_{(9,2)}^{3,1}(n) \, = \, &
\frac{1}{353808}\,\sigma_{5}(n) 
+ \frac{1}{110565}\,\sigma_{5}(\frac{n}{2})   
+ \frac{5}{22113}\,\sigma_{5}(\frac{n}{3})      \notag \\  &  
+ \frac{16}{22113}\,\sigma_{5}(\frac{n}{6})   
+ \frac{15}{728}\,\sigma_{5}(\frac{n}{9})   
+ \frac{6}{91}\,\sigma_{5}(\frac{n}{18})   
+ \frac{1}{24}\,\sigma_{3}(\frac{n}{9})      \notag \\  & 
- \frac{1}{8}\,n\,\sigma_{3}(\frac{n}{9}) 
- \frac{1}{240}\,\sigma(\frac{n}{2}) 
- \frac{361}{252720}\,\EuFrak{b}_{18,1}(n)   
- \frac{1993}{442260}\,\EuFrak{b}_{18,2}(n)       \notag \\  & 
+ \frac{2417}{117936}\,\EuFrak{b}_{18,3}(n)   
- \frac{2}{2835}\,\EuFrak{b}_{18,4}(n)   
- \frac{11}{840}\,\EuFrak{b}_{18,5}(n)       \notag \\  & 
- \frac{19}{29484}\,\EuFrak{b}_{18,6}(n)   
- \frac{142}{945}\,\EuFrak{b}_{18,7}(n)   
+ \frac{269}{945}\,\EuFrak{b}_{18,8}(n)   
+ \frac{277}{630}\,\EuFrak{b}_{18,9}(n)      \notag \\  &  
- \frac{457}{945}\,\EuFrak{b}_{18,10}(n)   
+ \frac{97}{68040}\,\EuFrak{b}_{18,11}(n), 
\label{convolutionSum-w_2_9}
\end{align}
\begin{align}
 W_{(2,9)}^{3,1}(n)\,  =\, W_{(9,2)}^{1,3}(n) \, = \, &
\frac{1}{19656}\,\sigma_{5}(n) 
+ \frac{5}{4914}\,\sigma_{5}(\frac{n}{2}) 
+ \frac{1}{2457}\,\sigma_{5}(\frac{n}{3})  
+ \frac{20}{2457}\,\sigma_{5}(\frac{n}{6})   \notag \\  & 
+ \frac{27}{7280}\,\sigma_{5}(\frac{n}{9}) 
+ \frac{27}{364}\,\sigma_{5}(\frac{n}{18}) 
+ \frac{1}{24}\,\sigma_{3}(\frac{n}{2})    
- \frac{1}{72}\,n\,\sigma_{3}(\frac{n}{2})   \notag \\  & 
- \frac{1}{240}\,\sigma(\frac{n}{9}) 
+ \frac{331}{28080}\,\EuFrak{b}_{18,1}(n)   
+ \frac{5321}{98280}\,\EuFrak{b}_{18,2}(n)   
- \frac{2489}{13104}\,\EuFrak{b}_{18,3}(n)      \notag \\  & 
+ \frac{139}{630}\,\EuFrak{b}_{18,4}(n)   
- \frac{69}{560}\,\EuFrak{b}_{18,5}(n)   
+ \frac{4763}{6552}\,\EuFrak{b}_{18,6}(n)   
- \frac{1}{210}\,\EuFrak{b}_{18,7}(n)      \notag \\  & 
- \frac{163}{210}\,\EuFrak{b}_{18,8}(n)   
- \frac{599}{140}\,\EuFrak{b}_{18,9}(n)   
+ \frac{629}{210}\,\EuFrak{b}_{18,10}(n)   
- \frac{179}{15120}\,\EuFrak{b}_{18,11}(n),    
\label{convolutionSum-w_9_2}
\end{align}
\begin{align}
 W_{(1,27)}^{1,3}(n)\,  =\, W_{(27,1)}^{3,1}(n) \, = \, &
\frac{1}{6823440}\,\sigma_{5}(n) 
+ \frac{1}{85293}\,\sigma_{5}(\frac{n}{3}) 
+ \frac{1}{1053}\,\sigma_{5}(\frac{n}{9})      \notag \\  & 
+ \frac{9}{104}\,\sigma_{5}(\frac{n}{27}) 
+ \frac{1}{24}\,\sigma_{3}(\frac{n}{27}) 
- \frac{3}{24}\,n\,\sigma_{3}(\frac{n}{27})      \notag \\  & 
- \frac{1}{240}\,\sigma(n) 
- \frac{224605}{682344}\,\EuFrak{b}_{27,1}(n) 
- \frac{4292}{2187}\,\EuFrak{b}_{27,2}(n)     \notag \\  & 
- \frac{73501}{25272}\,\EuFrak{b}_{27,3}(n) 
+ \frac{163}{81}\,\EuFrak{b}_{27,4}(n) 
+ \frac{1}{27}\,\EuFrak{b}_{27,5}(n)     \notag \\  & 
+ \frac{4}{27}\,\EuFrak{b}_{27,6}(n) 
- \frac{79}{9}\,\EuFrak{b}_{27,7}(n) 
+ \frac{1}{3}\,\EuFrak{b}_{27,8}(n) 
+ \frac{35}{104}\,\EuFrak{b}_{27,9}(n)     \notag \\  & 
- 80\,\EuFrak{b}_{27,10}(n) 
+ \frac{1}{3}\,\EuFrak{b}_{27,12}(n), 
\label{convolutionSum-w_1_27}
\end{align}
\begin{align}
 W_{(1,27)}^{3,1}(n) \, =\, W_{(27,1)}^{1,3}(n) \, = \, &
\frac{1}{8424}\,\sigma_{5}(n)    
 + \frac{1}{1053}\,\sigma_{5}(\frac{n}{3})    
 + \frac{1}{117}\,\sigma_{5}(\frac{n}{9})     \notag \\  &    
 + \frac{81}{1040}\,\sigma_{5}(\frac{n}{27})   
+ \frac{1}{24}\,\sigma_{3}(n)    
- \frac{1}{216}\,n\,\sigma_{3}(n)  
- \frac{1}{240}\,\sigma(\frac{n}{27})     \notag \\  & 
 + \frac{227135}{8424}\,\EuFrak{b}_{27,1}(n)    
 + \frac{4360}{27}\,\EuFrak{b}_{27,2}(n)   
 + \frac{25061}{104}\,\EuFrak{b}_{27,3}(n)         \notag \\  & 
 - \frac{491}{3}\,\EuFrak{b}_{27,4}(n)  
 -3\,\EuFrak{b}_{27,5}(n)    
 - \frac{28}{3}\,\EuFrak{b}_{27,6}(n)   
 + 719\,\EuFrak{b}_{27,7}(n)        \notag \\  & 
  -9\,\EuFrak{b}_{27,8}(n)     
  - \frac{2913}{104}\,\EuFrak{b}_{27,9}(n)   
 + 6552\,\EuFrak{b}_{27,10}(n)   
  -27\,\EuFrak{b}_{27,12}(n), 
\label{convolutionSum-w_27_1}
\end{align}
\begin{align}
 W_{(1,32)}^{1,3}(n) \, =\, W_{(32,1)}^{3,1}(n) \, = \, &
\frac{1}{15728640}\,\sigma_{5}(n)
+ \frac{1}{1048576}\,\sigma_{5}(\frac{n}{2})
+ \frac{1}{65536}\,\sigma_{5}(\frac{n}{4})     \notag \\  & 
+ \frac{1}{4096}\,\sigma_{5}(\frac{n}{8})
+ \frac{1}{256}\,\sigma_{5}(\frac{n}{16})
+ \frac{1}{12}\,\sigma_{5}(\frac{n}{32})
+ \frac{1}{24}\,\sigma_{3}(\frac{n}{32})      \notag \\  & 
- \frac{3}{24}\,n\,\sigma_{3}(\frac{n}{32}) 
- \frac{1}{240}\,\sigma(n) 
- \frac{257775}{1048576}\,\EuFrak{b}_{32,1}(n)      \notag \\  & 
+ \frac{819}{65536}\,\EuFrak{b}_{32,2}(n) 
- \frac{12015}{4096}\,\EuFrak{b}_{32,3}(n) 
+ \frac{119}{4096}\,\EuFrak{b}_{32,4}(n)      \notag \\  & 
+ \frac{529}{256}\,\EuFrak{b}_{32,5}(n) 
+ \frac{51}{256}\,\EuFrak{b}_{32,6}(n) 
- \frac{1007}{64}\,\EuFrak{b}_{32,7}(n)      \notag \\  & 
+ \frac{15}{256}\,\EuFrak{b}_{32,8}(n) 
+ \frac{65}{16}\,\EuFrak{b}_{32,9}(n) 
+ \frac{3}{16}\,\EuFrak{b}_{32,10}(n) 
+ \frac{1}{4}\,\EuFrak{b}_{32,11}(n)      \notag \\  & 
+ \frac{7}{16}\,\EuFrak{b}_{32,12}(n) 
+ \frac{67}{4}\,\EuFrak{b}_{32,13}(n) 
+ \frac{3}{4}\,\EuFrak{b}_{32,14}(n) 
+ 2\,\EuFrak{b}_{32,15}(n)      \notag \\  & 
+ \frac{1}{4}\,\EuFrak{b}_{32,16}(n), 
\label{convolutionSum-w_1_32}
\end{align}
\begin{align}
 W_{(1,32)}^{3,1}(n) \, =\, W_{(32,1)}^{1,3}(n) \, = \, &
\frac{1}{12288}\,\sigma_{5}(n)  
+ \frac{1}{4096}\,\sigma_{5}(\frac{n}{2}) 
 + \frac{1}{1024}\,\sigma_{5}(\frac{n}{4}) 
 + \frac{1}{256}\,\sigma_{5}(\frac{n}{8})     \notag \\  & 
 + \frac{1}{64}\,\sigma_{5}(\frac{n}{16}) 
 + \frac{1}{15}\,\sigma_{5}(\frac{n}{32})  
+ \frac{1}{24}\,\sigma_{3}(n)    
- \frac{1}{256}\,n\,\sigma_{3}(n)   \notag \\  & 
- \frac{1}{240}\,\sigma(\frac{n}{32}) 
 + \frac{130917}{4096}\,\EuFrak{b}_{32,1}(n)  
 - \frac{315}{1024}\,\EuFrak{b}_{32,2}(n)    \notag \\  &  
 + \frac{6123}{16}\,\EuFrak{b}_{32,3}(n)  
     - \frac{511}{256}\,\EuFrak{b}_{32,4}(n)  
     - \frac{1049}{4}\,\EuFrak{b}_{32,5}(n)  
 - 9\,\EuFrak{b}_{32,6}(n)      \notag \\  & 
 + 2041\,\EuFrak{b}_{32,7}(n)   
      - \frac{585}{64}\,\EuFrak{b}_{32,8}(n)  
     - \frac{1055}{2}\,\EuFrak{b}_{32,9}(n)  
     - \frac{63}{2}\,\EuFrak{b}_{32,10}(n)     \notag \\  & 
 -50\,\EuFrak{b}_{32,11}(n)      
 - \frac{73}{2}\,\EuFrak{b}_{32,12}(n)  
 -2174\,\EuFrak{b}_{32,13}(n)    
 -18\,\EuFrak{b}_{32,14}(n)      \notag \\  & 
  -256\,\EuFrak{b}_{32,15}(n)  
  -32\,\EuFrak{b}_{32,16}(n).
\label{convolutionSum-w_32_1}
\end{align}
\end{corollary}
 
\begin{proof} 
Follows immediately when we set for example 
$(\alpha,\beta)=(1,9)$, $(1,16)$, $(1,18)$, $(1,27)$, $(1,32)$ in 
\autoref{convolution_a_b} and apply \autoref{convolutionSum-theor-sp}. 
\end{proof}

\begin{remark}
In all explicit 
examples evaluated as yet,  we notice that for all $\chi\in\EuScript{C}$ and for all $s\in D(\chi)$ the 
value of $Z(\chi)_{s}$ is zero; That means that the value of $Z(\chi)_{s}$ always  
vanish for all $\alpha\beta$ belonging to $\mathbb{N}^{*}\setminus\N$.	
\end{remark}


\section{Formulae for the Number of Representations of a positive Integer}  
\label{representations}

We make use of the convolution sums evaluated in 
\hyperref[convolution_33_40_56]{Section \ref*{convolution_33_40_56}}
among others to determine 
explicit formulae for the number of representations of a positive
integer $n$  by the quadratic forms (\hyperref[introduction-eq-1]{\ref*{introduction-eq-1}}) 
-- (\hyperref[introduction-eq-4]{\ref*{introduction-eq-4}}).

\subsection{Representations by Quadratic Forms (\hyperref[introduction-eq-1]{\ref*{introduction-eq-1}}) and 
(\hyperref[introduction-eq-2]{\ref*{introduction-eq-2}}) 
}

We give formulae for the number of representations of a positive integer $n$ 
by the Quadratic Form (\hyperref[introduction-eq-1]{\ref*{introduction-eq-1}}) and  (\hyperref[introduction-eq-2]{\ref*{introduction-eq-2}}). We apply among others the evaluated convolution sums for the levels $4$, 
$8$, $12$, $16$. To achieve  
these results, we recall that for example $12=2^{2}\cdot 3$ and $16=2^{4}$, which are of the 
restricted form in 
\hyperref[representations_a_b]{Section \ref*{representations_a_b}}. Therefore, 
we apply \hyperref[representation-prop-1]{Proposition \ref*{representation-prop-1}} 
to conclude that $\Omega_{4}=\{(1,3)\}$ in case $\alpha\beta=12$ and 
$\Omega_{4}=\{(1,4)\}$ in case $\alpha\beta=16$. 

\begin{corollary} \label{representations-coro-40_56}
	Let $n\in\mathbb{N}$ and $a,b)=(1,1)$, $(1,2)$, $(1,4)$. Then  
	\begin{align*}
N_{(1,1)}^{4,8}(n)\, = \,N_{(1,1)}^{8,4}(n) & = 
8\,\sigma_{5}(n) - 512\,\sigma_{5}(\frac{n}{4}) + 16\,\EuFrak{b}_{32,1}(n)\, =\, r_{12}(n),
	\end{align*}
	\begin{align*}
N_{(1,2)}^{4,8}(n)\, =\,N_{(2,1)}^{8,4}(n)\, =\, &  
\frac{1}{2}\,\sigma_{5}(n) 
- \frac{1}{2}\,\sigma_{5}(\frac{n}{2}) 
+ 8\,\sigma_{5}(\frac{n}{4}) 
- 512\,\sigma_{5}(\frac{n}{8}) 
+ \frac{15}{2}\,\EuFrak{b}_{32,1}(n)   \\ &
+ 24\,\EuFrak{b}_{32,2}(n)
+ 128 \,\EuFrak{b}_{32,3}(n),
	\end{align*}
	\begin{align*}
N_{(1,2)}^{8,4}(n)\, =\,N_{(2,1)}^{4,8}(n)\, =\, &
2\,\sigma_{5}(n) 
- 2\,\sigma_{5}(\frac{n}{2}) 
+ 8\,\sigma_{5}(\frac{n}{4}) 
- 512\,\sigma_{5}(\frac{n}{8}) 
+ 14\,\EuFrak{b}_{32,1}(n)   \\ &
- 16\,\EuFrak{b}_{32,1}(\frac{n}{2})  
+ 72\,\EuFrak{b}_{32,2}(n)
+ 256 \,\EuFrak{b}_{32,3}(n),
\end{align*}
\begin{align*}
N_{(1,4)}^{4,8}(n)\, = \,N_{(4,1)}^{8,4}(n)\, = \,  &
\frac{1}{32}\,\sigma_{5}(n) 
+ \frac{15}{32}\,\sigma_{5}(\frac{n}{2}) 
- \frac{65}{2}\,\sigma_{5}(\frac{n}{4}) 
+ 40\,\sigma_{5}(\frac{n}{8}) 
- 512\,\sigma_{5}(\frac{n}{16})   \\ &
+ \frac{255}{32}\,\EuFrak{b}_{32,1}(n) 
- 32\,\EuFrak{b}_{32,1}(\frac{n}{4})  
+ \frac{45}{2}\,\EuFrak{b}_{32,2}(n)
+ 120 \,\EuFrak{b}_{32,3}(n)  \\ &
+ 56\,\EuFrak{b}_{32,4}(n) 
+ 128\,\EuFrak{b}_{32,5}(n)   
+ 384\,\EuFrak{b}_{32,6}(n)  
+ 512\,\EuFrak{b}_{32,7}(n),
	\end{align*}
\begin{align*}
N_{(1,4)}^{8,4}(n)\, =\, N_{(4,1)}^{4,8}(n)\, =\, &
\frac{1}{2}\,\sigma_{5}(n) 
- \frac{5}{2}\,\sigma_{5}(\frac{n}{2}) 
+ 130\,\sigma_{5}(\frac{n}{4}) 
- 120\,\sigma_{5}(\frac{n}{8}) 
- 512\,\sigma_{5}(\frac{n}{16})   \\ &
+ \frac{31}{2}\,\EuFrak{b}_{32,1}(n) 
- 28\,\EuFrak{b}_{32,1}(\frac{n}{2})  
- 128\,\EuFrak{b}_{32,1}(\frac{n}{4})  
+ 126\,\EuFrak{b}_{32,2}(n)  \\ &
- 144\,\EuFrak{b}_{32,2}(\frac{n}{2})  
+ 512 \,\EuFrak{b}_{32,3}(n)
- 512\,\EuFrak{b}_{32,3}(\frac{n}{2})  
+ 584\,\EuFrak{b}_{32,4}(n)   \\ &  
+ 1792\,\EuFrak{b}_{32,5}(n)
+ 2304\,\EuFrak{b}_{32,6}(n)
+ 1024\,\EuFrak{b}_{32,7}(n).
	\end{align*}
\end{corollary}
\begin{proof} 
	These formulae follow immediately from \autoref{representations-theor_a_b} 
	when we set for example $(a,b)=(1,4)$. 
	\begin{align*}
N_{(1,4)}^{4,8}(n)  &  = 
8\,\sigma(n) - 32\,\sigma(\frac{n}{4}) + 16\,\sigma_{3}(\frac{n}{4}) - 32\,\sigma_{3}(\frac{n}{8}) 
+ 256\sigma_{3}(\frac{n}{16})   \\ & 
+ 128\, W_{(1,4)}^{1,3}(n) - 256\, W_{(1,8)}^{1,3}(n) + 2048\, W_{(1,16)}^{1,3}(n) - 512\, W_{(1,1)}^{1,3}(\frac{n}{4})      \\ & 
+ 1024\, W_{(1,2)}^{1,3}(\frac{n}{4})     
- 8192\, W_{(1,4)}^{1,3}(\frac{n}{4}) 
\end{align*}
\begin{align*}
N_{(1,4)}^{8,4}(n) & = ~ 
8\,\sigma(\frac{n}{4}) - 32\,\sigma(\frac{n}{16}) 
+ 16\,\sigma_{3}(n) - 32\,\sigma_{3}(\frac{n}{2}) 
+ 256\sigma_{3}(\frac{n}{4})   \\ & 
+ 128\, W_{(1,4)}^{3,1}(n) - 512\, W_{(1,16)}^{3,1}(n) - 256\, W_{(1,2)}^{3,1}(\frac{n}{2})  + 1024\, W_{(1,8)}^{3,1}(\frac{n}{2})       \\ & 
+ 2048\, W_{(1,1)}^{3,1}(\frac{n}{4}) - 8192\, W_{(1,4)}^{3,1}(\frac{n}{4}) 
\end{align*}
	One can then use the result of 
	\begin{itemize}
		\item  
	(\hyperref[Ramanujan-ident-1-3-gen]{\ref*{Ramanujan-ident-1-3-gen}}),   
	\jgH\ et al.\ \cite[Thrm 6, p.22]{huardetal}, 
		(\hyperref[convolutionSum-w_1_4]{\ref*{convolutionSum-w_1_4}}),  
		(\hyperref[convolutionSum-w_1_8]{\ref*{convolutionSum-w_1_8}})   
	and	(\hyperref[convolutionSum-w_1_16]{\ref*{convolutionSum-w_1_16}}) 
		for the sake of simplification of $N_{(1,4)}^{4,8}(n)$.  
		\item 
	(\hyperref[Ramanujan-ident-1-3-gen]{\ref*{Ramanujan-ident-1-3-gen}}),   
\jgH\ et al.\ \cite[Thrm 6, p.22]{huardetal}, 
		(\hyperref[convolutionSum-w_4_1]{\ref*{convolutionSum-w_4_1}}),  
		(\hyperref[convolutionSum-w_8_1]{\ref*{convolutionSum-w_8_1}}) 
	and	(\hyperref[convolutionSum-w_16_1]{\ref*{convolutionSum-w_16_1}}) 
		to simplify the formulae $N_{(1,4)}^{8,4}(n)$.
	\end{itemize}
\end{proof}

\subsection{Representations by the Quadratic Forms (\hyperref[introduction-eq-1]{\ref*{introduction-eq-3}}) and 
(\hyperref[introduction-eq-1]{\ref*{introduction-eq-4}})}

We determine formulae for the number of representations of a positive integer $n$ 
by the Quadratic Form (\hyperref[introduction-eq-3]{\ref*{introduction-eq-3}}) and 
(\hyperref[introduction-eq-4]{\ref*{introduction-eq-4}}). 
We mainly apply the 
evaluation of the convolution sums for the levels $3$, $6$ and $9$.  

In order to achieve these results, we recall that for example $6=3\cdot 2$ and 
$9=3\cdot 3$, so that by 
\hyperref[representation-prop-2]{Proposition \ref*{representation-prop-2}} 
we get $\Omega_{3}=\{(1,2),(1,3)\}$.
We then deduce the following result:
\begin{corollary} \label{representations-coro_33}
	Let $n\in\mathbb{N}$ and $c,d)=(1,1)$, $(1,2)$, $(1,3)$, $(1,7)$.
	 Then  
 \begin{align*}
 		R_{(1,1)}^{4,8}(n)\, = \,R_{(1,1)}^{8,4}(n)\, = \,  & 
 		\frac{252}{13}\sigma_{5}(n) 
 		- \frac{6804}{13}\,\sigma_{5}(\frac{n}{3}) 
 		+ \frac{216}{13}\,\EuFrak{b}_{18,1}(n)\, =\, s_{12}(n),
 \end{align*}
\begin{align*}
R_{(1,2)}^{4,8}(n)\, =\,R_{(2,1)}^{8,4}(n)\, =\, & 
 		\frac{12}{13}\sigma_{5}(n) 
+ \frac{240}{13}\,\sigma_{5}(\frac{n}{2}) 
- \frac{324}{13}\,\sigma_{5}(\frac{n}{3}) 
- \frac{6480}{13}\,\sigma_{5}(\frac{n}{6})  \\ &
+ \frac{144}{13}\,\EuFrak{b}_{12,1}(n)
+ \frac{1008}{13}\,\EuFrak{b}_{12,2}(n),
	\end{align*}  
\begin{align*}
R_{(1,2)}^{8,4}(n)\, =\,R_{(2,1)}^{4,8}(n)\, =\, & 
 \frac{60}{13}\sigma_{5}(n) 
+ \frac{192}{13}\,\sigma_{5}(\frac{n}{2}) 
- \frac{1620}{13}\,\sigma_{5}(\frac{n}{3}) 
- \frac{5184}{13}\,\sigma_{5}(\frac{n}{6})  \\ &
+ \frac{252}{13}\,\EuFrak{b}_{12,1}(n)
+ \frac{2304}{13}\,\EuFrak{b}_{12,2}(n),
\end{align*}
\begin{align*}
R_{(1,3)}^{4,8}(n)\, =\,R_{(3,1)}^{8,4}(n)\, =\, & 
\frac{4}{13}\,\sigma_{5}(n) 
- \frac{724}{13}\,\sigma_{5}(\frac{n}{3}) 
- \frac{5832}{13}\,\sigma_{5}(\frac{n}{9})
+ \frac{152}{13}\,\EuFrak{b}_{18,1}(n)    \\ & 
- \frac{324}{13}\,\EuFrak{b}_{18,1}(\frac{n}{3})
+ 96\,\EuFrak{b}_{18,2}(n)  
+ \frac{2916}{13}\,\EuFrak{b}_{18,3}(n), 
\end{align*}
\begin{align*}
R_{(1,3)}^{8,4}(n)\, =\,R_{(3,1)}^{4,8}(n)\, =\, & 
\frac{24}{13}\,\sigma_{5}(n) 
+ \frac{2172}{13}\,\sigma_{5}(\frac{n}{3}) 
- \frac{8748}{13}\,\sigma_{5}(\frac{n}{9})
+ \frac{288}{13}\,\EuFrak{b}_{18,1}(n)     \\ & 
+ \frac{972}{13}\,\EuFrak{b}_{18,1}(\frac{n}{3})
+ 288\,\EuFrak{b}_{18,2}(n) 
+ \frac{11340}{13}\,\EuFrak{b}_{18,3}(n) 
\end{align*}
If $n\not\equiv 0\pmod{3}$ then 
\begin{align*}
R_{(1,7)}^{4,8}(n)\, =\,R_{(7,1)}^{8,4}(n)\, =\, & 
\frac{84}{10621}\,\sigma_{5}(n) 
- \frac{26244}{817}\,\sigma_{5}(\frac{n}{3}) 
+ \frac{205800}{10621}\,\sigma_{5}(\frac{n}{7})     \\ & 
- \frac{2458836}{10621}\,\sigma_{5}(\frac{n}{21}) 
+ \frac{127368}{10621}\,\EuFrak{b}_{21,1}(n)
- \frac{26244}{817}\,\EuFrak{b}_{21,1}(\frac{n}{3})     \\ & 
+ \frac{38448}{247}\,\EuFrak{b}_{21,2}(n)    
- \frac{7776}{19}\,\EuFrak{b}_{21,2}(\frac{n}{3})   
+ \frac{9009864}{10621}\,\EuFrak{b}_{21,3}(n)          \\ & 
- \frac{1138212}{817}\,\EuFrak{b}_{21,3}(\frac{n}{3})
+ \frac{9504}{13}\,\EuFrak{b}_{21,4}(n) 
+ \frac{7776}{13}\,\EuFrak{b}_{21,5}(n)      \\ & 
+ \frac{143100}{13}\,\EuFrak{b}_{21,6}(n) 
+ \frac{34344}{13}\,\EuFrak{b}_{21,7}(n) 
- \frac{864}{13}\,\EuFrak{b}_{21,8}(n)      \\ & 
+ \frac{1722600}{13}\,\EuFrak{b}_{21,9}(n)    
- \frac{28512}{13}\,\EuFrak{b}_{21,10}(n) 
- \frac{57888}{13}\,\EuFrak{b}_{21,11}(n)      \\ & 
+ \frac{6195528}{13}\,\EuFrak{b}_{21,12}(n), 
\end{align*}
\end{corollary}
\begin{proof} We only consider the case $(c,d)=(1,3)$ since the other cases
 can be proved in a similar way.

The proof follows immediately from \autoref{representations-theor-c_d} 
	with $(c,d)=(1,3)$; that is, 
 	\begin{align*}
 R_{(1,3)}^{4,8}(n) =  & 
 12\,\sigma(n) - 36\,\sigma(\frac{n}{3}) + 24\,\sigma_{3}(\frac{n}{3}) +
 216\,\sigma_{3}(\frac{n}{9}) + 288\, W_{(1,3)}^{1,3}(n) \\ & 
 + 2592\, W_{(1,9)}^{1,3}(n) 
 - 864\,W_{(1,1)}^{1,3}(\frac{n}{3}) -  7776\,W_{(1,3)}^{1,3}(\frac{n}{3}) 
 \end{align*}
 \begin{align*}
 R_{(1,3)}^{8,4}(n) =  & 
 24\,\sigma_{3}(n) + 216\,\sigma_{3}(\frac{n}{3}) + 12\,\sigma(\frac{n}{3}) -
 36\,\sigma(\frac{n}{9}) + 288\, W_{(1,3)}^{3,1}(n) \\ & 
 - 864\, W_{(1,9)}^{3,1}(n) 
 + 2592\,W_{(1,1)}^{3,1}(\frac{n}{3}) - 7776\,W_{(1,3)}^{3,1}(\frac{n}{3}) 
 \end{align*}
	One can then make use of 
	\begin{itemize}
		\item
	(\hyperref[Ramanujan-ident-1-3-gen]{\ref*{Ramanujan-ident-1-3-gen}}), 
	\exwX\ and \oxmY\ \cite{yao-xia-2014} or 
	(\hyperref[convolutionSum-w_3_1]{\ref*{convolutionSum-w_3_1}}) 
	 and 
	(\hyperref[convolutionSum-w_1_9]{\ref*{convolutionSum-w_1_9}}) 
	to simplify the formula $R_{(1,3)}^{4,8}(n)$ and then obtain the stated result.  
		\item  
	(\hyperref[Ramanujan-ident-1-3-gen]{\ref*{Ramanujan-ident-1-3-gen}}), 
	(\hyperref[convolutionSum-w_1_3]{\ref*{convolutionSum-w_1_3}}) 
	and 
	(\hyperref[convolutionSum-w_9_1]{\ref*{convolutionSum-w_9_1}}) 
	to simplify $R_{(1,3)}^{8,4}(n)$ and obtain the stated result.
	\end{itemize}
\end{proof}








\begin{thebibliography}{10}

\bibitem{2016arXiv160309412A}
A.~{Alaca}, S.~{Alaca}, and Z.~{Selcuk Aygin}.
\newblock {A family of Eta Quotients and an Extension of the Ramanujan-Mordell
  Theorem}.
\newblock {\em ArXiv e-prints}, Mar. 2016.

\bibitem{alaca-uygul-williams-2012}
{\c{S}}.~{Alaca}, F.~{Uygul}, and K.~S. {Williams}.
\newblock {Some arithmetic Identities involving Divisor Functions}.
\newblock {\em Functiones et Approximatio}, 46.2:261–--271, 2012.

\bibitem{cheng-williams-2005}
N.~{Cheng} and K.~S. {Williams}.
\newblock {Evaluation of some convolution sums involving the Sum of Divisors
  Functions}.
\newblock {\em Yokohama Math. J.}, 52:39–--57, 2005.

\bibitem{huardetal}
J.~G. Huard, Z.~M. Ou, B.~K. Spearman, and K.~S. Williams.
\newblock Elementary evaluation of certain convolution sums involving divisor
  functions.
\newblock {\em Number Theory Millenium}, 7:229--274, 2002.
\newblock A K Peters, Natick, MA.

\bibitem{kilford}
L.~J.~P. Kilford.
\newblock {\em Modular forms: A classical and computational introduction}.
\newblock Imperial College Press, London, 2008.

\bibitem{koblitz-1993}
N.~Koblitz.
\newblock {\em Introduction to Elliptic Curves and Modular Forms}, volume~97 of
  {\em Graduate Texts in Mathematics}.
\newblock Springer Verlag, New York, 2 edition, 1993.

\bibitem{koehler}
G.~K\"{o}hler.
\newblock {\em Eta Products and Theta Series Identities}, volume 3733 of {\em
  Springer Monographs in Mathematics}.
\newblock Springer Verlag, Berlin Heidelberg, 2011.

\bibitem{kokluce-2017}
B.~{K\"{o}kl\"{u}ce}.
\newblock {Evaluation of some convolution sums and representation of integers
  by certain quadratic forms in 12 variables}.
\newblock 2017.
\newblock submitted for publication.

\bibitem{kokluce-eser-2017}
B.~{K\"{o}kl\"{u}ce} and H.~{Eser}.
\newblock {Evaluation of some convolution sums and representation of integers
  by certain quadratic forms in 12 variables}.
\newblock {\em Int. J. Number Theory}, 13(3):775–--799, 2017.

\bibitem{lahiri1946}
D.~B. {Lahiri}.
\newblock {On Ramanujan’s function $\tau(n)$ and the divisor function
  $\sigma_{k}(n)$ - I}.
\newblock {\em Bull. Calcutta Math. Soc.}, 38:193--206, 1946.

\bibitem{lahiri1947}
D.~B. {Lahiri}.
\newblock {On Ramanujan’s function $\tau(n)$ and the divisor function
  $\sigma_{k}(n)$ - II}.
\newblock {\em Bull. Calcutta Math. Soc.}, 39:33--52, 1947.

\bibitem{ligozat_1975}
G.~Ligozat.
\newblock Courbes modulaires de genre $1$.
\newblock {\em Bull. Soc. Math. France}, 43:5--80, 1975.

\bibitem{lomadze}
G.~A. Lomadze.
\newblock Representation of numbers by sums of the quadratic forms
  $x_{1}^{2}+x_{1}x_{2}+x_{2}^{2}$.
\newblock {\em Acta Arith.}, 54(1):9--36, 1989.

\bibitem{miyake1989}
T.~Miyake.
\newblock {\em Modular Forms}.
\newblock Springer monographs in Mathematics. Springer Verlag, New York, 1989.

\bibitem{newman_1957}
M.~Newman.
\newblock Construction and application of a class of modular functions.
\newblock {\em Proc. Lond. Math. Soc.}, 7(3):334--350, 1957.

\bibitem{newman_1959}
M.~Newman.
\newblock Construction and application of a class of modular functions {II}.
\newblock {\em Proc. Lond. Math. Soc.}, 9(3):373--387, 1959.

\bibitem{ntienjem2016b}
E.~{Ntienjem}.
\newblock {Evaluation of the Convolution Sum involving the Sum of Divisors
  Function for 22, 44 and 52}.
\newblock {\em Open Mathematics}, 15(1):446--458, 2017.

\bibitem{pizer1976}
A.~Pizer.
\newblock The representability of modular forms by theta series.
\newblock {\em J. Math. Soc. Japan}, 28(4):689--698, 10 1976.

\bibitem{ramanujan}
S.~Ramanujan.
\newblock On certain arithmetical functions.
\newblock {\em Trans. Camb. Phil. Soc.}, 22:159--184, 1916.

\bibitem{wstein}
W.~A. Stein.
\newblock {\em Modular Forms, A Computational Approach}, volume~79.
\newblock American Mathematical Society, Graduate Studies in Mathematics, 2011.
\newblock http://wstein.org/books/modform/modform/.

\bibitem{williams2001}
K.~S. {Williams}.
\newblock {An arithmetic Proof of Jacobi's Eight Squares Theorem}.
\newblock {\em Far East J. Math. Sci.}, 3(6):1001--1005, 2001.

\bibitem{williams2011}
K.~S. Williams.
\newblock {\em Number Theory in the Spirit of Liouville}, volume~76 of {\em
  London Mathematical Society Student Texts}.
\newblock Cambridge University Press, Cambridge, 2011.

\bibitem{yao-xia-2014}
E.~X.~W. {Xia} and O.~X.~M. {Yao}.
\newblock {Evaluation of the convolution sums
  $\overset{}{\underset{l+3m=n}{\sum}\sigma(l)\sigma_{3}(m)}$ and
  $\overset{}{\underset{3l+m=n} {\sum}\sigma(l)\sigma_{3}(m)}$}.
\newblock {\em Int. J. Number Theory}, 10(1):115–--123, 2014.

\end{thebibliography}



\section*{Tables}

	\begin{longtable}{|c|cccccc|} \hline
		& \textbf{1}  &  \textbf{2}  & \textbf{3} & \textbf{4}  &
		\textbf{6} &  \textbf{12}  \\ \hline
		\textbf{1}  & 6 & 0 & 6 & 0 & 0 & 0  \\ \hline
		\textbf{2}  & 0 & 6 & 0 & 0 & 6 & 0   \\ \hline
		\textbf{3}  & 4 & -2 & -4 & 0 & 14 & 0    \\ \hline
		\textbf{4}  & 0 & 0 & 0 & 6 & 0 & 6   \\ \hline
		\textbf{5}  & 0 & 2 & 0 & 2 & -2 & 10   \\ \hline
		\textbf{6} & 0 & 4 & 0 & -2 & -4 & 14   \\ \hline
		\textbf{7} & -2 & 5 & -2 & -5 & 1 & 15   \\ \hline
		\caption{Power of $\eta$-quotients being basis elements of $\S_{6}(\Gamma_{0}(12))$}
		\label{convolutionSums-12-table}
	\end{longtable}

\begin{longtable}{|c|cccc|} \hline
	& \textbf{1}  &  \textbf{2}  & \textbf{7} & \textbf{14}  \\ \hline
	\textbf{1}  & 10 & 0 & 2 & 0   \\ \hline
	\textbf{2}  & 6 & 0 & 6 & 0   \\ \hline
	\textbf{3}  & 2 & 0 & 10 & 0   \\ \hline
	\textbf{4}  & 0 & 6 & 0 & 6   \\ \hline
	\textbf{5}  & 3 & -1 & 3 & 7 \\ \hline
	\textbf{6}  & 0 & 2 & 0 & 10   \\ \hline
	\textbf{7}  & 4 & -2 & -4 & 14 \\ \hline
	\textbf{8}  & 1 & 1 & -7 & 17  \\ \hline
	\caption{Power of $\eta$-quotients being basis elements of $\S_{6}(\Gamma_{0}(14))$}
	\label{convolutionSums-14-table}
\end{longtable}

\begin{longtable}{|c|cccc|} \hline
	& \textbf{1}  &  \textbf{3}  & \textbf{5} & \textbf{15}  \\ \hline
	\textbf{1}  & 6 & 6 & 0 & 0   \\ \hline
	\textbf{2}  & 0 & 6 & 6 & 0  \\ \hline
	\textbf{3}  & 3 & 3 & 3 & 3  \\ \hline
	\textbf{4}  & 6 & 0 & 0 & 6   \\ \hline
	\textbf{5}  & 0 & 0 & 6 & 6   \\ \hline
	\textbf{6}  & 7 & -1 & -5 & 11   \\ \hline
	\textbf{7}  & 1 & -1 & 1 & 11 \\ \hline
	\textbf{8}  & 1 & -1 & 13 & -1   \\ \hline
	\caption{Power of $\eta$-quotients being basis elements of $\S_{6}(\Gamma_{0}(15))$}
	\label{convolutionSums-15-table}
\end{longtable}

%

	\begin{longtable}{|c|cccccc|} \hline
		& \textbf{1}  &  \textbf{2}  & \textbf{3} & \textbf{6}  &
	\textbf{9} & \textbf{18} \\ \hline
		\textbf{1}  & 6 & 0 & 6 & 0 & 0 & 0  \\ \hline
		\textbf{2}  & 3 & 0 & 6 & 3 & 0 & 0   \\ \hline
		\textbf{3}  & 0 & 0 & 6 & 0 & 6 & 0  \\ \hline
		\textbf{4}  & 0 & 3 & 0 & 6 & 0 & 3  \\ \hline
		\textbf{5}  & 0 & 0 & 0 & 2 & 8 & 2  \\ \hline
		\textbf{6} & 0 & 0 & 0 & 6 & 0 & 6  \\ \hline
		\textbf{7} & 0 & 0 & 2 & 0 & 2 & 8  \\ \hline
		\textbf{8}  & 0 & 0 & 2 & 4 & -6 & 12 \\ \hline
		\textbf{9}  & 0 & 0 & 4 & -2 & -4 & 14 \\ \hline
		\textbf{10}  & 0 & 0 & 5 & -3 & -7 & 17 \\ \hline
		\textbf{11}  & 0 & 6 & 8 & -2 & 0 & 0 \\ \hline
		\caption{Power of $\eta$-quotients being basis elements of $\S_{6}(\Gamma_{0}(18))$}
		\label{convolutionSums-18-table}
	\end{longtable}

	\begin{longtable}{|c|cccccc|} \hline
		& \textbf{1}  &  \textbf{2}  & \textbf{4} & \textbf{5}  &
	\textbf{10} & \textbf{20} \\ \hline
		\textbf{1}  & 0 & 12 & 0 & 0 & 0 & 0  \\ \hline
		\textbf{2}  & 0 & 0 & 2 & 8 & 4 & -2   \\ \hline
		\textbf{3}  & 0 & 0 & 8 & 0 & 4 & 0  \\ \hline
		\textbf{4}  & 0 & 0 & 4 & 8 & -4 & 4  \\ \hline
		\textbf{5}  & 0 & 0 & 0 & 0 & 12 & 0  \\ \hline
		\textbf{6} & 0 & 2 & 0 & -8 & 18 & 0  \\ \hline
		\textbf{7} & 0 & 0 & 2 & 0 & 4 & 6  \\ \hline
		\textbf{8}  & 0 & 4 & -4 & -8 & 16 & 4 \\ \hline
		\textbf{9}  & 0 & 0 & 4 & 0 & -4 & 12  \\ \hline
		\textbf{10}  & 1 & -3 & 0 & -5 & 11 & 8 \\ \hline
		\textbf{11}  & 0 & 2 & 0 & 0 & -6 & 16 \\ \hline
		\textbf{12} & 1 & -2 & -1 & -5 & 6 & 13 \\ \hline
		\caption{Power of $\eta$-quotients being basis elements of $\S_{6}(\Gamma_{0}(20))$}
		\label{convolutionSums-20-table}
	\end{longtable}

\begin{longtable}{|c|cccc|} \hline
	& \textbf{1}  &  \textbf{3}  & \textbf{7} & \textbf{21}  \\ \hline
	\textbf{1}  & 10 & 0 & 2 & 0   \\ \hline
	\textbf{2}  & 6 & 0 & 6 & 0  \\ \hline
	\textbf{3}  & 2 & 0 & 10 & 0  \\ \hline
	\textbf{4}  & 0 & 4 & 6 & 2   \\ \hline
	\textbf{5}  & 2 & 2 & 4 & 4   \\ \hline
	\textbf{6}  & 0 & 6 & 0 & 6   \\ \hline
	\textbf{7}  & 0 & 0 & 6 & 6 \\ \hline
	\textbf{8}  & 1 & 3 & -1 & 9   \\ \hline
	\textbf{9}  & 0 & 2 & 0 & 10  \\ \hline
	\textbf{10}  & 2 & 0 & -2 & 12   \\ \hline
	\textbf{11}  & 1 & -1 & -1 & 13 \\ \hline
	\textbf{12}  & 0 & -2 & 0 & 14   \\ \hline
	\caption{Power of $\eta$-quotients being basis elements of $\S_{6}(\Gamma_{0}(21))$}
	\label{convolutionSums-21-table}
\end{longtable}

\begin{longtable}{|c|cccc|} \hline
	& \textbf{1}  & \textbf{3} & \textbf{9} & \textbf{27}  \\ \hline
\textbf{1} & 6 & 6 & 0 & 0  \\ \hline
\textbf{2} & 3 & 6 & 3 & 0  \\ \hline
\textbf{3} & 0 & 6 & 6 & 0  \\ \hline
\textbf{4} & 0 & 2 & 10 & 0  \\ \hline
\textbf{5} & 0 & -2 & 14 & 0  \\ \hline
\textbf{6} & 0 & 3 & 6 & 3  \\ \hline
\textbf{7}  & 0 & 8 & -2 & 6 \\ \hline
\textbf{8}  & 0 & 4 & 2 & 6   \\ \hline
\textbf{9}  & 0 & 0 & 6 & 6  \\ \hline
\textbf{10}  & 0 & 5 & -2 & 9   \\ \hline
\textbf{11}  & 0 & 1 & 2 & 9 \\ \hline
\textbf{12}  & 0 & 2 & 2 & 0   \\ \hline  
\caption{Power of $\eta$-functions being basis elements of $\S_{6}(\Gamma_{0}(27))$}
\label{convolutionSums-27-table}
\end{longtable}

\begin{longtable}{|c|cccccc|} \hline
	& \textbf{1}  &  \textbf{2}  & \textbf{4} & \textbf{8} & \textbf{16} & \textbf{32}  \\ \hline
\textbf{1} & 0 & 12 & 0 & 0 & 0 & 0  \\ \hline
\textbf{2} & 0 & 0 & 12 & 0 & 0 & 0  \\ \hline
\textbf{3} & 0 & 4 & 0 & 8 & 0 & 0  \\ \hline
\textbf{4} & 0 & 0 & 0 & 12 & 0 & 0  \\ \hline
\textbf{5} & 0 & 0 & 2 & 6 & 4 & 0  \\ \hline
\textbf{6} & 0 & 0 & 4 & 0 & 8 & 0  \\ \hline
\textbf{7} & 0 & 0 & 6 & -6 & 12 & 0  \\ \hline
\textbf{8} & 0 & 0 & 0 & 0 & 12 & 0  \\ \hline
\textbf{9} & 0 & 0 & 6 & -4 & 6 & 4  \\ \hline
\textbf{10} & 0 & 0 & 0 & 2 & 6 & 4  \\ \hline
\textbf{11} & 0 & 0 & 6 & -2 & 0 & 8  \\ \hline
\textbf{12} & 0 & 0 & 0 & 4 & 0 & 8  \\ \hline
\textbf{13} & 0 & 0 & 6 & 0 & -6 & 12  \\ \hline
\textbf{14} & 0 & 0 & 0 & 6 & -6 & 12  \\ \hline
\textbf{15} & 0 & 0 & 2 & 0 & -2 & 12  \\ \hline
\textbf{16} & 0 & 0 & 2 & 2 & 0 & 0  \\ \hline   
\caption{Power of $\eta$-functions being basis elements of $\S_{6}(\Gamma_{0}(32))$}
\label{convolutionSums-32-table}
\end{longtable}




 



\end{document}